\documentclass[11pt,letterpaper]{amsart}

\usepackage{latexsym,rawfonts}
\usepackage{amsfonts,amssymb}
\usepackage{amsmath,amsthm}

\usepackage[plainpages=false]{hyperref}
\usepackage{graphicx}
\usepackage{color}
\usepackage[table]{xcolor}
\usepackage{longtable}
\usepackage{cite}

\DeclareMathOperator{\DIV}{div}

\DeclareMathOperator{\TR}{tr}

\DeclareMathOperator{\DIAM}{diam}

\newcommand{\R}{\mathbb{R}}

\newcommand{\dd}{\mathop{}\!\mathrm{d}}
\newcommand{\abs}[1]{\left\vert#1\right\vert}
\newcommand{\set}[1]{\left\{#1\right\}}
\newcommand{\norm}[1]{\left\Vert#1\right\Vert}

\newcommand{\pd}{\partial}
\newcommand{\delbar}{\overline{\nabla}}

\newcommand{\uS}{\mathbb{S}^{n-1}}
\newcommand{\II}{I\!\!I}

\newcommand{\ssOmega}{{\scriptscriptstyle\Omega}}

\newcommand{\rnnn}{\mathbb R^{n}}

\newcommand{\rt}{\mathbb R^{n-1}}
\newcommand{\sn}{ {\mathbb{S}^{n-1}}}
\newcommand{\rn}{\mathbb R}
\newcommand{\N}{\mathbb N}

\newcommand{\psum}{{+_{\negthinspace\kern-2pt p}}\,}
\newcommand{\qsum}[1]{{+_{\negthinspace\kern-2pt #1}}\,}
\newcommand{\dpsum}{{\tilde+_{\negthinspace\kern-1pt p}}\,}
\newcommand{\dqsum}[1]{{\tilde+_{\negthinspace\kern-1pt #1}}\,}
\newcommand{\lsub}[1]{\hskip -1.5pt\lower.5ex\hbox{$_{#1}$}}

\numberwithin{equation}{section}

\newtheorem{theo}{Theorem}[section]
\newtheorem{coro}[theo]{Corollary}

\newtheorem{theorem}[theo]{Theorem}
\newtheorem{lemma}[theo]{Lemma}

\newtheorem{remark}[theo]{Remark}

\theoremstyle{definition}

\begin{document}

\title{Smooth solutions to the chord log-Minkowski problem}

\author[J. Hu]{Jinrong Hu}
\address{Institut f\"{u}r Diskrete Mathematik und Geometrie, Technische Universit\"{a}t Wien, Wiedner Hauptstrasse 8-10, 1040 Wien, Austria
 }
\email{jinrong.hu@tuwien.ac.at}

\author[Y. Huang]{Yong Huang}
\address{School of Mathematics, Hunan University, Changsha, 410082, Hunan Province, China}
\email{huangyong@hnu.edu.cn}

\author[J. Lu]{Jian Lu}
\address{School of Mathematics and Statistics,
  Nanjing University of Science and Technology,
  Nanjing 210094, P.R. China}
\email{lj-tshu04@163.com}

\begin{abstract}
In integral geometry generalized with Aleksandrov's variational theory, Lutwak-Xi-Yang-Zhang \cite{LXYZ} recently opened the door to researching the cone-chord measures and their log-Minkowski problem stemming from the chord integrals, named as the chord log-Minkowski problem. In the smooth category, the solvability of the chord log-Minkowski problem amounts to dealing with a nonlocal {M}onge-{A}mp\`ere equation involving a Riesz potential. In this paper, to study the chord log-Minkowski problem, we first  present some novel results on the boundary regularity of the Riesz potential. Based on these results, we obtain the regularity and existence for the chord log-Minkowski problem from the perspective of a nonlocal {M}onge-{A}mp\`ere equation and a nonlocal Gauss curvature flow equation.

\end{abstract}


% \makeatletter
% \@namedef{subjclassname@2020}{\textup{2020} Mathematics Subject Classification}
% \makeatother

%\subjclass[2010]{35J96, 35K96, 31B15, 52A38, 58J35}

\thanks{This work was supported by the National Natural Science Foundation of
  China (12171144, 12231006 to Huang, 12122106 to Lu) and the Austrian Science Fund (FWF):
10.55776/ESP1358925.
The author Lu was also supported by the Stability Support Program for Basic Research (WZJC202601002).
}
\maketitle

\baselineskip18pt

\parskip3pt

\section{Introduction}

The field of convex geometric analysis in $\rnnn$ is largely driven by the investigation of geometric invariants and geometric measures associated with convex bodies. These two notions are closely linked, differentiating geometric invariants typically yields geometric measures. Within the Brunn-Minkowski theory, the quermassintegrals stand out as the essential geometric invariants, while the area measures (due to Aleksandrov, Fenchel, and Jessen) and the curvature measures (due to Federer) constitute the main classes of geometric measures. A cornerstone problem in this context is the Minkowski-type problem, which asks for a convex body with a prescribed geometric measure. The classic instance is the Minkowski problem for the surface area measure, originally proposed and solved by Minkowski himself \cite{M897,M903} in both the polytopal and absolutely continuous settings. In the 1930s, Aleksandrov \cite{A42,A39} extended this result by solving the generalized Minkowski problem. Following Aleksandrov's ground-breaking work, a vast body of subsequent research has emerged, as documented in a series of works \cite{CY76,CW06,FJ38,F68,G09,G03,HgLYZ05,LW13,L93,LYZ04,N53,O07,P52,Zhu14,Zhu15} and the references therein. The evolution of the Minkowski problem has, in turn, fueled major developments across several disciplines, including partial differential equations, differential geometry, and optimal transport. This far-reaching impact is thoroughly reviewed in \cite{FM17}.

 The dual Brunn-Minkowski theory, as a parallelism of the Brunn-Minkowski theory,  was launched by Lutwak in the 1970s \cite{L75}, but remained relatively latent until Huang-LYZ \cite{HLYZ16} discovered a family of fundamental geometric measures, named as the dual curvature measure. These measures serve as duals to Federer's curvature measures, playing a role in the dual theory analogous to that of area measures and curvature measures in the Brunn-Minkowski framework. The dual curvature measures arise from the differentials of dual quermassintegrals. The associated Minkowski problem, termed the dual Minkowski problem, was first proposed and solved by Huang-LYZ \cite{HLYZ16}.  Since then, this problem has evolved into a highly active research area, inspiring a wealth of significant contributions and diverse applications. A representative collection of such developments includes \cite{BF19,BLYZ19,BLYZ20,CL18,HJ19,HLYZ18,HZ18,JW17,LSW20,Z17,Z18}.

We now come to a new branch in geometry from the viewpoint of probability, known as integral geometry (geometric probability). It was established by Blaschke around the mid-1930s,  see Ren \cite{R94} and Santal\'{o} \cite{S04}. Analogous to convex geometry, geometric invariants and geometric measures lie at the core of integral geometry. The chord integral, a significant quantity stemming from integral geometry, forms another major family of geometric invariants for convex bodies. It constitutes, alongside the quermassintegrals and dual quermassintegrals, one of the fundamental concepts in geometry. The chord integral $I_{q}(\Omega)$ of a convex body $\Omega\subset \rnnn$ is defined by
\begin{equation}\label{chord}
I_{q}(\Omega)=\int_{\pounds^{n}}|\Omega \cap \ell|^{q}\dd\ell, \quad {\rm real}\ q\geq 0,
\end{equation}
where $|\Omega \cap \ell|$ denotes the length of the chord $\Omega \cap \ell$,
and the integration is with respect to Haar measure on $\pounds^{n}$.

From an analytical perspective, as demonstrated in \cite{R94,S04}, for $q>1$,  there is a link between the chord integrals and the Riesz energies of the characteristic function of the convex body $\Omega$,
\begin{equation*}\label{q1} I_{q}(\Omega)=\frac{q(q-1)}{n\omega_{n}}\int_{\rnnn}\int_{\rnnn}\frac{\chi_{\Omega}(y)\chi_{\Omega}(z)}{|y-z|^{n+1-q}}\dd y \dd z, \quad  q>1.
\end{equation*}
For $0<q<1$, by using the Blaschke-Petkantschin formula in Schneider-Weil's book \cite{SW08} (see also Ludwig \cite{Lu14}),  Qin \cite{Q23} recently observed that
\begin{equation*}\label{Pse}
I_q(\Omega)=\frac{q(1-q)}{n \omega_n} \int_{\mathbb{R}^n} \int_{\mathbb{R}^n} \frac{\chi_{\Omega} (y) \chi_{\Omega^c}(z)}{|y-z|^{n-q+1}} \dd y \dd z,
    \end{equation*}
where $\Omega^c$ is the complement of $\Omega$. One sees that $I_{q}$ in fact coincides to the variant of the $(1-q)$-th perimeter for $0<q<1$.

We turn to talk about another important quantity, geometric measures, coming from integral geometry. A major triumph in integral geometry is that Lutwak-Xi-Yang-Zhang \cite{LXYZ} introduced the  ($q$-th) chord measure, which is produced by the differential of chord integrals under the \emph{Minkowski combination}. Precisely, for convex bodies $\Omega$ and $L$ in $\rnnn$, they obtain
\[
\frac{d}{dt}\Big|_{t=0}I_q(\Omega+tL)=\int_{\sn}h_{L}(x)\dd F_{q}(\Omega,x), \quad q\geq 0,
\]
where the chord measure $F_q(\Omega,\cdot)$ is defined by
\begin{equation}\label{FD}
  F_{q}(\Omega,\eta)=\frac{2q}{\omega_{n}}\int_{\nu^{-1}_{\Omega}(\eta)}\widetilde{V}_{q-1}(\Omega,z)\dd \mathcal{H}^{n-1}(z), \quad {\rm Borel}\ \eta\subset \sn,
\end{equation}
  here $\widetilde{V}_{q-1}(\Omega,z)$ is called as the $(q-1)$-th (generalized) \emph{dual
  quermassintegral} of $\Omega$ with regard to $z\in \partial\Omega$ (see Section \ref{sec2} for more details).
  Chord measures serve as the fourth major family of geometric measures of convex bodies, alongside the Aleksandrov–Fenchel–Jessen's area measures, Federer's curvature measures, and the dual curvature measures introduced in \cite{HLYZ16}.
 Furthermore, Lutwak-Xi-Yang-Zhang \cite{LXYZ} also defined the cone-chord measure $G_{q}(\Omega,\cdot)$ of $\Omega$ for $q>0$,
\begin{equation}\label{GD}
 \dd G_{q}(\Omega,\cdot)=\frac{1}{n+q-1}h_{\Omega}(\cdot)\dd F_{q}(\Omega,\cdot).
\end{equation}

 In \cite{LXYZ}, the authors proposed the Minkowski problem and the log-Minkowski problem associated with the chord measure and the cone-chord measure, termed the chord Minkowski problem and the chord log-Minkowski problem, respectively. Furthermore, they introduced the $L_{p}$ chord Minkowski problem. Subsequently, its existence was investigated from a variational perspective in \cite{GXZ23, XYZ23}, while nonuniqueness for this problem was studied in \cite{Li24}. In this paper, we focus on the chord log-Minkowski problem.

\emph{The Chord Log-Minkowski Problem:} Suppose $q\geq 0$, given a finite Borel measure $\mu$ on the unit sphere $\sn$, what are the necessary and sufficient conditions on $\mu$ such that there exists a convex body  $\Omega$ in $\rnnn$ satisfying
\begin{equation}\label{CLMP0}
\mu=G_{q}(\Omega,\cdot)?
\end{equation}
The authors \cite{LXYZ} proved the existence of the solution of \eqref{CLMP0} by variational arguments when the index $q$ is in $[1,n+1]$ with the even assumption on $\mu$. Recently, Guo-Xi-Zhao \cite{GXZ23} solved the chord log-Minkowski problem in the polytope case when the given normal vectors are in \emph{general position} (not necessarily even).
Note that, in terms of \eqref{CLMP0}, this problem reduces to the classical log-Minkowski problem prescribing the cone volume measure when $q=1$ or $q=n+1$. This special case was first resolved by B\"{o}r\"{o}czky-LYZ  \cite{BLYZ12} for even measures $\mu$, and when $q=0$, it becomes the log-Christoffel-Minkowski problem as shown in \cite{LXYZ}.

In the 1930s, Aleksandrov \cite{A96,A42}  incorporated nonlinear partial differential equations into the Brunn-Minkowski theory. For instance, he demonstrated the  mixed Aleksandrov-Fenchel volume inequality using continuous methods and eigenvalue theory following the approach of Hilbert, and developed a variational theory of weak solutions for the classical Minkowski problem. This ground-breaking work makes a significant contribution to convex geometry and the regularity of Minkowski problem, which subsequently promotes the development of potential-theoretic problem for convex bodies as presented by Jerison \cite{J96}. The intrinsic partial differential equations  of the dual Brunn-Minkowski theory introduced by Huang-LYZ \cite{HLYZ16}  have waited a full 40 years to emerge after the birth of the dual theory in the 1970s. In both the Brunn–Minkowski theory and its dual theory, the associated partial differential equations are local. We now turn to nonlocal partial differential equations that arise in integral geometry.

 Purely as an aside, when the given measure $\mu$ has a positive density, say $\frac{2q}{\omega_{n}(n+q-1)}f$, for $q>0$, the solvability of the geometric problem \eqref{CLMP0} amounts to tackling the following nonlocal
 {M}onge-{A}mp\`ere type equation involving the (generalized) dual quermassintegrals,
 \begin{equation}\label{CLMP}
h\det(\nabla^{2}h+hI)\widetilde{V}_{q-1}(\Omega_{h},\overline{\nabla} h)=f,\quad {\rm on} \ \sn,
\end{equation}
where $h:\sn \rightarrow (0,+\infty)$ is an unknown function, $\Omega_{h}$ is a convex body determined by $h$,  $\overline{\nabla}h$ is the Euclidean gradient of $h$ in $\rnnn$, $\nabla^{2}h$ is the spherical Hessian of $h$ with respect to a local orthonormal frame on $\sn$, and $I$ is the identity matrix.

Within the framework of analysis, there is a connection between the Riesz potential and $\widetilde{V}_q$ (note that $\widetilde{V}_q$ is a constant on $\pd\Omega$ when $q=0, n$), which is explicitly shown in \cite{LXYZ} for $q>0$ and in \cite{Q23} for  $-1<q<0$,  i.e., for every $ z\in \partial \Omega$,
\begin{equation}\label{Vqq}
\widetilde{V}_{q}(\Omega,z):=\widetilde{V}_{q}(z)= \begin{cases}\frac{q}{n}\int_{\Omega}\frac{1}{|y-z|^{n-q}} \mathrm{d} y, & q>0, \\ \frac{-q}{2n}\left(\int_{\Omega^{c}}\frac{1}{|y-z|^{n-q}}\mathrm{d} y-\int_{\Omega}\frac{1}{|y-z|^{n-q}}\mathrm{d}y\right), & -1<q<0,\end{cases}
\end{equation}
where $\Omega^{c}$ is the complement of $\Omega$.

Using \eqref{Vqq}, on the one hand, for $q>1$,  as shown in \cite{LXYZ}, \eqref{CLMP} becomes
\begin{equation*}\label{1.9a}
  h\det(\nabla^{2}h+hI)
  \int_{\Omega_{h}}\frac{\dd y}{|y-\delbar h|^{n+1-q}}
  =\frac{n}{q-1}f.
\end{equation*}
On the other hand, for $0<q<1$, as given in \cite{Q23}, \eqref{CLMP} turns into
\begin{equation*}\label{1.9b}
  h\det(\nabla^2h+hI) \left(
    \int_{\Omega^c_h}\frac{\dd y}{|y-\overline{\nabla}h|^{n+1-q}}
    -\int_{\Omega_h}\frac{\dd y}{|y-\overline{\nabla}h|^{n+1-q}}
  \right)
  =\frac{2n}{1-q}f.
\end{equation*}

First, we have the following results regarding the regularity of chord log-Minkowski problem.

\begin{theo}\label{thm3}

Suppose $q\in [1,n+1]$, $k\in \N$ and $\alpha \in (0,1)$. Let $\Omega$ be an origin-symmetric convex body in $\rnnn$. If the cone-chord measure $G_{q}(\Omega,\cdot)$ of $\Omega$ is absolutely continuous with respect to the spherical Lebesgue measure, with a strictly positive density of class $C^{k,\alpha}(\sn)$, then the boundary of $\Omega$ is $C^{k+2,\alpha}$ smooth.
\end{theo}

As a direct consequence of Theorem \ref{thm3}, we have
\begin{coro}
Suppose $q\in [1,n+1]$ and $\alpha \in (0,1)$. For any given even positive function $f\in C^{\alpha}(\sn)$, there exists an even  solution $h\in C^{2,\alpha}(\sn)$ to Eq. \eqref{CLMP}. If $f$ is smooth, then the solution is also smooth.
\end{coro}

Regularity is central to Minkowski-type problems.
Cheng-Yau~\cite{CY76} established regularity results for the classical
Minkowski problem, while Caffarelli~\cite{Ca89,Ca90,Ca901} developed a general
regularity theory for the Monge-Amp\`ere equation.
For positive H\"older continuous data and strictly positive support
functions, Caffarelli's theory applies directly to the $L_p$ Minkowski
and dual Minkowski problems once the necessary bounds are established,
yielding $C^{2,\alpha}$ regularity and smoothness for smooth
data; see, for example, Chou-Wang \cite{CW06}, Huang-Zhao \cite{HZ18} and Li-Sheng-Wang \cite{LSW20}.
For the chord log-Minkowski problem, applying this theory additionally
requires boundary regularity estimates for the nonlocal term $\widetilde{V}_{q-1}$. Theorem \ref{thm2} provides the 
ingredient needed to apply Caffarelli's regularity theory in this setting. Combining this with the existence result for weak solutions to
\eqref{CLMP} established in \cite[Theorem 1.2]{LXYZ}
for $1\leq q\leq n+1$, we obtain Theorem \ref{thm3}.

Since there are no results for weak solutions to the chord log-Minkowski problem for $q>n+1$, in this case, we treat the existence of a solution by using a geometric flow, which corresponds to the solvability of \eqref{CLMP}. The geometric flow involving Gauss curvature was first introduced
by Firey \cite{F74}, it describes the shape of worn stone. Since then, numerous variants of Gauss curvature flow have emerged, for instance, Chou-Wang
\cite{CW00} employed a logarithmic Gauss curvature flow to tackle the
classical Minkowski problem, Li-Sheng-Wang \cite{LSW20} and Chen-Huang-Zhao \cite{CHZ19} respectively made use of the related
Gauss curvature flows to solve the dual Minkowski type problems. In contrast to previous Gauss curvature flows, we study a  family of convex hypersurfaces $\partial \Omega_{t}$ parameterized by a smooth map
$X: \mathbb{S}^{n-1}\times[0,T) \rightarrow \R^{n}$
 satisfying the following nonlocal Gauss curvature flow equation:
\begin{equation}\label{flow}
  \begin{split}
  \left\{
\begin{array}{lr}
    \frac{\pd X}{\pd t} (x,t)
    = -\frac{f(\nu) \kappa}{\widetilde{V}_{q-1}(\Omega_t,X)} \nu +  X,\\
    X(x,0)= X_{0}(x),
    \end{array}\right.
  \end{split}
\end{equation}
with $I_{q}(\Omega_0)
  =\frac{2q\int_{\uS}f\dd x}{(q+n-1)\omega_{n}}$.
Here $\nu$ is the unit outer normal vector of the hypersurface $\partial \Omega_t$ at the boundary
point $X(x,t)$,  $\kappa$ is the Gauss curvature of
$\partial \Omega_{t}$ at $X(x,t)$, and $T$ is the maximal time for which the solution exists. To the best of our knowledge, the flow \eqref{flow} has never been studied before.

We now establish that the solution $h(x,t)$ of the flow \eqref{flow} exists for all time $t\in(0,+\infty)$ and converges as $t\to\infty$ to a function $h(x)$, which is a solution of \eqref{CLMP}. We emphasize that the non-trivial challenge in analyzing the flow comes from the nonlocal term  $\widetilde{V}_{q}$ associated with the boundary. This term fundamentally reflects the inherent complexity of handling the flow \eqref{flow}, presenting substantial difficulties. Our result is as follows.
\begin{theo}\label{thm4}
Suppose $q\in(3,+\infty)$.  Let $\Omega_{0}$ be a smooth, origin-symmetric and uniformly convex body in $\rnnn$ satisfying $I_{q}(\Omega_{0})=\frac{2q\int_{\sn}f\dd x}{(q+n-1)\omega_{n}}$. Let $f$ be a positive, even and smooth function on $\sn$. Then
there exists a smooth, origin-symmetric and uniformly convex solution $\Omega_{t}$ to the
flow \eqref{flow} for all $t>0$,  and  its subsequence in $C^{\infty}$ converges to a smooth, origin-symmetric and uniformly convex solution. As a result, there exists a smooth, origin-symmetric and uniformly convex solution to Eq. \eqref{CLMP}.
\end{theo}

 It is worth emphasizing that when $q > n+1$, Theorem \ref{thm4} provides the first existence result for a smooth solution to the chord log-Minkowski problem; while the existence for general prescribed measures remains open.  Here the condition $q>3$ ensures the completion of a priori estimates for \eqref{flow}. Note that the presence of the nonlocal term $\widetilde{V}_{q-1}$ in \eqref{CLMP} makes it a novel {M}onge-{A}mp\`ere equation arising in Minkowski-type problems. Hence, the treatment of \eqref{flow} requires a somewhat different approach from those used for earlier Minkowski flows.

 To prove Theorem \ref{thm4}, it is essential to derive the smoothness estimation on $\widetilde{V}_{q}$ along the boundary of a convex body.
In previous studies, the Riesz potential (cf. \cite{RI26,RI30,RI38,RI49}) has been used to derive embedding results for functions in Sobolev spaces,  H\"{o}lder spaces, and other spaces (see, e.g., \cite{BA80,Bu23,Ga58,HL28,HL30,L83,N59,Re09,SF3,SI17,SO38} and the references therein). However, as far as we know, regularity estimates of the Riesz potential on the boundary of a domain have largely remained elusive.

In light of \eqref{Vqq}, since $\widetilde{V}_q$ is invariant under translations in $\R^n$, we will always assume that the convex
body $\Omega\subset\R^n$ contains the origin in its interior throughout this paper.
Then the boundary $\pd\Omega$ can be given via its radial function
$\rho : \uS\to(0,+\infty)$ as
\begin{equation*} \label{eq:21}
  \pd\Omega =\set{\rho(u)u : \ \ \forall u\in\uS}.
\end{equation*}

 We have the following boundary regularity results for $\widetilde{V}_{q}$.

\begin{theorem}\label{thm2} The generalized dual quermassintegral $\widetilde{V}_q$   demonstrates the following regularity properties.
\begin{itemize}
\item[$(i)$]
Assume $q>1$, $q\neq n$. Let $k$ be a positive integer.
If $\rho\in C^{k,\alpha}(\uS)$ for some $\alpha\in [0,1]$,
then $\widetilde{V}_{q}\in C^{k,\alpha}(\pd\Omega)$, and there holds
\begin{equation*}\label{eq:223}
  \Vert\widetilde{V}_{q}\Vert_{k,\alpha;\pd\Omega} \leq
  \hat{C}
\end{equation*}
for a positive constant $\hat{C}$ depending only on $n,q,\alpha, k, \norm{\rho}_{k,\alpha}$ and $ 1/{\rho_{\min}}$.

\item[$(ii)$]
When $0<q\leq 1$.
Let $k$ be a positive integer.
If $\rho\in C^{k,\alpha}(\uS)$ for some $\alpha\in (1-q,1]$,
then $\widetilde{V}_q\in C^{k,\alpha'}(\pd\Omega)$ for any $\alpha'\in (0, \alpha+q-1)$, and there holds
\begin{equation*}
  \Vert\widetilde{V}_q\Vert_{k,\alpha';\pd\Omega} \leq
  \bar{C}
\end{equation*}
for a positive constant $\bar{C}$ depending only on $n, q, \alpha, \alpha', k, \norm{\rho}_{k,\alpha}$ and $ 1/{\rho_{\min}}$.

\item[$(iii)$]
When $-1<q< 0$.
Let $k$ be a positive integer.
If $\rho\in C^{k,\alpha}(\uS)$ for some $\alpha\in (-q,1)$,
then $\widetilde{V}_q\in C^{k-1,\alpha^{*}}(\pd\Omega)$ for any $\alpha^{*}\in (0, \alpha+q)$, and there holds
\begin{equation*}
  \Vert\widetilde{V}_q\Vert_{k-1,\alpha^{*};\pd\Omega} \leq
  C^{*}
\end{equation*}
for a positive constant $C^{*}$ depending only on $n, q, \alpha,  \alpha^{*}, k, \norm{\rho}_{k,\alpha}$ and $ 1/{\rho_{\min}}$.
\end{itemize}
\end{theorem}

We briefly note that when $-1<q<0$, $\widetilde{V}_{q}(\cdot)$ is indeed named
as the nonlocal (fractional) mean curvature, which is given by the variational
formula of $(-q)$-th perimeter when the boundary of a domain is of class
$C^{1,\alpha}$ for $\alpha>-q$, as first discovered in \cite{CR10}, see also
\cite{CFM18,F15}. Note that, after suitable normalization, the limit of the
nonlocal mean curvature $(1+q)\widetilde{V}_{q}$ as $q\rightarrow (-1)^{+}$ is the classical mean curvature, 
as established in \cite{AV14,CDNP21,LXYZ}. Within these years, some important results with regard to the regularity of the nonlocal mean curvature have been derived.  Figalli-Fusco-Maggi-Millot-Morini \cite[Proposition 6.3]{F15} say that when $\partial \Omega$ is of class $C^{1,\alpha}$, with $\alpha >-q$, $\widetilde{V}_{q}(z)$ is finite for all $z\in \partial \Omega$ and is a continuous function on $\partial \Omega$. Moreover, if $\partial \Omega$ is of class $C^{2,\alpha}$ with $\alpha>-q$, as exposed in \cite[Proposition 2.1]{Cab18},  then $\widetilde{V}_{q}(z)\in C^{1}(\partial \Omega)$. The third assertion in Theorem \ref{thm2}  may represent some progress toward the higher regularity estimates of nonlocal mean curvature.

To prove Theorem \ref{thm2}, we study a generalized Riesz potential on the boundary of a convex body $\Omega$:
\begin{equation} \label{Na}
  N_{a}(z)
  =\int_{\Omega} \omega\Bigl( \frac{y-z}{|y-z|} \Bigr) \frac{\phi(y)}{|y-z|^{n-a}} \dd y,
  \quad a>0,\
  \forall z\in\pd\Omega,
\end{equation}
where $\omega$ and $\phi$ are functions defined on $\uS$ and $\overline{\Omega}$
respectively. When $\omega=\phi=1$, $N_a$ reduces to $\widetilde{V}_a$ up to a constant factor. Note that $N_a$ remains unchanged under translation in $\R^n$. Observe first that if
$\rho, \omega\in C^{\alpha}(\uS)$ for some $\alpha\in[0,1]$,
and $\phi\in L^{\infty}(\Omega)$,
then for each $a>0$, $N_a\in C^{\alpha'}(\pd\Omega)$ for some
$\alpha'\in[0,\alpha]$.
This fact can be easily seen by a direct integral estimation.
However, higher regularity seems impossible to obtain via direct integral estimations and one requires more sophisticated techniques. Importantly, by employing a spherical transformation technique, we can improve the regularity.

\begin{theorem}\label{thm1} The generalized Riesz potential defined in equation \eqref{Na} exhibits the following regularity properties.
\begin{itemize}
\item[$(i)$] When $a>1$.
If $\rho\in C^{k,\alpha}(\uS)$, $\omega\in C^{k,\alpha}\cap C^{1,1}(\uS)$
and $\phi\in C^{k-1,\alpha}(\Omega)$ for some integer $k\geq1$ and some
$\alpha\in[0,1]$,
then $N_a\in C^{k,\alpha}(\pd\Omega)$, and there holds
\begin{equation*}
  \norm{N_a}_{k,\alpha} \leq
C
\end{equation*}
for a positive constant $C$ depending only on $n, a, \alpha, k, \norm{\omega}_{k,\alpha}, \norm{\omega}_{1,1}, \norm{\phi}_{k-1,\alpha},
  \norm{\rho}_{k,\alpha}$ and $ 1/{\rho_{\min}}$.
\item[$(ii)$]
When $0<a\leq 1$.
If $\rho\in C^{k,\alpha}(\uS)$, $\omega\in C^{k,\alpha}\cap C^{1,1}(\uS)$ and $\phi\in C^{k,\alpha}(\Omega)$
for some integer $k\geq1$ and some $\alpha\in(1-a,1]$,
then $N_a$ is $C^{k,\alpha'}(\pd\Omega)$ for any $\alpha'\in(0,\alpha+a-1)$, and there holds
\begin{equation*}\label{eq:23}
  \norm{N_a}_{k,\alpha'} \leq
 \widetilde{ C}
\end{equation*}
for a positive constant $\widetilde{C}$ depending only on $n, a, \alpha, \alpha', k, \norm{\omega}_{k,\alpha}, \norm{\omega}_{1,1},
  \norm{\phi}_{k,\alpha}, \norm{\rho}_{k,\alpha}$ and  $1/{\rho_{\min}}.$
  \end{itemize}
\end{theorem}
Theorem \ref{thm1} says that for $a>1$, $N_a$ has the same smoothness as
$\pd\Omega$, this is clearly the optimal regularity for $N_a$ on $\pd\Omega$.
For $0<a\leq1$, the regularity stated in (ii) is also optimal. Notice that $\omega\in C^{1,1}$ in Theorem \ref{thm1} is
automatically satisfied when $k\geq2$.
For $k=1$, if we only assume $\omega\in C^{1,\alpha}$ with $\alpha<1$,
$N_a$ is still $C^{1,\alpha}$ when $a\geq2$,
$C^{1,\tilde{\alpha}}$ for some $\tilde{\alpha}\in[0,\alpha]$ when $1<a<2$,
and $C^{1,\alpha'}$ for some $\alpha'\in(0,\alpha+a-1)$ when $0<a\leq1$.
See remark \ref{rmk0107}.  It is worth mentioning that Theorem \ref{thm1} also helps to obtain the existence of smooth solutions to other chord Minkowski-type problems, such as the $L_p$ chord Minkowski problem. Further discussion is presented in subsequent papers (see, e.g., \cite{CLL26,Hu25}).

The organization of this paper goes as follows. In Section \ref{sec2}, we list some basic facts from convex  geometry, differential geometry and integral geometry. In Section \ref{sec3}, we provide the proofs of Theorem \ref{thm2} and Theorem \ref{thm1}, which establish regularity results for the Riesz potential on the boundary of a convex domain. As a consequence, we obtain Theorem \ref{thm3}. Moving forward to Section \ref{sec4}, we establish the related nonlocal Gauss curvature flow and the functional, moreover, $C^{0}$, $C^{1}$ estimates are also derived. In Section \ref{sec5}, the derivatives and variations of the Riesz potential are obtained.  We show $C^{2}$ estimates in Section \ref{sec6}. The proof of Theorem \ref{thm4} will be presented in  Section \ref{sec7}.

\section{Preliminaries}
\label{sec2}
Our setting will be in the $n$-dimensional Euclidean space ${\rnnn}$. Denote by $B_{R}$ the Euclidean ball centred at the origin with radius $R>0$ in ${\rnnn}$ , by ${\sn}$ the unit sphere.  $\omega_{n}$ is the volume of unit ball in $\rnnn$. For $y,z\in {\rnnn}$, $\langle y,z\rangle$ denotes the standard inner product.  In what follows, for the readers' convenience, we recall some basic facts from convex geometry, differential geometry and integral geometry.

  A convex body is a compact convex set of ${\rnnn}$ with non-empty interior. There are many standard references on convex body, such as the books of Gardner \cite{G06} and Schneider \cite{S14}.

For two convex bodies  $\Omega_1, \Omega_2$ in $\rnnn$, the \emph{Minkowski combination}  $\Omega_1+t\Omega_2$ with $t\geq 0$ can be defined by
\[
\Omega_1+t\Omega_2=\{y+tz: y\in \Omega_1, z\in \Omega_2\}.
\]

As in  \cite{LXYZ}, the (extended) radial function $\rho_{\Omega,z}(x):\rnnn \backslash \{0\}\rightarrow \R$ regarding $z\in \rnnn$, is defined by
\begin{equation*}\label{p1}
\rho_{\Omega,z}(y)=\max\{\lambda\in\R:\lambda y\in \Omega-z\},\quad y\in \rnnn \backslash \{0\}.
\end{equation*}
From the definition, it is clear to observe that
$ \rho_{\Omega,y+z}=\rho_{\Omega-y,z}=\rho_{\Omega-y-z}, \quad y,z\in \rnnn $.

The (extended) support function $h_{\Omega,z}:\rnnn\rightarrow \R$ of $\Omega$ with respect to $z$, is defined by
\begin{equation*}\label{h1}
h_{\Omega,z}(x)=\max \{\langle x,y\rangle:y \in \Omega-z\}, \quad x\in \rnnn.
\end{equation*}
 Obviously, $ h_{\Omega,y+z}(x)=h_{\Omega-y,z}=h_{\Omega}(x)-\langle y+z,x\rangle, \quad x,y,z\in \rnnn.$

For the support function and radial function of $\Omega$ with respect to the origin, we will later write $h_{\Omega}$, $\rho_{\Omega}$ rather than $h_{\Omega,o}$, $\rho_{\Omega,o}$.

For a convex body $\Omega$ containing the origin in its interior, the radial map of $\Omega$ is:
\begin{equation}\label{radialMap}
r_{\Omega}: \sn\rightarrow \partial \Omega, \quad {\rm by} \ r_{\Omega}(u)=\rho_{\Omega}(u)u \ {\rm for} \ u\in \sn.
\end{equation}

Given a compact convex set $\Omega$ in $\rnnn$, for $\mathcal{H}^{n-1}$ almost all $X\in \partial \Omega$, the outer unit normal of $\Omega$ at $X$ is unique. In this case, we denote the Gauss map by $\nu_{\Omega}$, which maps $X\in \partial \Omega$ to its unique outer unit normal.
Meanwhile, for $\omega\subset {\sn}$, the inverse Gauss map $\nu^{-1}_{\Omega}$ is defined as
\begin{equation*}
\nu_{\Omega}^{-1}(\omega)=\{X\in \partial \Omega: \nu_{\Omega}(X) {\rm  \ is \ defined \ and }\ \nu_{\Omega}(X)\in \omega\}.
\end{equation*}

For a smooth and strictly convex body $\Omega$, namely, the boundary of $\Omega$ is of class $C^{\infty}$ and of positive Gauss curvature, the support function of $\Omega$ can be expressed as
\begin{equation*}\label{hhom}
h_{\Omega}(x)=\langle x,\nu^{-1}_{\Omega}(x)\rangle=\langle\nu_{\Omega}(X), X\rangle, \ {\rm where} \ x\in {\sn}, \ \nu_{\Omega}(X)=x \ {\rm and} \ X\in \partial \Omega.
\end{equation*}
 Let $\{e_{1},e_{2},\ldots, e_{n-1}\}$ be a local orthonormal frame on ${\sn}$, $h_{i}$ and $h_{ij}$ be the first and second order covariant derivatives of $h_{\Omega}$ on ${\sn}$ with respect to the frame, then we have
\begin{equation*}\label{Fdef}
\nu^{-1}_{\Omega}(x)=\overline{\nabla}h_{\Omega}(x)=\sum_{i} h_{i}e_{i}+h_{\Omega}(x)x=\nabla h_{\Omega}(x)+h_{\Omega}(x)x,
\end{equation*}
where $\nabla$ is the (standard) spherical gradient. Moreover, the Gauss curvature $\kappa$ of $\partial\Omega$ at $X$ is expressed as
\begin{equation*}
\kappa=\frac{1}{\det(h_{ij}+h\delta_{ij})}.
\end{equation*}

  As a counterpart of \eqref{chord}, given a convex body $\Omega$ in $\rnnn$, the chord integral $I_{q}(\Omega)$ can also be represented in terms of $X$-rays as (see \cite{G06})
\begin{equation*}\label{I1}
I_{q}(\Omega)=\frac{1}{n\omega_{n}}\int_{\sn}\int_{\Omega|u^{\bot}}X_{\Omega}(y,u)^{q}\dd y \dd u,\quad q> -1,
\end{equation*}
where $\Omega|u^{\bot}$ is the image of the orthogonal projection of $\Omega$ onto $u^{\bot}$, and the parallel $X$-ray of $\Omega$ is given as
\begin{equation*}\label{}
X_{\Omega}(y,u)=|\Omega \cap (y+\rn u)|, \quad y \in \rnnn, \ u\in \sn.
\end{equation*}
Note that the relation between the $X$-ray function and the radial function is revealed as
\begin{equation*}\label{xray}
X_{\Omega}(y,u)=\rho_{\Omega,z}(u)+\rho_{\Omega,z}(-u), \ {\rm  when} \ \Omega \cap (y+\rn u)=\Omega \cap (z+\rn u)\neq\emptyset.
\end{equation*}
An important property of $I_{q}$ is that it is homogeneous of degree $(n+q-1)$ for $q>-1$ (see \cite[Lemma 3.2]{LXYZ}).

 As established in \cite[Theorem 5.5, Lemma 5.6]{LXYZ}, for a convex body $\Omega$ in $\rnnn$, for $q>0$,
the chord measure $F_{q}(\Omega,\cdot)$ defined in \eqref{FD} arises from the first variation of chord integrals under the Minkowski combination, while the cone-chord measure $G_{q}(\Omega,\cdot)$ defined in \eqref{GD} stems from the first variation under the log-Minkowski sum.
Note that the homogeneities of the chord measure and cone-chord measure are $(n+q-2)$ and $(n+q-1)$ respectively.

Note that for $q>0$, by \cite[Theorem 4.3]{LXYZ}, one sees
\begin{equation}\label{c-i-formula}
  I_{q}(\Omega)=\frac{1}{q+n-1}\int_{\uS}h_{\Omega}(x)\dd F_{q}(\Omega,x).
\end{equation}

 As shown in \cite{LXYZ}, for $q\in \R$, on the one hand, the $q$-th (generalized) dual quermassintegral $\widetilde{V}_{q}(\Omega,z)$ of $\Omega$ with respect to $z\in \Omega$ is defined by
\begin{equation*}\label{VQ}
\widetilde{V}_{q}(\Omega,z)=\widetilde{V}^{+}_{q}(\Omega,z).
\end{equation*}
On the other hand, for $q>0$, and $z\notin \Omega$, define $\widetilde{V}_{q}(\Omega,z)$ by
\[
\widetilde{V}_{q}(\Omega,z)=\widetilde{V}^{+}_{q}(\Omega,z)-\widetilde{V}^{-}_{q}(\Omega,z).
\]
Here
\[
\widetilde{V}^{+}_{q}(\Omega,z)=\frac{1}{n}\int_{S^{+}_{z}}\rho_{\Omega,z}(u)^{q}du,\quad \widetilde{V}^{-}_{q}(\Omega,z)=\frac{1}{n}\int_{S^{-}_{z}}|\rho_{\Omega,z}(u)|^{q}du,
\]
where
\[
S^{+}_{z}=\{u\in \sn: \ \rho_{\Omega,z}(u)>0\},\quad S^{-}_{z}=\{u\in \sn: \ \rho_{\Omega,z}(u)<0\}.
\]
Since $S^{+}_{z}=\sn$ when $z\in {\rm int} \ \Omega$, $\rho_{\Omega,z}(u)=0$ when $z\in \partial \Omega$ and $u$ is in the interior of $\sn \backslash S^{+}_{z}$, we obtain
\[
\widetilde{V}_{q}(\Omega,z)=\frac{1}{n}\int_{\sn}\rho_{\Omega,z}(u)^{q}du,
\]
when either $q\in \R$ and $z\in {\rm int} \ \Omega$, or $q>0$ and $z\in \partial \Omega$.

In the whole paper,  unless it causes confusion, we abbreviate $h_{\Omega}(\cdot)$, $\rho_{\Omega}(\cdot)$ as $h(\cdot)$, $\rho(\cdot)$, and abbreviate $h_{\Omega,z}(\cdot)$, $\rho_{\Omega,z}(\cdot)$ as  $h_z(\cdot)$, $\rho_z (\cdot)$ respectively.

\section{Boundary regularity of the Riesz potential} \label{sec3}
In this section, our purpose is to prove Theorems \ref{thm2} and \ref{thm1}. Since Theorem \ref{thm1} is a generalization of Theorem \ref{thm2}, we only need to prove Theorem \ref{thm1}.  Recall
\begin{equation} \label{Na2}
  N_{a}(z)
  =\int_{\Omega} \omega\Bigl( \frac{y-z}{|y-z|} \Bigr) \frac{\phi(y)}{|y-z|^{n-a}} \dd y,
  \quad
  \forall z\in\pd\Omega,
\end{equation}
where $a>0$, and $\omega,\phi$ are functions defined on $\uS$ and $\Omega$
respectively.

 We first present the following
H\"older regularity result of $N_{a}$ for $a>0$.

\subsection{H\"older regularity of $N_a$ for $a>0$}

The local H\"older continuity of the Riesz potential for $z\in\R^n$ is well known 
and directly implies its H\"older continuity on the boundary. 
For the reader's convenience and for later use, we provide a complete proof.

We begin by establishing a basic $C^0$ estimate.

\begin{lemma}\label{lem1128}
Let $a>0$.
Assume $\omega\in L^{\infty}(\uS)$ and $\phi\in L^{\infty}(\Omega)$.
For any $z\in\R^n$, let $R$ be a positive number such that
$\Omega\subset B_{R}(z)$, then
\begin{equation}\label{eq:106}
  |N_{a}(z)|
  \leq
  \frac{n\omega_n}{a}
  \norm{\omega}_{\infty} \norm{\phi}_{\infty}
  R^a.
\end{equation}
Therefore, for any compact set $\Omega'\supset\Omega$,
$N_a\in L^{\infty}(\Omega')$, and there is
\begin{equation} \label{Na-C0}
  \sup_{\Omega'}|N_a| \leq
  \frac{n\omega_n}{a}
  \norm{\phi}_{\infty}
  \norm{\omega}_{\infty}
  \DIAM(\Omega')^a.
\end{equation}
\end{lemma}

\begin{proof}{}
For any given $z\in\R^n$,
if $\rho_z(u)>0$ for some $u\in\uS$, define
\begin{equation*}
  \sigma_z(u)=\inf\set{\lambda\geq0:z+\lambda u\in\Omega}.
\end{equation*}
Note that $0\leq\sigma_z(u)\leq\rho_z(u)$.
Using the spherical coordinates with respect to $z$,
we can rewrite $N_a(z)$ as
\begin{equation}\label{eq:90}
  N_{a}(z)
  =
  \int_{|u|=1 \atop \rho_z(u)>0} \omega(u) \dd u
  \int_{\sigma_z(u)}^{\rho_z(u)} \phi(z+ru) r^{a-1} \dd r.
\end{equation}
Since $\Omega\subset B_R(z)$, we have $\rho_z(u)\leq R$.
Therefore,
\begin{equation*}
  \begin{split}
    |N_{a}(z)|
    &\leq
    \norm{\omega}_{\infty} \norm{\phi}_{\infty}
    \int_{|u|=1 \atop \rho_z(u)>0} \dd u
    \int_{\sigma_z(u)}^{\rho_z(u)} r^{a-1} \dd r \\
    &\leq
    \norm{\omega}_{\infty} \norm{\phi}_{\infty}
    \int_{|u|=1 \atop \rho_z(u)>0} \dd u
    \int_{0}^{R} r^{a-1} \dd r \\
    &\leq
    \frac{n\omega_n}{a}
    \norm{\omega}_{\infty} \norm{\phi}_{\infty} R^a,
  \end{split}
\end{equation*}
which is just the inequality \eqref{eq:106}.
\end{proof}

To prove H\"older continuity, we will employ the following simple but useful inequality as below.

\begin{lemma}\label{lem011}
For any $\beta\in\R$, there exists a positive constant $C_{\beta}$ depending
only on $\beta$, such that for any $0<\gamma\leq1$, we have
\begin{equation} \label{eq:65}
  |s^{\beta}-t^{\beta}|\leq
  C_{\beta} \max\set{s^{\beta-\gamma},t^{\beta-\gamma}} |s-t|^{\gamma},
  \quad
  \forall s,t>0.
\end{equation}
\end{lemma}

\begin{proof}{}
Without loss of generality, we only need to prove inequality \eqref{eq:65} for
$0<t<s$.
Writing $t=rs$ with $0<r<1$, it is equivalent to
\begin{equation*}
  |1-r^{\beta}|\leq C_{\beta} \max\set{1,r^{\beta-\gamma}} (1-r)^{\gamma},
  \quad
  \forall\, r\in(0,1),
\end{equation*}
namely
\begin{equation} \label{eq:66}
  \min\set{1,r^{\gamma-\beta}}
  |1-r^{\beta}|
  (1-r)^{-\gamma}
  \leq C_{\beta},
  \quad
  \forall\, r\in(0,1).
\end{equation}
Denote the left hand side of \eqref{eq:66} by $\varphi(r)$.
One should verify
\begin{equation} \label{eq:67}
\varphi(r)\leq C_{\beta},
\quad
\forall\,r\in(0,1).
\end{equation}

Since $r^\beta$ is smooth in the closed interval $[1/2,1]$, there exists a
positive constant $c_\beta$ such that
\begin{equation*}
|1-r^{\beta}|\leq c_{\beta} (1-r),
\quad
\forall\, r\in[\tfrac{1}{2},1].
\end{equation*}
Therefore, we have
\begin{equation} \label{eq:77}
  \begin{split}
    \varphi(r)
    &\leq
    |1-r^{\beta}|
    (1-r)^{-\gamma} \\
    &\leq
    c_{\beta} (1-r)^{1-\gamma} \\
    & \leq
    c_{\beta},
    \quad
    \forall\, r\in[\tfrac{1}{2},1).
  \end{split}
\end{equation}
For $r\in(0,\tfrac{1}{2})$, there is $(1-r)^{-\gamma}<2^{\gamma}\leq2$, implying
\begin{equation} \label{eq:82}
  \varphi(r)
  \leq
  2 \min\set{r^{\gamma-\beta},1}
  |1-r^{\beta}|,
  \quad
  \forall\,r\in(0,\tfrac{1}{2}).
\end{equation}
When $\gamma-\beta\geq0$, we have $r^{\gamma-\beta}\leq1$.
Then, \eqref{eq:82} becomes
\begin{equation*}
  \varphi(r)
  \leq
  2r^{\gamma-\beta}
  |1-r^{\beta}|
  =
  2|r^{\gamma-\beta}-r^{\gamma}|
  \leq 4.
\end{equation*}
When $\gamma-\beta<0$, we have $\beta>0$ and $r^{\gamma-\beta}>1$.
Now, \eqref{eq:82} reads
\begin{equation*}
  \varphi(r)
  \leq
  2 |1-r^{\beta}|
  \leq 4.
\end{equation*}
Thus
\begin{equation} \label{eq:87}
  \varphi(r)\leq4,
  \quad
  \forall\,r\in(0,\tfrac{1}{2}).
\end{equation}
Now combining \eqref{eq:77} with \eqref{eq:87}, we obtain
\eqref{eq:67} if $C_{\beta}$ is chosen as $\max\set{c_{\beta},4}$.
The proof of this lemma is completed.
\end{proof}

Based on previous Lemmas \ref{lem1128} and \ref{lem011}, we are now in a position to establish the local H\"older continuity of the Riesz potential.

\begin{lemma}\label{lem601}
Let $a>0$.
Assume $\phi\in L^{\infty}(\Omega)$ and $\omega\in C^{\alpha}(\uS)$
where $\alpha\in[0,1]$ and $\alpha<a$.
Then for any compact set $\Omega'\supset\Omega$,
$N_a\in C^{\alpha}(\Omega')$,
and there exists a positive constant $C$ depending only on $n$ and $a$, such that
\begin{equation} \label{Na-C-alpha}
  \norm{N_a}_{\alpha;\Omega'}
  \leq
  C
  \norm{\phi}_{\infty}
  \norm{\omega}_{\alpha}
  \left(
    \frac{1}{a} \DIAM(\Omega')^{a}
    +
    \frac{1}{a-\alpha} \DIAM(\Omega')^{a-\alpha}
  \right).
\end{equation}
\end{lemma}

\begin{proof}{}
We first consider the case for $0<\alpha\leq1$.
Recalling the definition of $N_a$, for any $z_1,z_2\in\Omega'$, we have
\begin{equation}\label{eq:86}
  \begin{split}
    |N_a(z_1)&-N_a(z_2)| \\
    &=
    \biggl|
    \int_{\Omega} \phi(y)\omega\Bigl( \frac{y-z_1}{|y-z_1|} \Bigr) |y-z_1|^{a-n} \dd y
    -\int_{\Omega} \phi(y)\omega\Bigl( \frac{y-z_2}{|y-z_2|} \Bigr) |y-z_2|^{a-n} \dd y
    \biggr|
    \\
    &\leq
    \int_{\Omega}
    |\phi(y)|
    \cdot
    \Bigl| \omega\Bigl( \frac{y-z_1}{|y-z_1|} \Bigr) \Bigr|
    \cdot
    \bigl| |y-z_1|^{a-n} - |y-z_2|^{a-n} \bigr|
    \dd y
    \\
    &\hskip4.9em
    +\int_{\Omega}
    |\phi(y)| \cdot |y-z_2|^{a-n}
    \cdot
    \Bigl|
    \omega\Bigl( \frac{y-z_1}{|y-z_1|} \Bigr)
    -\omega\Bigl( \frac{y-z_2}{|y-z_2|} \Bigr)
    \Bigr|
    \dd y
    \\
    &=: I +\II.
  \end{split}
\end{equation}

For the first term $I$, by virtue of Lemma \ref{lem011}, there is
\begin{equation*}
  \begin{split}
    \bigl| |y-z_1|^{a-n}-|y-z_2|^{a-n} \bigr|
    &\leq
    C_1 \left(|y-z_1|^{a-n-\alpha}+|y-z_2|^{a-n-\alpha}\right)
    \bigl| |y-z_1|-|y-z_2| \bigr|^{\alpha} \\
    &\leq
    C_1 \left(|y-z_1|^{a-n-\alpha}+|y-z_2|^{a-n-\alpha}\right)
    |z_1-z_2|^{\alpha},
  \end{split}
\end{equation*}
where $C_1$ is a positive constant depending only on $a-n$.
Thus,
\begin{equation} \label{eq:88}
  \begin{split}
    I
    &\leq
    \norm{\phi}_{\infty}
    \norm{\omega}_{\infty}
    C_1
    |z_1-z_2|^{\alpha}
    \int_{\Omega}
    \left(|y-z_1|^{a-n-\alpha}+|y-z_2|^{a-n-\alpha}\right)
    \dd y
    \\
    &\leq
    \norm{\phi}_{\infty}
    \norm{\omega}_{\infty}
    C_1
    |z_1-z_2|^{\alpha}
    \cdot
    \frac{2n\omega_n}{a-\alpha}
    \DIAM(\Omega')^{a-\alpha},
  \end{split}
\end{equation}
where the second inequality follows from estimation \eqref{Na-C0} with $a$ replaced by $a-\alpha$.

For the second item $\II$, we note that for any non-zero vectors $y_1,y_2\in\R^n$, there holds
\begin{equation}\label{eq:85}
  \begin{split}
    \left|
      \frac{y_1}{|y_1|} - \frac{y_2}{|y_2|}
    \right|
    &\leq
    \left|
      \frac{y_1}{|y_1|} - \frac{y_1}{|y_2|}
    \right|
    +
    \left|
      \frac{y_1}{|y_2|} - \frac{y_2}{|y_2|}
    \right|
    \\
    &=
    \frac{1}{|y_2|}
    \bigl| |y_2| - |y_1| \bigr|
    +
    \frac{1}{|y_2|}
    \left| y_1-y_2 \right|
    \\
    &\leq
    \frac{2}{|y_2|}
    \left| y_1-y_2 \right|.
  \end{split}
\end{equation}
Then, recalling $\omega\in C^{\alpha}(\uS)$, we have the estimates:
\begin{equation*}
  \begin{split}
    \left|
      \omega\Bigl( \frac{y-z_1}{|y-z_1|} \Bigr)
      -\omega\Bigl( \frac{y-z_2}{|y-z_2|} \Bigr)
    \right|
    & \leq
    \left|
      \frac{y-z_1}{|y-z_1|} - \frac{y-z_2}{|y-z_2|}
    \right|^{\alpha}
    \cdot
    [\omega]_{\alpha}
    \\
    &\leq
    \frac{2 [\omega]_{\alpha}}{|y-z_2|^\alpha}
    \cdot
    \left| z_1-z_2 \right|^{\alpha}.
  \end{split}
\end{equation*}
Now, we find
\begin{equation} \label{eq:89}
  \begin{split}
    \II
    &\leq
    \norm{\phi}_{\infty}
    \int_{\Omega}
    |y-z_2|^{a-n}
    \cdot
    \Bigl|
    \omega\Bigl( \frac{y-z_1}{|y-z_1|} \Bigr)
    -\omega\Bigl( \frac{y-z_2}{|y-z_2|} \Bigr)
    \Bigr|
    \dd y
    \\
    &\leq
    2\norm{\phi}_{\infty}
    [\omega]_{\alpha}
    \left| z_1-z_2 \right|^{\alpha}
    \int_{\Omega}
    |y-z_2|^{a-n-\alpha}
    \dd y
    \\
    &\leq
    \norm{\phi}_{\infty}
    [\omega]_{\alpha}
    \left| z_1-z_2 \right|^{\alpha}
    \cdot
    \frac{2n\omega_n}{a-\alpha}
    \DIAM(\Omega')^{a-\alpha},
  \end{split}
\end{equation}
where the last inequality follows from estimation \eqref{Na-C0} with $a$ replaced by $a-\alpha$.

Combining \eqref{eq:86}, \eqref{eq:88} and \eqref{eq:89}, we have
\begin{equation}\label{eq:109}
  |N_a(z_1)-N_a(z_2)|
  \leq
  \frac{C_2}{a-\alpha}
  \norm{\phi}_{\infty}
  \norm{\omega}_{\alpha}
  \DIAM(\Omega')^{a-\alpha}
  \left| z_1-z_2 \right|^{\alpha},
\end{equation}
where $C_2$ is a positive constant depending only on $n$ and $a$.
Therefore, we can  obtain
\begin{equation} \label{Na-semi}
  [N_a]_{\alpha;\Omega'}
  \leq
  \frac{C_2}{a-\alpha}
  \norm{\phi}_{\infty}
  \norm{\omega}_{\alpha}
  \DIAM(\Omega')^{a-\alpha},
\end{equation}
which, together with the $C^0$ estimate \eqref{Na-C0}, implies \eqref{Na-C-alpha} for $0<\alpha\leq1$.

Based on \eqref{Na-C-alpha}, for some $0<\alpha_0\leq1$ and $\alpha_0<a$, we see that
$\displaystyle\int_{\Omega} |y-z|^{a-n} \dd y \in C^{\alpha_0}(\Omega')$, this implies
\begin{equation}\label{eq:91}
  \lim_{z\to\tilde{z}}
  \int_{\Omega} |y-z|^{a-n} \dd y
  =
  \int_{\Omega} |y-\tilde{z}|^{a-n} \dd y.
\end{equation}
For the remaining case when $\alpha=0$, observing
\begin{equation*}
  |N_a(z)|
  \leq
  \norm{\phi}_{\infty}
  \norm{\omega}_{0}
  \int_{\Omega} |y-z|^{a-n} \dd y,
\end{equation*}
by virtue of the dominated convergence theorem and \eqref{eq:91}, then we conclude that $N_a(z)$ is
continuous on $\Omega'$. The estimate \eqref{Na-C-alpha} for $\alpha=0$ is
exactly the $C^0$ estimate \eqref{Na-C0}.
The proof of this lemma is completed.
\end{proof}
Next, our aim is to establish the higher regularity of $N_a$ for $a>1$.

\subsection{Higher regularity of $N_a$ for $a>1$}

Note that if the differentiation of $N_a$ can be taken under the integral sign,
with the help of Lemma \ref{lem601},
one will clearly see that $N_a\in C^k$ with $k$ the greatest integer less than $a$.
We first verify this observation.

For notational convenience, we adopt the following notations.
Let
\begin{gather*}
  A, A_1, A_2, \cdots \in \set{1, 2, \cdots, n}, \\
  i, i_1, i_2, \cdots,
  \tilde{\alpha}, \alpha_1, \alpha_2, \cdots
  \in \set{1, 2, \cdots, n-1}
\end{gather*}
be summation indices,
$\set{E_1, \cdots, E_n}$ be a fixed orthonormal frame in $\R^n$, and
$\delbar$ be the differentiation in $\R^n$, as before.
For $z\in\pd\Omega$ and $\varsigma\in\uS$, we can write
\begin{equation*}
  z=z^AE_A
  \quad \text{ and } \quad
  \varsigma=\varsigma^AE_A.
\end{equation*}
Note that the functions $z^A$ and $\varsigma^A$ are  defined on $\pd\Omega$ and $\uS$
respectively.
Let $\set{z_1, \cdots, z_{n-1}}$ be a local frame on $\pd\Omega$,
$\set{e_1, \cdots, e_{n-1}}$ be a local orthonormal frame on $\uS$,
and $\nabla$ be both connections on $\pd\Omega$ and $\uS$.

For a positive integer $k$ and any function $\omega$ defined on $\uS$, we will
use $\pd^{k}\omega$ in the following sense:
\begin{equation*}
  \pd^{k}\omega \   \text{ denotes every } \   \pd^k_{\alpha_k\cdots \alpha_2 \alpha_1}\omega,
\end{equation*}
based on that, let
\begin{equation*}
  \pd^{\leq k}\omega \  \text{ be every } \  \omega, \pd^1\omega, \cdots, \pd^k\omega.
\end{equation*}
Similarly, for any function $z^A$ defined on $\pd\Omega$, let
$\nabla^kz^A$ be its every covariant derivative of order $k$ with respect to
the local frame $\set{z_1, \cdots, z_{n-1}}$, namely
\begin{equation*}
  \nabla^{k}z^A \   \text{ denotes every } \   z^A_{i_1 i_2 \cdots i_k}.
\end{equation*}
Then we have the corresponding notation $\nabla^{\leq k}z^A$.

Finally, for a positive integer $k$, we introduce the following index sets:
\begin{equation*}
  I_m^k= \set{
    (p_1, p_2, \cdots, p_m) :
    p_1+p_2+\cdots+p_m=k
  },
  \quad
  m=1, 2, \cdots, k,
\end{equation*}
where $p_1, p_2, \cdots, p_m$ are positive integers.

Building upon the above preparations, we now show an explicit expression for the $k$-th derivative of
 $N_a$ regarding the boundary of domain.

\begin{lemma}\label{lemNak}
Let $a>1$
and $k$ be a positive integer with $k<a$.
Assume
$\phi\in L^{\infty}(\Omega)$,
$\omega\in C^k(\uS)$,
and $\rho\in C^k(\uS)$.
Then $N_a\in C^k(\pd\Omega)$, and we have
\begin{equation}\label{dNa}
  \nabla^k N_a(z)
  =
  \sum_{m=1}^k
  \sum_{(p_1,\cdots,p_m)\in I^k_m}
  (\nabla^{p_1}z^{A_1})
  \cdots
  (\nabla^{p_m}z^{A_m})
  \widetilde{N}_{m;A_1 \cdots A_m}(z),
\end{equation}
where
\begin{equation} \label{eq:8}
  \widetilde{N}_{m;A_1 \cdots A_m}(z)
  = \int_{\Omega}
  P_{m+1}(\pd^{\leq m}\omega, \pd^{\leq m}\varsigma^A)\Bigl( \frac{y-z}{|y-z|} \Bigr)
  \frac{\phi(y)}{|y-z|^{n-(a-m)}} \dd y.
\end{equation}
Here $P_{m+1}$ $(\text{short for } P_{m+1;A_1 \cdots A_m})$ is an $(m+1)$-th
degree polynomial of variables taken in the set
$\set{\pd^{\leq m}\omega, \pd^{\leq m}\varsigma^A}$.
\end{lemma}

\begin{proof}
\textup{(1)}
We first use induction to prove \eqref{dNa} under the assumption that the
derivatives of $N_a$ can be taken under the integral sign.

For any fixed $y\in\Omega$, the composite function
$\varsigma^A\bigl( \frac{y-z}{|y-z|} \bigr)$
is a function defined on $\pd\Omega$.
In the following calculations, we abbreviate it as $\varsigma^A$ when no confusion arises.
By a direct computation, we derive
\begin{equation}\label{eq:10}
  \nabla_i |y-z|
  = -\Bigl\langle
  z_i,\frac{y-z}{|y-z|}
  \Bigr\rangle
  = -z^{A_1}_i \varsigma^{A_1},
\end{equation}
and
\begin{equation*}
  \nabla_i
  \left( \frac{y-z}{|y-z|} \right)
  =-|y-z|^{-1}
  \left(
    z_i
    -\Bigl\langle
    z_i,\frac{y-z}{|y-z|}
    \Bigr\rangle
    \frac{y-z}{|y-z|}
  \right).
\end{equation*}
Then, there holds
\begin{equation}\label{eq:9}
  \begin{split}
    \nabla_i \omega
    &= \nabla_i \left(
      \omega
      \Bigl( \frac{y-z}{|y-z|} \Bigr)
    \right) \\
    &= -|y-z|^{-1} \langle z_i,\varsigma_{\tilde{\alpha}} \rangle \pd_{\tilde{\alpha}} \omega \\
    &= -|y-z|^{-1} z^{A_1}_i \pd_{\tilde{\alpha}} \varsigma^{A_1} \cdot \pd_{\tilde{\alpha}} \omega.
  \end{split}
\end{equation}
In particular, we have
\begin{equation} \label{eq:11}
  \nabla_i \varsigma^A
  = \nabla_i \left(
    \varsigma^A\Bigl( \frac{y-z}{|y-z|} \Bigr)
  \right)
  = -|y-z|^{-1} z^{A_1}_i \pd_{\tilde{\alpha}} \varsigma^{A_1} \cdot \pd_{\tilde{\alpha}} \varsigma^A.
\end{equation}
Now recalling $N_a(z)$:
\begin{equation} \label{eq:12}
  N_{a}(z)
  =\int_{\Omega} \omega(\varsigma)|y-z|^{a-n} \phi(y) \dd y,
\end{equation}
and noting
\begin{equation*}
  \begin{split}
    \nabla_i ( \omega(\varsigma)|y-z|^{a-n} )
    &=
    -|y-z|^{a-1-n} z^{A_1}_i \pd_{\tilde{\alpha}} \varsigma^{A_1} \pd_{\tilde{\alpha}} \omega
    +(n-a) \omega |y-z|^{a-1-n} z^{A_1}_i \varsigma^{A_1} \\
    &=
    |y-z|^{a-1-n} z^{A_1}_i
    [(n-a) \omega  \varsigma^{A_1}
    - \pd_{\tilde{\alpha}} \varsigma^{A_1} \pd_{\tilde{\alpha}} \omega],
  \end{split}
\end{equation*}
thus we have
\begin{equation} \label{eq:13}
  \nabla_i N_{a}(z)
  = z^{A_1}_i
  \int_{\Omega}
  [(n-a) \omega  \varsigma^{A_1}
  - \pd_{\tilde{\alpha}} \varsigma^{A_1} \pd_{\tilde{\alpha}} \omega]
  \cdot
  |y-z|^{a-1-n}
  \phi(y) \dd y,
\end{equation}
which indicates that \eqref{dNa} is true for $k=1$.

Assuming \eqref{dNa} holds for $k$ with $k+1<a$, we need to show that it holds for $k+1$.
We start with doing the following differentiation:
\begin{equation} \label{eq:14}
  \begin{split}
    \nabla_i [ &
    (\nabla^{p_1}z^{A_1})
    \cdots
    (\nabla^{p_m}z^{A_m})
    \widetilde{N}_{m;A_1 \cdots A_m}(z) ] \\
    &=
    (\nabla^{p_1}z^{A_1})
    \cdots
    (\nabla^{p_\ell+1}z^{A_\ell})
    \cdots
    (\nabla^{p_m}z^{A_m})
    \widetilde{N}_{m;A_1 \cdots A_m}(z) \\
    &\hskip1.2em +
    (\nabla^{p_1}z^{A_1})
    \cdots
    (\nabla^{p_m}z^{A_m})
    \nabla_i [ \widetilde{N}_{m;A_1 \cdots A_m}(z) ].
  \end{split}
\end{equation}
Obviously, $(p_1, \cdots, p_\ell+1, \cdots, p_m)\in I^{k+1}_m$.
Therefore, the first term of the right hand side of \eqref{eq:14} is a summation
term of \eqref{dNa} with $k+1$.

It remains to handle the second term.
To compute $\nabla_i [ \widetilde{N}_{m;A_1 \cdots A_m}(z) ]$, we first carry out
the following computations.
By \eqref{eq:9}, we find
\begin{equation*}
  \begin{split}
    \nabla_i (\pd^{\leq m}\omega)
    &= -|y-z|^{-1} z^{A_1}_i \pd_{\tilde{\alpha}} \varsigma^{A_1} \cdot \pd_{\tilde{\alpha}} (\pd^{\leq m}\omega) \\
    &= -|y-z|^{-1} z^{A_1}_i \pd_{\tilde{\alpha}} \varsigma^{A_1} (\pd^{\leq m+1}\omega),
  \end{split}
\end{equation*}
which leads to
\begin{equation*}
  \begin{split}
    (\pd P_{m+1})(\pd^{\leq m}\omega, \pd^{\leq m}\varsigma^A) \cdot \nabla_i (\pd^{\leq m}\omega)
    &=
    P_{m}(\pd^{\leq m}\omega, \pd^{\leq m}\varsigma^A)
    \nabla_i (\pd^{\leq m}\omega) \\
    &=
    |y-z|^{-1} z^{A_1}_i
    P_{m}(\pd^{\leq m}\omega, \pd^{\leq m}\varsigma^A)
    \pd_{\tilde{\alpha}} \varsigma^{A_1} (\pd^{\leq m+1}\omega) \\
    &=
    |y-z|^{-1} z^{A_1}_i
    P_{m+2}(\pd^{\leq m+1}\omega, \pd^{\leq m}\varsigma^A).
  \end{split}
\end{equation*}
Here and in what follows, the polynomial $P$ may vary from line to line.
Particularly, we have
\begin{equation*}
  (\pd P_{m+1})(\pd^{\leq m}\omega, \pd^{\leq m}\varsigma^A) \cdot \nabla_i (\pd^{\leq m}\varsigma^A)
  =
  |y-z|^{-1} z^{A_1}_i
  P_{m+2}(\pd^{\leq m}\omega, \pd^{\leq m+1}\varsigma^A).
\end{equation*}
Then, one sees
\begin{equation*}
  \nabla_i [P_{m+1}(\pd^{\leq m}\omega, \pd^{\leq m}\varsigma^A) ]
  =
  |y-z|^{-1} z^{A_1}_i
  P_{m+2}(\pd^{\leq m+1}\omega, \pd^{\leq m+1}\varsigma^A),
\end{equation*}
which implies
\begin{multline} \label{eq:16}
  \nabla_i [ P_{m+1}(\pd^{\leq m}\omega, \pd^{\leq m}\varsigma^A) ] \cdot
  |y-z|^{(a-m)-n} \\
  =
  |y-z|^{a-(m+1)-n} z^{A_1}_i
  P_{m+2}(\pd^{\leq m+1}\omega, \pd^{\leq m+1}\varsigma^A).
\end{multline}
Since
\begin{multline*}
  P_{m+1}(\pd^{\leq m}\omega, \pd^{\leq m}\varsigma^A)
  \nabla_i |y-z|^{(a-m)-n} \\
  =
  (a-m-n)|y-z|^{a-(m+1)-n} \nabla_i |y-z|
  \cdot
  P_{m+1}(\pd^{\leq m}\omega, \pd^{\leq m}\varsigma^A),
\end{multline*}
which together with \eqref{eq:10}, yields
\begin{multline}\label{eq:18}
  P_{m+1}(\pd^{\leq m}\omega, \pd^{\leq m}\varsigma^A)
  \nabla_i |y-z|^{(a-m)-n} \\
  =
  |y-z|^{a-(m+1)-n}
  z^{A_1}_i
  P_{m+2}(\pd^{\leq m}\omega, \pd^{\leq m}\varsigma^A).
\end{multline}
Combining \eqref{eq:16} and \eqref{eq:18}, we obtain
\begin{equation*}
  \nabla_i [
  P_{m+1}(\pd^{\leq m}\omega, \pd^{\leq m}\varsigma^A)
  |y-z|^{(a-m)-n}
  ]
  =
  |y-z|^{a-(m+1)-n} z^{A_1}_i
  P_{m+2}(\pd^{\leq m+1}\omega, \pd^{\leq m+1}\varsigma^A).
\end{equation*}
Based on this formula,  from the expression \eqref{eq:8}, we have
\begin{equation}\label{eq:19}
  \begin{split}
    \nabla_i[\widetilde{N}_{m;A_1 \cdots A_m}(z)]
    &= \int_{\Omega}
    \nabla_i [
    P_{m+1}(\pd^{\leq m}\omega, \pd^{\leq m}\varsigma^A)
    |y-z|^{(a-m)-n}
    ]
    \phi(y) \dd y \\
    &= z^{A_1}_i \int_{\Omega}
    |y-z|^{a-(m+1)-n}
    P_{m+2}(\pd^{\leq m+1}\omega, \pd^{\leq m+1}\varsigma^A)
    \phi(y) \dd y \\
    &= z^{A_{m+1}}_i
    \widetilde{N}_{m+1;A_1 \cdots A_{m+1}}(z).
  \end{split}
\end{equation}
Now the second term of \eqref{eq:14} reads
\begin{multline}\label{eq:20}
  (\nabla^{p_1}z^{A_1})
  \cdots
  (\nabla^{p_m}z^{A_m})
  \nabla_i [ \widetilde{N}_{m;A_1 \cdots A_m}(z) ] \\
  =
  (\nabla^{p_1}z^{A_1})
  \cdots
  (\nabla^{p_m}z^{A_m})
  (\nabla_i^1 z^{A_{m+1}})
  \widetilde{N}_{m+1;A_1 \cdots A_{m+1}}(z).
\end{multline}
Apparently, $(p_1, \cdots, p_m,1)\in I^{k+1}_{m+1}$, which reveals that the second term
of the right hand side of \eqref{eq:14} is a summation term of \eqref{dNa} with
$k+1$.

We have demonstrated that
\begin{equation*}
  \nabla_i [
  (\nabla^{p_1}z^{A_1})
  \cdots
  (\nabla^{p_m}z^{A_m})
  \widetilde{N}_{m;A_1 \cdots A_m}(z) ]
\end{equation*}
is a summation term of \eqref{dNa} with $k+1$.
Thus, \eqref{dNa} holds for $k+1$.

\textup{(2)}
To complete the proof of this lemma,
it remains to prove that
the derivatives of $N_a$ can be taken under the integral sign.
Observing that the integrals appearing in \eqref{dNa} have the same form as
$N_a$, we only need to verify it for the first order derivatives of $N_a$.

Given any $z_0\in\pd\Omega$, write $\nu_0=\nu_{\ssOmega}(z_0)$.
Since $\rho\in C^1(\uS)$, there exists a small positive number $\delta$ such that
\begin{equation*}
  |\nu_{\ssOmega}(z)-\nu_0|<\frac{1}{2},
  \quad \forall |z-z_0|<\delta.
\end{equation*}
Let $\Gamma=\pd\Omega\cap B_{\delta}(z_0)$.
For any $\epsilon\in(0,1)$, consider
\begin{equation*}
  N_{\epsilon}(z)
  :=
  \int_{\Omega} \omega\Bigl( \frac{y-\epsilon\nu_0-z}{|y-\epsilon\nu_0-z|} \Bigr)
  \frac{\phi(y)}{|y-\epsilon\nu_0-z|^{n-a}} \dd y
  =N_a(\epsilon\nu_0+z),
  \quad
  \forall z\in\Gamma.
\end{equation*}
For any $y\in\Omega$ and $z\in\Gamma$, there is
\begin{equation*}
  \begin{split}
    |y-\epsilon\nu_0-z|
    &\geq
    |y-z-\epsilon\nu_{\ssOmega}(z)|
    - |\epsilon\nu_{\ssOmega}(z)-\epsilon\nu_0| \\
    &\geq
    \epsilon-\epsilon|\nu_{\ssOmega}(z)-\nu_0| \\
    &\geq
    \frac{\epsilon}{2}.
  \end{split}
\end{equation*}
Thus, one can apply the Leibniz integral rule to $N_{\epsilon}$.
Just like obtaining \eqref{eq:13}, we have
\begin{equation*}
  \nabla_i N_{\epsilon}(z)
  = z^{A_1}_i
  \int_{\Omega}
  [(n-a) \omega  \varsigma^{A_1}
  - \pd_{\tilde{\alpha}} \varsigma^{A_1} \pd_{\tilde{\alpha}} \omega]
  \Bigl(
  \frac{y-\epsilon\nu_0-z}{|y-\epsilon\nu_0-z|}
  \Bigr)
  \cdot
  |y-\epsilon\nu_0-z|^{a-1-n}
  \phi(y) \dd y.
\end{equation*}
By Lemma \ref{lem601}, the integrals appeared in the above expressions of $N_{\epsilon}$
and $\nabla_iN_{\epsilon}$ are continuous on $B_{\DIAM(\Omega)+1}(0)$, and
therefore uniformly continuous.
Hence, $N_{\epsilon}(z)\rightrightarrows N_a(z)$ on $\Gamma$ as $\epsilon\to0^+$, and
\begin{equation*}
  \nabla_i N_{\epsilon}(z)
  \rightrightarrows
  z^{A_1}_i
  \int_{\Omega}
  [(n-a) \omega  \varsigma^{A_1}
  - \pd_{\tilde{\alpha}} \varsigma^{A_1} \pd_{\tilde{\alpha}}\omega]
  \Bigl(
  \frac{y-z}{|y-z|}
  \Bigr)
  \cdot
  |y-z|^{a-1-n}
  \phi(y) \dd y,
\end{equation*}
illustrating that $N_a\in C^1(\Gamma)$, and
\begin{equation*}
  \nabla_i N_{a}(z)
  =
  z^{A_1}_i
  \int_{\Omega}
  [(n-a) \omega  \varsigma^{A_1}
  - \pd_{\tilde{\alpha}} \varsigma^{A_1} \pd_{\tilde{\alpha}} \omega]
  \Bigl(
  \frac{y-z}{|y-z|}
  \Bigr)
  \cdot
  |y-z|^{a-1-n}
  \phi(y) \dd y.
\end{equation*}
Recalling \eqref{eq:13} again, we see that the first order derivatives of $N_a$ can
indeed be taken under the integral sign.
The proof of this lemma is completed.
\end{proof}

\begin{remark}
In Lemma \ref{lemNak}, when $\omega\equiv1$,
$P_{m+1}(\pd^{\leq m}\omega, \pd^{\leq m}\varsigma^A)$ in \eqref{eq:8} will be simplified into
$P_m(\varsigma^A)$.
One can easily see this fact by carefully checking the proof.
For the case when $\omega\equiv1$ and $\phi\equiv1$, clear expressions for
$P_m(\varsigma^A)$ will be worked out for $k=2$ in Lemma \ref{lem007}.
\end{remark}

We remark that if $k\geq a$, the integral $\widetilde{N}_{k;A_1 \cdots A_k}(z)$
in \eqref{eq:8} of Lemma \ref{lemNak} will diverge for every $z\in\pd\Omega$,
which means that the differentiation formula \eqref{dNa} is invalid for $k\geq a$,
suggesting that $N_a$ is at most $C^{\lceil a \rceil-1 }$ smooth,
regardless of how smooth the boundary $\pd\Omega$ is.

However, the smoothness of $N_a$ can be higher than $\lceil a \rceil-1 $. To get it,
we need to derive another expression of $N_a$ as follows.

\begin{lemma}\label{lem012}
For $a>0$,
$N_a$ has the following expression:
\begin{equation}\label{NaB}
  N_{a}(z)
  =\int_{\pd\Omega} \omega\Bigl( \frac{y-z}{|y-z|} \Bigr)
  \frac{\langle y-z,\nu_y \rangle}{|y-z|^{n-a}}
  \phi_a(y,z)
  \dd y,
\end{equation}
where $\nu_y$ denotes the unit outer normal of $\pd\Omega$ at the boundary point
$y$, and $\phi_a$ is given by
\begin{equation}\label{phi-tilde}
  \phi_a(y,z)
  =
  \int_0^1 \phi(ty+(1-t)z) t^{a-1} \dd t.
\end{equation}
\end{lemma}

\begin{proof}{}
Recalling the expression \eqref{eq:90} with $\sigma_z(u)=0$ for $z\in\pd\Omega$,
and utilizing the variable substitution $r=t\rho_z(u)$, we have
\begin{equation*}
  \begin{split}
    \int_0^{\rho_z(u)} \phi(z+ru) r^{a-1} \dd r
    =
    \rho_z(u)^{a} \int_0^1 \phi(z+t\rho_z(u)u) t^{a-1} \dd t.
  \end{split}
\end{equation*}
Then
\begin{equation} \label{eq:92}
  N_{a}(z)
  =
  \int_{|u|=1 \atop \rho_z(u)>0}
  \omega(u)
  \rho_z(u)^{a}
  \dd u
  \int_0^1 \phi(z+t\rho_z(u)u) t^{a-1} \dd t.
\end{equation}

We continue to use another variable substitution
$y=z+\rho_z(u)u\in\pd\Omega$.
Since
\begin{equation*}
  \dd u =\rho_z(u)^{1-n}\langle u,\nu_y \rangle \dd y,
\end{equation*}
we have
\begin{equation*}
  N_{a}(z)
  =
  \int_{\Gamma}
  \omega(u)
  \rho_z(u)^{a+1-n}
  \langle u,\nu_y \rangle \dd y
  \int_0^1 \phi(z+t\rho_z(u)u) t^{a-1} \dd t,
\end{equation*}
where $\Gamma\subset\pd\Omega$ is the image set of $y$
on the set $\set{u\in\uS : \rho_z(u)>0}$.
Note that
\begin{align*}
  \rho_z(u)&=|y-z|, \\
  u&=\frac{y-z}{|y-z|}, \\
  z+t\rho_z(u)u&=ty+(1-t)z,
\end{align*}
the above $N_a$ tells
\begin{equation}\label{eq:93}
  N_{a}(z)
  =
  \int_{\Gamma}
  \omega\Bigl( \frac{y-z}{|y-z|} \Bigr)
  |y-z|^{a-n}
  \langle y-z,\nu_y \rangle \dd y
  \int_0^1 \phi(ty+(1-t)z) t^{a-1} \dd t.
\end{equation}

For any fixed $y\in\pd\Omega\backslash\Gamma$ and $y\neq z$, let
\begin{equation*}
  \tilde{u}=\frac{y-z}{|y-z|}
  \quad \text{ and } \quad
  \tilde{y}=z+\rho_z(\tilde{u})\tilde{u}.
\end{equation*}
Observe that $z+|y-z|\tilde{u}=y\in\pd\Omega$,
the definition of $\rho_{z}$ gives  $\rho_z(\tilde{u})\geq |y-z|>0$.
Furthermore, from the definition of $\Gamma$, we conclude  $\tilde{y}\in\Gamma$.
Therefore,
\begin{equation*}
  \rho_z(\tilde{u})>|y-z|.
\end{equation*}
Since $\langle \tilde{y}-y,\nu_y \rangle\leq0$, and
\begin{equation*}
  \tilde{y}-y
  =\left( \rho_z(\tilde{u})-|y-z| \right) \tilde{u},
\end{equation*}
we thus obtain  $\langle \tilde{u},\nu_y \rangle\leq0$,
namely $\langle y-z,\nu_y \rangle\leq0$.
However, since  $\nu_y$ is the outer normal, we have
$\langle z-y,\nu_y \rangle\leq0$. Thus there holds
\begin{equation}\label{eq:94}
  \langle y-z,\nu_y \rangle=0.
\end{equation}
Using \eqref{eq:94}, it yields
\begin{equation*}
  \int_{\pd\Omega\backslash\Gamma}
  \omega\Bigl( \frac{y-z}{|y-z|} \Bigr)
  |y-z|^{a-n}
  \langle y-z,\nu_y \rangle \dd y
  \int_0^1 \phi(ty+(1-t)z) t^{a-1} \dd t
  =0.
\end{equation*}
This integral together with \eqref{eq:93} illustrates
\begin{equation*}
  N_{a}(z)
  =
  \int_{\pd\Omega}
  \omega\Bigl( \frac{y-z}{|y-z|} \Bigr)
  |y-z|^{a-n}
  \langle y-z,\nu_y \rangle \dd y
  \int_0^1 \phi(ty+(1-t)z) t^{a-1} \dd t,
\end{equation*}
which is just the expression \eqref{NaB}.
\end{proof}

Utilizing the new representation \eqref{NaB} in Lemma \ref{lem012}, we can
obtain a new regularity result, which is independent of $a$.

\begin{lemma}\label{lem1205}
Let $a>0$, $k\geq0$, $0\leq\alpha\leq1$,
and $\rho\in C^{k+1,\alpha}(\uS)$.
Assume $\varphi\in C^{k,\alpha}(\uS\times\pd\Omega\times\pd\Omega)$ satisfies
\begin{equation}\label{eq:110}
  |\varphi(\varsigma,y,z)-\varphi(\tilde{\varsigma},y,z)|
  \leq c_{\varphi}|\varsigma-\tilde{\varsigma}|,
  \quad
  \forall\, \varsigma,\tilde{\varsigma}\in\uS, \, y,z\in\pd\Omega
\end{equation}
for a positive constant $c_{\varphi}$.
Then
\begin{equation}\label{Ta}
  \mathcal{T}_a(z)
  :=\int_{\pd\Omega}
  |y-z|^{1-n+a}
  \varphi\bigl( \tfrac{y-z}{|y-z|},y,z \bigr)
  \dd y
\end{equation}
is $C^{k,\alpha}(\pd\Omega)$, and
\begin{equation}\label{eq:127}
  \norm{\mathcal{T}_a}_{k,\alpha} \leq
  C(n, a, k, c_{\varphi}, \norm{\varphi}_{k,\alpha}, \norm{\rho}_{k+1,\alpha}, 1/{\rho_{\min}})
\end{equation}
for some positive constants $C$ depending only on these quantities listed above.
\end{lemma}

\begin{proof}{}
To avoid potential integral divergence in the differentiation formula from
Lemma \ref{lemNak}, we need to develop another differentiation method.
The key is to find a good separation representation for the singular term $y-z$.

Using the radial map \eqref{radialMap}, one can write
\begin{equation} \label{tau}
  y-z
  = \vec{r}(u)-\vec{r}(\bar{z})
  = \rho(u)u-\rho(\bar{z})\bar{z}
  =:\tau,
\end{equation}
where $\bar{z}=\frac{z}{|z|}\in\uS$.
Let $\theta\in[0,\pi]$ be the angle between vectors $y$ and $z$, namely
\begin{equation}\label{eq:119}
  \theta=\arccos\langle u,\bar{z} \rangle.
\end{equation}
Note that $|u-\bar{z}|=\sqrt{2-2\cos\theta}=2\sin\frac{\theta}{2}$, there is
\begin{equation}\label{eq:118}
  \frac{2}{\pi}\theta\leq|u-\bar{z}|\leq\theta.
\end{equation}
Since $\rho\in C^1(\uS)$, we have
\begin{equation}\label{eq:154}
  \begin{split}
    |\tau|
    &\leq
    |\rho(u)u-\rho(u)\bar{z}|
    +
    |\rho(u)-\rho(\bar{z})|\cdot|\bar{z}|
    \\
    &\leq
    \rho(u)\theta+\theta\max|\nabla\rho|
    \\
    &\leq
    \norm{\rho}_1\theta.
  \end{split}
\end{equation}
On the other hand, one sees
\begin{equation*}
  \begin{split}
    |(\rho(u)+\rho(\bar{z}))(u-\bar{z})|
    &\leq
    |\rho(u)u-\rho(\bar{z})\bar{z}|
    +
    |\rho(\bar{z})u-\rho(u)\bar{z}|
    \\
    &=
    2|\rho(u)u-\rho(\bar{z})\bar{z}|,
  \end{split}
\end{equation*}
namely
\begin{equation} \label{eq:121}
  |\rho(u)u-\rho(\bar{z})\bar{z}|
  \geq
  \frac{\rho(u)+\rho(\bar{z})}{2}|u-\bar{z}|,
\end{equation}
which, in combination with \eqref{eq:118}, implies
\begin{equation*}
  |\tau|\geq \frac{2\rho_{\min}}{\pi}\theta.
\end{equation*}
We therefore derive the following key observation:
\begin{equation} \label{eq:123}
  \frac{2\rho_{\min}}{\pi}\theta
  \leq
  |\tau|
  \leq
  \norm{\rho}_1\theta.
\end{equation}

Now we can construct the desired separation representation of $\tau$.
For each $u$ that is not $\bar{z}$ or $-\bar{z}$, namely $\theta\in(0,\pi)$,
there exists a unique vector $v\in\uS$ such that $v\bot\bar{z}$, and
\begin{equation} \label{uvtheta}
  u= \bar{z} \cos\theta +v\sin\theta.
\end{equation}
Define a path on $\pd\Omega$ connecting $\vec{r}(\bar{z})$ and $\vec{r}(u)$ as
\begin{equation}\label{eq:3}
  P(t)=\vec{r}( \bar{z} \cos t\theta+v\sin t\theta),
  \quad t\in[0,1].
\end{equation}
Then
\begin{equation*}
  \begin{split}
    P'(t)
    &=\nabla\vec{r}( \bar{z} \cos t\theta+v\sin t\theta) ( \bar{z} \cos t\theta+v\sin t\theta)' \\
    &=\nabla\vec{r}( \bar{z} \cos t\theta+v\sin t\theta) (v\cos t\theta- \bar{z} \sin t\theta)\theta,
  \end{split}
\end{equation*}
which leads to
\begin{equation} \label{xi}
  \begin{split}
    \tau =\vec{r}(u)-\vec{r}( \bar{z} )
    &=P(1)-P(0) =\int_0^1 P'(t) \dd t \\
    &= \theta
    \int_0^1 \nabla\vec{r}( \bar{z} \cos t\theta+v\sin t\theta) (v\cos t\theta- \bar{z} \sin t\theta) \dd t \\
    &=: \theta\,
    \xi( \bar{z} ,v,\theta).
  \end{split}
\end{equation}
By the definition of $\vec{r}$ in \eqref{radialMap}, there is
\begin{equation}\label{eq:111}
  \nabla\vec{r}(u)=\rho(u)I_n+u\nabla\rho(u),
  \quad \forall\, u\in\uS.
\end{equation}
Then
\begin{equation*}
  \begin{split}
    \nabla\vec{r}(\bar{z} \cos t\theta+v\sin t\theta)
    &= \rho(\bar{z} \cos t\theta+v\sin t\theta) I_n \\
    &\hskip1.12em
    +(\bar{z} \cos t\theta+v\sin t\theta)
    \nabla\rho(\bar{z} \cos t\theta+v\sin t\theta),
  \end{split}
\end{equation*}
which, together with $\rho\in C^{k+1,\alpha}(\uS)$, implies that $\xi$ defined in
\eqref{xi} is $C^{k,\alpha}$ with respect to $(\bar{z},v,\theta)$.
Moreover, recalling inequality \eqref{eq:123}, it follows that
\begin{equation} \label{eq:124}
  0<
  \frac{2\rho_{\min}}{\pi}
  \leq
  |\xi(\bar{z},v,\theta)|
  \leq
  \norm{\rho}_1,
  \quad
  \forall\, \bar{z},v\in\uS, \ v\bot\bar{z}, \ \theta\in[0,\pi].
\end{equation}

Building upon the above separation form \eqref{xi} and applying some variable
substitutions, we will derive a new representation of $\mathcal{T}_a(z)$, which is
very convenient to take derivatives.
In fact, corresponding to \eqref{tau}, we first use the variable substitution
$y=\vec{r}(u)=\rho(u)u$. Then,
\begin{equation*}
  \mathcal{T}_a(\bar{z})
  :=
  \mathcal{T}_a(z)
  =\int_{\uS}
  |\tau|^{1-n+a}
  \varphi\bigl( \tfrac{\tau}{|\tau|},\vec{r}(u),\vec{r}(\bar{z}) \bigr)
  \rho(u)^{n-2} |\delbar\rho(u)| \dd u.
\end{equation*}
Now, for each fixed $\bar{z}$, \eqref{uvtheta} provides the polar coordinates
$(v,\theta)$ for almost any $u\in\uS$ except $\pm\bar{z}$.
By the following integral formula
\begin{equation*}
  \int_{\uS}\dd u =
  \int_0^{\pi} \sin^{n-2}\theta \dd\theta
  \int_{v\bot \bar{z} } \dd v,
\end{equation*}
and \eqref{xi}, it suffices to gain
\begin{equation}\label{eq:126}
  \mathcal{T}_a(\bar{z}) =
  \int_0^{\pi}
  \theta^{1-n+a}
  \sin^{n-2}\theta
  \dd\theta
  \int_{v\bot \bar{z}}
  g(\bar{z},v,\theta)
  \dd v,
\end{equation}
where
\begin{equation}\label{eq:125}
  g(\bar{z},v,\theta)=
  |\xi|^{1-n+a}
  \varphi\bigl( \tfrac{\xi}{|\xi|},\vec{r}(u),\vec{r}(\bar{z}) \bigr)
  \rho(u)^{n-2} |\delbar\rho(u)|.
\end{equation}
Under the assumptions on $\varphi$ and by estimate \eqref{eq:124}, it follows
that $g$ is $C^{k,\alpha}$ with respect to $(\bar{z},v)$ and their norms are
uniformly bounded for $\theta\in(0,\pi)$.
Hence, from \eqref{eq:126} with $a>0$, we conclude that
$\mathcal{T}_a(\bar{z})\in C^{k,\alpha}$ and its norm satisfies the estimate
\eqref{eq:127}. The proof of this lemma is now completed.
\end{proof}

Applying Lemma \ref{lem1205} to $N_a$ provided in Lemma \ref{lem012}, we
have the following lemma.

\begin{lemma}\label{lem1207}
Let $a>0$, $k\geq0$, $0\leq\alpha\leq1$.
Assume $\rho\in C^{k+1,\alpha}(\uS)$,
$\omega\in C^{k,\alpha}(\uS)\cap C^{0,1}(\uS)$, and
$\phi\in C^{k,\alpha}(\Omega)$.
Then $N_a$ is $C^{k,\alpha}(\pd\Omega)$, and
\begin{equation}\label{eq:128}
  \norm{N_a}_{k,\alpha} \leq
  C(n, a, k, \norm{\omega}_{k,\alpha}, \norm{\omega}_{0,1}, \norm{\phi}_{k,\alpha},
  \norm{\rho}_{k+1,\alpha}, 1/{\rho_{\min}})
\end{equation}
for some positive constants $C$ depending only on these quantities listed above.
\end{lemma}

\begin{proof}{}
Let $\varphi(\varsigma,y,z)=\langle \varsigma,\nu_y \rangle \omega(\varsigma) \phi_a(y,z)$, then
$N_a(z)$ in \eqref{NaB} is just $\mathcal{T}_a(z)$ in \eqref{Ta}.
By the assumptions of this lemma, $\varphi$ satisfies the requirements of Lemma
\ref{lem1205}.
Now the conclusions of Lemma \ref{lem1205} yield this lemma.
\end{proof}

At this moment, we have established two different regularity results about $N_a$,
namely Lemmas \ref{lemNak} and \ref{lem1207}.
But neither of them is optimal.
Fortunately, by combining these two results, we can gain the following optimal
regularity for $N_a$ when $a>1$.

\begin{lemma}\label{lemNa>1}
Let $a>1$, $k\geq0$, $0\leq\alpha\leq1$.
Assume $\rho\in C^{k+1,\alpha}(\uS)$,
$\omega\in C^{k+1,\alpha}(\uS)\cap C^{1,1}(\uS)$, and
$\phi\in C^{k,\alpha}(\Omega)$.
Then $N_a$ is $C^{k+1,\alpha}(\pd\Omega)$, and
\begin{equation}\label{eq:17}
  \norm{N_a}_{k+1,\alpha} \leq
  C(n, a, k, \norm{\omega}_{k+1,\alpha}, \norm{\omega}_{1,1}, \norm{\phi}_{k,\alpha},
  \norm{\rho}_{k+1,\alpha}, 1/{\rho_{\min}})
\end{equation}
for some positive constants $C$ depending only on these quantities listed above.
\end{lemma}

\begin{proof}{}
Since $a>1$, by Lemma \ref{lemNak}, $N_a\in C^1(\pd\Omega)$, and \eqref{eq:13}, there holds
\begin{equation*}
  \nabla N_{a}(z)
  = \nabla z^{A_1}
  \int_{\Omega}
  [(n-a) \omega  \varsigma^{A_1}
  - \pd_{\tilde{\alpha}} \varsigma^{A_1} \pd_{\tilde{\alpha}} \omega]
  \Bigl( \frac{y-z}{|y-z|} \Bigr)
  \frac{\phi(y)}{|y-z|^{n-(a-1)}} \dd y.
\end{equation*}
Note that the above integral has the same form as $N_a$, by the assumptions of this
lemma, we can apply Lemma \ref{lem1207} to it.
Therefore, we have $\nabla N_a\in C^{k,\alpha}(\pd\Omega)$, and
\begin{equation}\label{eq:120}
  \norm{\nabla N_a}_{k,\alpha} \leq
  C(n, a, k, \norm{\omega}_{k+1,\alpha}, \norm{\omega}_{1,1}, \norm{\phi}_{k,\alpha},
  \norm{\rho}_{k+1,\alpha}, 1/{\rho_{\min}}).
\end{equation}
Hence, $N_a\in C^{k+1,\alpha}(\pd\Omega)$, and then
combining \eqref{eq:120} and Lemma \ref{lem1128}, we obtain the desired estimate \eqref{eq:17}.
\end{proof}

\begin{remark}\label{rmk1209}
Utilizing arguments analogous to those in Lemmas \ref{lemNak} and \ref{lemNa>1}, one
can easily obtain the optimal regularity for $\mathcal{T}_a$ with $a>1$ defined in Lemma
\ref{lem1205}. That is,
if $a>1$, $k\geq0$, $0\leq\alpha\leq1$,
$\rho\in C^{k+1,\alpha}(\uS)$, and
$\varphi\in C^{k+1,\alpha}\cap C^{1,1}(\uS\times\pd\Omega\times\pd\Omega)$,
then $\mathcal{T}_a$ is $C^{k+1,\alpha}(\pd\Omega)$, and the corresponding norm satisfies
\begin{equation}\label{eq:122}
  \norm{\mathcal{T}_a}_{k+1,\alpha} \leq
  C(n, a, k, \norm{\varphi}_{k+1,\alpha}, \norm{\varphi}_{1,1},
  \norm{\rho}_{k+1,\alpha}, 1/{\rho_{\min}})
\end{equation}
for some positive constants $C$ depending only on these quantities listed above.
\end{remark}

On the other hand, we can also derive the higher regularity results of $N_a$ with respect to the boundary of domain for $0<a\leq 1$.

\subsection{Higher regularity of $N_a$ for $0<a\leq1$}
To obtain that, we first require the following lemma.

\begin{lemma} \label{lem1211}
Let $F$ be a continuous function defined in the open set
\begin{equation*}
  \Pi=\set{(u,v)\in\uS\times\uS : u\neq v},
\end{equation*}
whose partial derivative $\nabla_{v_i}F$ is also continuous.
Assume there exist positive constants $C_1$, $C_2$, $\beta_1$, $\beta_2$,
such that for any $(u,v)\in\Pi$,
\begin{align}
  |F(u,v)|&\leq C_1|u-v|^{2-n+\beta_1},
            \label{eq:141} \\
  |\nabla_{v_i}F(u,v)|&\leq C_2|u-v|^{1-n+\beta_2}.
                        \label{eq:142}
\end{align}
Then, the function $g:\uS\to\R$ given by
\begin{equation} \label{eq:151}
  g(v)=\int_{\uS} F(u,v) \dd u
\end{equation}
has a continuous partial derivative $\nabla_{v_i}g$.
Moreover, there holds
\begin{equation} \label{eq:143}
  \nabla_{v_i}g(v)=\int_{\uS} \nabla_{v_i}F(u,v) \dd u.
\end{equation}
\end{lemma}

\begin{proof}{}
Let $F_m\in C^1(\Pi)$ be a sequence of functions satisfying
\begin{equation} \label{eq:144}
  F_m\rightrightarrows F
  \ \text{ and } \
  \nabla_{v_i}F_m\rightrightarrows \nabla_{v_i}F
\end{equation}
uniformly on any compact subset of $\Pi$ as $m\to+\infty$.
For each $\epsilon\in(0,1)$, define $g_{\epsilon}^m$ as
\begin{equation}\label{eq:15}
  g_{\epsilon}^m(v)=\int_{\arccos\langle u,v \rangle\geq\epsilon} F_m(u,v) \dd u,
  \quad \forall v\in\uS.
\end{equation}

Given any $v\in\uS$, there exists an orthogonal matrix $Q(v)$ such that
\begin{equation}\label{eq:131}
  Q(v)^Tv=e_n=(0,\cdots,0,1)^T\in\uS.
\end{equation}
For any fixed $v_0\in\uS$, we can find a small open neighborhood $O(v_0)\subset\uS$, such that
$Q(v)$ is smooth with respect to $v\in O(v_0)$.
Now, for each $v\in O(v_0)$, by applying the variable substitution $u=Q(v)\eta$
to the integral in \eqref{eq:15}, we obtain
\begin{equation*}
  g_{\epsilon}^m(v)=\int_{\arccos\langle \eta,e_n \rangle\geq\epsilon} F_m(Q(v)\eta,v) \dd\eta.
\end{equation*}
Since $F_m\in C^1(\Pi)$, using the Leibniz integral rule, we have
\begin{equation}\label{eq:129}
  \begin{split}
    \nabla_{v_i} g_{\epsilon}^m(v)
    &=\int_{\arccos\langle \eta,e_n \rangle\geq\epsilon}
    \nabla_{v_i} [F_m(Q(v)\eta,v)]
    \dd\eta \\
    &=
    \int_{\arccos\langle \eta,e_n \rangle\geq\epsilon}
    [
    \langle
    (\nabla_uF_m)(Q(v)\eta,v),\nabla_{v_i}Q(v)\eta
    \rangle
    +
    (\nabla_{v_i}F_m)(Q(v)\eta,v)
    ]
    \dd\eta \\
    &=
    \int_{\arccos\langle u,v \rangle\geq\epsilon}
    [
    \langle
    \nabla_uF_m(u,v),\nabla_{v_i}Q(v)Q(v)^Tu
    \rangle
    +
    \nabla_{v_i}F_m(u,v)
    ]
    \dd u,
  \end{split}
\end{equation}
where the variable substitution $\eta=Q(v)^Tu$ is used to obtain the last equality.

For the first integral term, by the divergence theorem, we find
\begin{multline}\label{eq:130}
  \int_{\arccos\langle u,v \rangle\geq\epsilon}
  \langle
  \nabla_uF_m(u,v),\nabla_{v_i}Q(v)Q(v)^Tu
  \rangle
  \dd u \\
  = \int_{\arccos\langle u,v \rangle=\epsilon}
  F_m(u,v)
  \langle
  \nabla_{v_i}Q(v)Q(v)^Tu,\vec{n}
  \rangle
  \dd \sigma(u) \\
  -\int_{\arccos\langle u,v \rangle\geq\epsilon}
  F_m\DIV_u(\nabla_{v_i}Q(v)Q(v)^Tu)
  \dd u,
\end{multline}
where $\vec{n}$ is the unit outer normal of the boundary of
$\set{u\in\uS : \arccos\langle u,v \rangle\geq\epsilon}$,
namely,
\begin{equation}\label{eq:132}
  \vec{n}=
  \frac{v-\langle v,u \rangle u}{|v-\langle v,u \rangle u|}.
\end{equation}
Recalling $Q(v)$ is an orthogonal matrix, namely $Q(v)Q(v)^T=I_n$, we obtain
\begin{equation*}
  \nabla_{v_i} Q(v) Q(v)^T
  = -Q(v) \nabla_{v_i} Q(v)^T
  = -\left(\nabla_{v_i} Q(v) Q(v)^T \right)^T.
\end{equation*}
Therefore,
\begin{equation}\label{eq:133}
  u^T \nabla_{v_i} Q(v) Q(v)^T u =0
  \quad \text{ and } \quad
  \TR \left(\nabla_{v_i} Q(v) Q(v)^T \right) =0,
\end{equation}
which give
\begin{equation*}
  \DIV_u(\nabla_{v_i}Q(v)Q(v)^Tu)=0.
\end{equation*}
Inserting it into \eqref{eq:130}, we derive
\begin{multline}\label{eq:145}
  \int_{\arccos\langle u,v \rangle\geq\epsilon}
  \langle
  \nabla_uF_m(u,v),\nabla_{v_i}Q(v)Q(v)^Tu
  \rangle
  \dd u \\
  = \int_{\arccos\langle u,v \rangle=\epsilon}
  F_m(u,v)
  \langle
  \nabla_{v_i}Q(v)Q(v)^Tu,\vec{n}
  \rangle
  \dd \sigma(u).
\end{multline}
Recalling \eqref{eq:131}, we have
$v^T\nabla_{v_i}Q(v)=-v_i^TQ(v)$,
which together with \eqref{eq:132} and \eqref{eq:133}, yields that
\begin{equation*}
  \begin{split}
    |v-\langle v,u \rangle u|
    \cdot
    \langle \nabla_{v_i}Q(v)Q(v)^Tu,\vec{n} \rangle
    &=
    \left(v^T-\langle v,u \rangle u^T\right)
    \nabla_{v_i}Q(v)Q(v)^Tu \\
    &= v^T \nabla_{v_i}Q(v)Q(v)^Tu \\
    &= -v_i^Tu \\
    &= \langle -v_i, u-\langle v,u \rangle v \rangle,
  \end{split}
\end{equation*}
namely,
\begin{equation}\label{eq:134}
  \langle \nabla_{v_i}Q(v)Q(v)^Tu,\vec{n} \rangle
  =
  \Bigl\langle
  -v_i, \frac{u-\langle v,u \rangle v}{|u-\langle v,u \rangle v|}
  \Bigr\rangle.
\end{equation}
Therefore, the equality \eqref{eq:145} reads
\begin{multline}\label{eq:146}
  \int_{\arccos\langle u,v \rangle\geq\epsilon}
  \langle
  \nabla_uF_m(u,v),\nabla_{v_i}Q(v)Q(v)^Tu
  \rangle
  \dd u \\
  =
  \int_{\arccos\langle u,v \rangle=\epsilon}
  F_m(u,v)
  \Bigl\langle
  -v_i, \frac{u-\langle v,u \rangle v}{|u-\langle v,u \rangle v|}
  \Bigr\rangle
  \dd \sigma(u).
\end{multline}
Now, by \eqref{eq:129}, we derive
\begin{multline}\label{eq:148}
  \nabla_{v_i} g_{\epsilon}^m(v)
  =
  \int_{\arccos\langle u,v \rangle=\epsilon}
  F_m(u,v)
  \Bigl\langle
  -v_i, \frac{u-\langle v,u \rangle v}{|u-\langle v,u \rangle v|}
  \Bigr\rangle
  \dd \sigma(u)
  \\
  +\int_{\arccos\langle u,v \rangle\geq\epsilon}
  \nabla_{v_i}F_m(u,v)
  \dd u.
\end{multline}
Note that this differentiation formula holds for any $v\in\uS$.

For each fixed $\epsilon$, as $m\to+\infty$,  from \eqref{eq:144} and
\eqref{eq:15}, we see that $g_{\epsilon}^m$ converges uniformly to $g_{\epsilon}$, which
is given by
\begin{equation}\label{eq:149}
  g_{\epsilon}(v)=\int_{\arccos\langle u,v \rangle\geq\epsilon} F(u,v) \dd u,
  \quad \forall v\in\uS.
\end{equation}
By \eqref{eq:144} and \eqref{eq:148}, we also see that
$\nabla_{v_i}g_{\epsilon}^m$ converges uniformly on $\uS$.
Thus, $\nabla_{v_i}g_{\epsilon}$ exists and satisfies
\begin{multline}\label{eq:150}
  \nabla_{v_i} g_{\epsilon}(v)
  =
  \int_{\arccos\langle u,v \rangle=\epsilon}
  F(u,v)
  \Bigl\langle
  -v_i, \frac{u-\langle v,u \rangle v}{|u-\langle v,u \rangle v|}
  \Bigr\rangle
  \dd \sigma(u)
  \\
  +\int_{\arccos\langle u,v \rangle\geq\epsilon}
  \nabla_{v_i}F(u,v)
  \dd u.
\end{multline}

By the assumption \eqref{eq:141}, one sees
\begin{equation*}
  \begin{split}
    \int_{\arccos\langle u,v \rangle<\epsilon} |F(u,v)| \dd u
    &\leq C_1 \int_{\arccos\langle u,v \rangle<\epsilon} |u-v|^{2-n+\beta_1} \dd u \\
    &= C_1 \int_0^{\epsilon} \sin^{n-2}\theta \dd\theta
    \int_{\zeta\bot v} \left( 2\sin\frac{\theta}{2} \right)^{2-n+\beta_1} \dd\zeta \\
    &\leq C_1 |\mathbb{S}^{n-2}| \int_0^{\epsilon} \theta^{\beta_1} \dd\theta \\
    &= \frac{C_1}{1+\beta_1} |\mathbb{S}^{n-2}| \epsilon^{1+\beta_1},
  \end{split}
\end{equation*}
tending to zero uniformly for $v\in\uS$ as $\epsilon\to+0$.
Therefore, $g_{\epsilon}$ in \eqref{eq:149} converges uniformly to $g$ given in
\eqref{eq:151}.
From assumptions \eqref{eq:141} and \eqref{eq:142}, we have
\begin{equation*}
  \begin{split}
    \int_{\arccos\langle u,v \rangle=\epsilon}
    \Bigl| F(u,v) \Bigl\langle
    -v_i, & \frac{u-\langle v,u \rangle v}{|u-\langle v,u \rangle v|}
    \Bigr\rangle \Bigr|
    \dd \sigma(u) \\
    &\leq C_1 |v_i| \int_{\arccos\langle u,v \rangle=\epsilon} |u-v|^{2-n+\beta_1} \dd \sigma(u) \\
    &= C_1 |v_i| \left( 2\sin\frac{\epsilon}{2} \right)^{2-n+\beta_1}
    |\mathbb{S}^{n-2}| \sin^{n-2}\epsilon \\
    &\leq |\mathbb{S}^{n-2}| C_1 |v_i| \epsilon^{\beta_1},
  \end{split}
\end{equation*}
and
\begin{equation*}
  \begin{split}
    \int_{\arccos\langle u,v \rangle<\epsilon} |\nabla_{v_i}F(u,v)| \dd u
    &\leq C_2 \int_{\arccos\langle u,v \rangle<\epsilon} |u-v|^{1-n+\beta_2} \dd u \\
    &= C_2 \int_0^{\epsilon} \sin^{n-2}\theta \dd\theta
    \int_{\zeta\bot v} \left( 2\sin\frac{\theta}{2} \right)^{1-n+\beta_2} \dd\zeta \\
    &\leq \frac{\pi}{2} C_2 |\mathbb{S}^{n-2}| \int_0^{\epsilon} \theta^{\beta_2-1} \dd\theta \\
    &= \frac{\pi C_2}{2\beta_2} |\mathbb{S}^{n-2}| \epsilon^{\beta_2},
  \end{split}
\end{equation*}
illustrating that $\nabla_{v_i}g_{\epsilon}$ in \eqref{eq:150} converges uniformly on $\uS$ as $\epsilon\to+0$.
Thus, $\nabla_{v_i}g$ exists and satisfies
\begin{equation*}
  \nabla_{v_i} g(v)
  = \int_{\uS} \nabla_{v_i}F(u,v) \dd u,
\end{equation*}
which is just \eqref{eq:143}.
The proof of this lemma is completed.
\end{proof}
Applying Lemma \ref{lem1211}, we can establish the following lemma.
\begin{lemma} \label{lem1226}
Assume $\pd\Omega$ is $C^1$.
Let $G$ be a continuous function defined in the open set
\begin{equation}\label{eq:160}
  \Sigma=\set{(y,z)\in\pd\Omega\times\pd\Omega : y\neq z},
\end{equation}
whose partial derivative $\nabla_{z_i}G$ is also continuous.
Suppose there exist positive constants $C_1$, $C_2$, $\beta_1$, $\beta_2$,
such that for any $(y,z)\in\Sigma$,
\begin{align}
  |G(y,z)|&\leq C_1|y-z|^{2-n+\beta_1},
            \label{eq:152} \\
  |\nabla_{z_i}G(y,z)|&\leq C_2|y-z|^{1-n+\beta_2}.
                        \label{eq:138}
\end{align}
Then, the function $\varphi:\pd\Omega\to\R$ given by
\begin{equation} \label{eq:147}
  \varphi(z)=\int_{\pd\Omega} G(y,z) \dd y
\end{equation}
has a continuous partial derivative $\nabla_{z_i}\varphi$.
Moreover, there holds
\begin{equation}\label{eq:153}
  \nabla_{z_i}\varphi(z)=\int_{\pd\Omega} \nabla_{z_i}G(y,z) \dd y.
\end{equation}
\end{lemma}

\begin{proof}
Using the variable substitution $y=\vec{r}(u)=\rho(u)u$, we have
\begin{equation*}
  \varphi(z)
  =\int_{\uS} G(\vec{r}(u),z)
  \rho(u)^{n-2} |\delbar\rho(u)| \dd u.
\end{equation*}
Define $g(v)=\varphi(\vec{r}(v))$, namely
\begin{equation*}
  g(v)
  = \int_{\uS} G(\vec{r}(u),\vec{r}(v))
  \rho(u)^{n-2} |\delbar\rho(u)| \dd u
  =: \int_{\uS} F(u,v) \dd u.
\end{equation*}
Since $\rho\in C^1$ and $G$ is continuous in $\Sigma$,
$F$ is continuous in $\Pi$.
Note $\max|\delbar\rho|\leq \rho_{\max}^2/\rho_{\min}$, by the assumption
\eqref{eq:152}, one sees
\begin{equation*}
  \begin{split}
    |F(u,v)|
    &\leq
    C_1\rho_{\max}^n\rho_{\min}^{-1}
    |\vec{r}(u)-\vec{r}(v)|^{2-n+\beta_1}
    \\
    &\leq
    C_1\rho_{\max}^n\rho_{\min}^{1-n}
    \left( \frac{\pi}{2}\norm{\rho}_1 \right)^{\beta_1}
    |u-v|^{2-n+\beta_1},
  \end{split}
\end{equation*}
where the last inequality follows from the estimates \eqref{eq:118}, \eqref{eq:154}
and \eqref{eq:121} obtained in the proof of Lemma \ref{lem1205}.

Let $v_i$ be a tangent vector field on $\uS$ such that $\vec{r}_*v_i=z_i$.
Then,
\begin{equation*}
  \begin{split}
    \nabla_{v_i}F(u,v)
    &= \nabla_{v_i}\left[G(\vec{r}(u),\vec{r}(v)) \rho(u)^{n-2} |\delbar\rho(u)|\right] \\
    &= (\nabla_{z_i}G)(\vec{r}(u),\vec{r}(v)) \rho(u)^{n-2} |\delbar\rho(u)|.
  \end{split}
\end{equation*}
Similarly, by the assumptions on $\nabla_{z_i}G$, we have that $\nabla_{v_i}F$ is
continuous in $\Pi$, and
\begin{equation*}
  \begin{split}
    |\nabla_{v_i}F(u,v)|
    &\leq
    C_2\rho_{\max}^n\rho_{\min}^{-1}
    |\vec{r}(u)-\vec{r}(v)|^{1-n+\beta_2}
    \\
    &\leq
    C_2\rho_{\max}^n\rho_{\min}^{-n}
    \left( \frac{\pi}{2}\norm{\rho}_1 \right)^{\beta_2}
    |u-v|^{1-n+\beta_2}.
  \end{split}
\end{equation*}

Observe that $F$ satisfies all assumptions of Lemma \ref{lem1211}, from which we conclude that
\begin{equation*}
  \nabla_{v_i}g(v)=\int_{\uS} \nabla_{v_i}F(u,v) \dd u,
\end{equation*}
this means
\begin{equation*}
  (\nabla_{z_i}\varphi)(\vec{r}(v))
  = \int_{\uS} (\nabla_{z_i}G)(\vec{r}(u),\vec{r}(v)) \rho(u)^{n-2} |\delbar\rho(u)| \dd u
  = \int_{\pd\Omega} (\nabla_{z_i}G)(y,\vec{r}(v)) \dd y,
\end{equation*}
which is just the conclusion \eqref{eq:153} of this lemma.
\end{proof}

Next, we proceed to derive the following lemma.

\begin{lemma}\label{lem1213}
Assume $0\leq\alpha\leq1$, and $\rho\in C^{1,\alpha}(\uS)$.
Then
\begin{equation} \label{eq:156}
  |\langle  y-z,\nu_y \rangle|
  \leq C |y-z|^{1+\alpha},
  \quad
  \forall \,y,z\in\pd\Omega,
\end{equation}
where $C$ is a positive constant depending only on $\norm{\rho}_{1,\alpha}$
and $1/\rho_{\min}$.

\end{lemma}

\begin{proof}{}
Since $\rho\in C^{1,\alpha}(\uS)$,
there exists a positive constant $C_1$ depending only on the norm $\norm{\rho}_{1,\alpha}$,
such that for any $\bar{z},u\in\uS$,
\begin{equation} \label{eq:136}
  |\rho(\bar{z})-\rho(u)-\langle \nabla \rho(u),\zeta \rangle \theta|
  \leq
  C_1 \theta^{1+\alpha},
\end{equation}
where $\theta\in[0,\pi]$ is the angle between vectors $\bar{z}$ and $u$,
and $\zeta$ is a unit vector satisfying $\zeta\bot u$ and
\begin{equation*}
  \bar{z}=u\cos\theta+\zeta\sin\theta.
\end{equation*}
Then, we obtain the following estimate:
\begin{equation}\label{eq:137}
  \begin{split}
    |\rho(\bar{z})-\rho(u)&-\langle \nabla \rho(u),\bar{z}-u \rangle| \\
    &=
    |\rho(\bar{z})-\rho(u)-\langle \nabla \rho(u),\zeta \rangle \sin\theta| \\
    &\leq
    |\rho(\bar{z})-\rho(u)-\langle \nabla \rho(u),\zeta \rangle \theta|
    +
    |\langle \nabla \rho(u),\zeta \rangle (\theta-\sin\theta)| \\
    &\leq
    C_1\theta^{1+\alpha}
    +
    \frac{1}{6}|\nabla\rho(u)|\theta^3 \\
    &\leq
    C_2\theta^{1+\alpha},
  \end{split}
\end{equation}
where $C_2$ is a positive constant depending only on $\norm{\rho}_{1,\alpha}$.
Furthermore, we have
\begin{equation*}
  \begin{split}
    |\rho(\bar{z})-\rho(u)&-\langle \delbar \rho(u),\bar{z}-u \rangle| \\
    &=
    |\rho(\bar{z})-\rho(u)-\langle \nabla \rho(u)-\rho(u)u,\bar{z}-u \rangle| \\
    &\leq
    |\rho(\bar{z})-\rho(u)-\langle \nabla \rho(u),\bar{z}-u \rangle|
    +
    |\rho(u)\langle u,\bar{z}-u \rangle| \\
    &\leq
    C_2\theta^{1+\alpha}
    +
    \rho(u)(1-\cos\theta),
  \end{split}
\end{equation*}
where \eqref{eq:137} is used in the last inequality.
Note that $1-\cos\theta\leq \frac{1}{2}\theta^2$, we find
\begin{equation} \label{eq:139}
  |\rho(\bar{z})-\rho(u)-\langle \delbar \rho(u),\bar{z}-u \rangle|
  \leq
  C_3\theta^{1+\alpha},
\end{equation}
where $C_3$ is a positive constant depending only on $\norm{\rho}_{1,\alpha}$.

Now, for any $y,z\in\pd\Omega$, making use of the expression \eqref{tau}, there holds
\begin{equation*}
  \begin{split}
    |\delbar\rho(u)| \langle z-y,-\nu_y \rangle
    &=
    \langle \rho(\bar{z})\bar{z}-\rho(u)u, \delbar\rho(u) \rangle \\
    &=
    \rho(\bar{z}) \langle \bar{z}-u, \delbar\rho(u) \rangle
    +
    (\rho(\bar{z})-\rho(u)) \langle u, \delbar\rho(u) \rangle \\
    &=
    \rho(\bar{z}) \langle \bar{z}-u, \delbar\rho(u) \rangle
    -\rho(u) (\rho(\bar{z})-\rho(u)) \\
    &=
    \rho(\bar{z}) [ \langle \bar{z}-u, \delbar\rho(u) \rangle -(\rho(\bar{z})-\rho(u)) ]
    +(\rho(\bar{z})-\rho(u))^2,
  \end{split}
\end{equation*}
which, together with \eqref{eq:139} and
$|\rho(\bar{z})-\rho(u)|\leq\theta\max|\nabla\rho|$, indicates
\begin{equation}\label{eq:140}
  |\delbar\rho(u)|
  \cdot
  |\langle z-y,-\nu_y \rangle|
  \leq
  C_4\theta^{1+\alpha},
\end{equation}
where $C_4$ is a positive constant depending only on $\norm{\rho}_{1,\alpha}$.
Noting $|\delbar\rho(u)|\geq\rho_{\min}$ and recalling \eqref{eq:123}, we derive
\begin{equation*}
  |\langle z-y,-\nu_y \rangle|
  \leq
  C_5|y-z|^{1+\alpha},
\end{equation*}
where $C_5$ is a positive constant depending only on $\norm{\rho}_{1,\alpha}$
and $1/\rho_{\min}$.
The proof of this lemma is completed.
\end{proof}

Now,  recall
\begin{equation*}
  N_{a}(z)
  =\int_{\pd\Omega}
  |y-z|^{a-n} \langle y-z,\nu_y \rangle
  \omega\Bigl( \frac{y-z}{|y-z|} \Bigr)
  \phi_a(y,z)
  \dd y.
\end{equation*}
Building upon the above lemmas, we can establish the following regularity result.
\begin{lemma}\label{lem1231}
Let $0<a\leq1$ and $1-a<\alpha\leq1$.
Assume
$\rho\in C^{1,\alpha}(\uS)$,
$\omega\in C^1(\uS)$, and
$\phi\in C^1(\Omega)$.
Then $N_a\in C^1(\pd\Omega)$, and we have
\begin{multline} \label{eq:163}
  \nabla_iN_a(z)
  =\int_{\pd\Omega} |y-z|^{a-n} \bigl[
  (n-a) \langle z_i,\tfrac{y-z}{|y-z|} \rangle \langle \tfrac{y-z}{|y-z|},\nu_y \rangle \omega \phi_a
  -\langle z_i,\nu_y \rangle \omega \phi_a \\
  -\langle \tfrac{y-z}{|y-z|},\nu_y \rangle \langle z_i,\nabla\omega \rangle \phi_a
  +\langle y-z,\nu_y \rangle \omega \nabla_{z_i} \phi_a
  \bigr] \dd y,
\end{multline}
where $\phi_a$ is given by \eqref{phi-tilde}.
\end{lemma}

\begin{proof}
Let $\Sigma$ be as in \eqref{eq:160}.
Define $G:\Sigma\to\R$ as
\begin{equation}\label{eq:161}
  G(y,z)=
  |y-z|^{a-n} \langle y-z,\nu_y \rangle
  \omega\Bigl( \frac{y-z}{|y-z|} \Bigr)
  \phi_a(y,z).
\end{equation}
By the assumptions, $G$ is continuous in $\Sigma$.
Since $\rho\in C^{1,\alpha}$, applying Lemma \ref{lem1213}, we have
\begin{equation*}
  |G(y,z)|\leq
  \frac{C}{a} \norm{\omega}_0 \norm{\phi}_0 \cdot
  |y-z|^{1+\alpha+a-n},
\end{equation*}
where $C$ is the same constant as in \eqref{eq:156}.
Note that $1+\alpha+a-n>2-n$ when $\alpha>1-a$.

Recalling \eqref{eq:10} and \eqref{eq:9} in the proof of Lemma \ref{lemNak},
there is
\begin{equation}\label{eq:162}
  \begin{split}
    \nabla_{z_i}G(y,z)
    &= (n-a) |y-z|^{a-1-n} \langle z_i,\tfrac{y-z}{|y-z|} \rangle
    \langle y-z,\nu_y \rangle \omega \phi_a
    -|y-z|^{a-n} \langle z_i,\nu_y \rangle \omega \phi_a \\
    &\hskip1.11em
    -|y-z|^{a-1-n} \langle y-z,\nu_y \rangle \langle z_i,\nabla\omega \rangle \phi_a
    +|y-z|^{a-n} \langle y-z,\nu_y \rangle \omega \nabla_{z_i} \phi_a,
  \end{split}
\end{equation}
which is continuous in $\Sigma$.
Since $\rho\in C^{1,\alpha}$, $\nu_y$ is of class $C^{\alpha}$ with respect to
$y\in\pd\Omega$. Thus
\begin{equation*}
  |\langle z_i,\nu_y \rangle|
  = |\langle z_i,\nu_y-\nu_z \rangle|
  \leq |z_i|\cdot |\nu_y-\nu_z|
  \leq C_1|z_i|\cdot |y-z|^{\alpha},
\end{equation*}
where $C_1$ is a positive constant depending only on $\norm{\rho}_{1,\alpha}$
and $1/\rho_{\min}$.
It combines with Lemma \ref{lem1213} to imply
\begin{equation*}
  |\nabla_{z_i}G(y,z)|
  \leq C_2 \max|z_i| \cdot |y-z|^{\alpha+a-n},
\end{equation*}
where $C_2$ is a positive constant depending only on $n$, $a$,
$\norm{\omega}_1$, $\norm{\phi}_1$, $\norm{\rho}_{1,\alpha}$
and $1/\rho_{\min}$.
Note that $\alpha+a-n>1-n$ when $\alpha>1-a$.

Now, we have already verified that $G$ in \eqref{eq:161} satisfies all
assumptions of Lemma \ref{lem1226}, whose conclusion says
\begin{equation*}
  \nabla_{z_i} \int_{\pd\Omega} G(y,z) \dd y
  =\int_{\pd\Omega} \nabla_{z_i}G(y,z) \dd y.
\end{equation*}
Recalling Lemma \ref{lem012}, the left hand side is just $\nabla_iN_a(z)$.
By replacing $\nabla_{z_i}G(y,z)$ with \eqref{eq:162}, we derive \eqref{eq:163}.
The proof of this lemma is completed.
\end{proof}

The following lemma plays a crucial role in achieving our aim.

\begin{lemma}\label{lem031}
The following conclusions hold:

\textup{(1)}
Assume $0<\alpha\leq1$ and $g\in C^{\alpha}(V\times W)$ where $V,W$ are normed sets.
If $g(v,0)=0$ for any $v\in V$, then for any nonzero $w\in W$ and any $\beta\in(0,\alpha)$,
we have
$g(\cdot,w)|w|^{\beta-\alpha}\in C^{\beta}(V)$, and
\begin{equation}\label{eq:101}
  \begin{split}
    \bigl\Vert g(\cdot,w)|w|^{\beta-\alpha}\bigr\Vert_{0;V}&\leq [g]_{\alpha}|w|^{\beta}, \\
    [g(\cdot,w)|w|^{\beta-\alpha}]_{\beta;V}&\leq 2[g]_{\alpha}.
  \end{split}
\end{equation}

\textup{(2)}
Assume $0<\alpha\leq1$, $k\geq0$ is an integer, and
$g\in C^{k,\alpha}(V\times W)$ where $V,W$ are compact sets.
If $g(v,0)=0$ for any $v\in V$, then for any nonzero $w\in W$ and any $\beta\in(0,\alpha)$,
we have
$g(\cdot,w)|w|^{\beta-\alpha}\in C^{k,\beta}(V)$, and
\begin{equation} \label{eq:102}
  \bigl\Vert g(\cdot,w)|w|^{\beta-\alpha} \bigr\Vert_{k,\beta;V}
  \leq (\DIAM(W)+3)\norm{g}_{k,\alpha}.
\end{equation}
\end{lemma}

\begin{proof}{}
\textup{(1)}
Since $g\in C^{\alpha}(V\times W)$, for any $v,\tilde{v}\in V$ and $w,\tilde{w}\in W$, there is
\begin{equation} \label{eq:97}
  |g(v,w)-g(\tilde{v},\tilde{w})|
  \leq
  [g]_{\alpha}
  \left(
    |v-\tilde{v}|^{\alpha} + |w-\tilde{w}|^{\alpha}
  \right).
\end{equation}
Letting $\tilde{v}=v, \tilde{w}=0$ and recalling $g(v,0)=0$, one sees
\begin{equation*}
  |g(v,w)|
  \leq
  [g]_{\alpha}
  |w|^{\alpha},
\end{equation*}
which implies the first inequality of \eqref{eq:101}, and
\begin{equation}
  \label{eq:99}
  |g(v,w)-g(\tilde{v},w)|
  \leq
  2 [g]_{\alpha} |w|^{\alpha}.
\end{equation}
On the other hand, \eqref{eq:97} with $\tilde{w}=w$ says
\begin{equation*}
  |g(v,w)-g(\tilde{v},w)|
  \leq
  [g]_{\alpha}
  |v-\tilde{v}|^{\alpha},
\end{equation*}
namely
\begin{equation}
  \label{eq:98}
  |g(v,w)-g(\tilde{v},w)|^{\frac{\beta}{\alpha-\beta}}
  \leq
  [g]_{\alpha}^{\frac{\beta}{\alpha-\beta}}
  |v-\tilde{v}|^{\frac{\alpha\beta}{\alpha-\beta}}.
\end{equation}
Multiplying \eqref{eq:99} and \eqref{eq:98}, we obtain
\begin{equation*}
  |g(v,w)-g(\tilde{v},w)|^{\frac{\alpha}{\alpha-\beta}}
  \leq
  2 [g]_{\alpha}^{\frac{\alpha}{\alpha-\beta}}
  |w|^{\alpha}
  |v-\tilde{v}|^{\frac{\alpha\beta}{\alpha-\beta}},
\end{equation*}
which leads to
\begin{equation*}
  |g(v,w)-g(\tilde{v},w)|
  \leq
  2 [g]_{\alpha}
  |w|^{\alpha-\beta}
  |v-\tilde{v}|^{\beta}.
\end{equation*}
Therefore, the second inequality of \eqref{eq:101} is true.

\textup{(2)}
When $k=0$, the estimate \eqref{eq:102} follows directly from \eqref{eq:101}.
We now consider $k\geq1$.

For each $m=0,\cdots,k-1$, there are $\pd_v^mg\in C^{0,1}(V\times W)$ and
$(\pd_v^mg)(v,0)=0$.
Applying the conclusion  \textup{(1)} to $\pd_v^mg$, we have
\begin{equation*}
  \bigl\Vert \pd_v^mg(\cdot,w)|w|^{\beta-1}\bigr\Vert_{0;V}
  \leq [\pd_v^mg]_{0,1}|w|^{\beta},
\end{equation*}
it yields
\begin{equation}
  \label{eq:103}
  \bigl\Vert \pd_v^mg(\cdot,w)|w|^{\beta-\alpha}\bigr\Vert_{0;V}
  \leq [\pd_v^mg]_{0,1}|w|^{\beta+1-\alpha}
  \leq \Vert\pd_v^{m+1}g\Vert_{0}(\DIAM(W)+1).
\end{equation}

Noting $\pd_v^kg\in C^{\alpha}(V\times W)$ and
$(\pd_v^kg)(v,0)=0$,
by the conclusion \textup{(1)}, we have
\begin{equation}
  \label{eq:104}
  \bigl\Vert \pd_v^kg(\cdot,w)|w|^{\beta-\alpha}\bigr\Vert_{0;V}
  \leq [\pd_v^kg]_{\alpha}|w|^{\beta}
  \leq [\pd_v^kg]_{\alpha}(\DIAM(W)+1),
\end{equation}
and
\begin{equation}
  \label{eq:105}
  [\pd_v^kg(\cdot,w)|w|^{\beta-\alpha} ]_{\beta;V}
  \leq 2[\pd_v^kg]_{\alpha}.
\end{equation}

Now combining \eqref{eq:103}, \eqref{eq:104} and \eqref{eq:105}, we obtain
\eqref{eq:102}.
The proof of this lemma is completed.
\end{proof}

Using the above lemmas, it is essential to establish the following result.

\begin{lemma}\label{lem117}
Let $k\geq0$, $0\leq\gamma<\alpha\leq1$, and $\rho\in C^{k+1,\alpha}(\uS)$.
Assume $\varphi\in C^{k,\alpha}(\uS\times\pd\Omega\times\pd\Omega)$ satisfies
\begin{equation}\label{eq:164}
  |\varphi(\varsigma,y,z)-\varphi(\tilde{\varsigma},y,z)|
  \leq c_{\varphi}|\varsigma-\tilde{\varsigma}|,
  \quad
  \forall\, \varsigma,\tilde{\varsigma}\in\uS, \, y,z\in\pd\Omega
\end{equation}
for a positive constant $c_{\varphi}$.
If for any $z\in\pd\Omega$,
\begin{equation} \label{eq:117}
  \lim_{\pd\Omega\ni y\to z} \varphi\bigl( \tfrac{y-z}{|y-z|},y,z \bigr) =0,
\end{equation}
then
\begin{equation}\label{Ig}
  \mathcal{I}_{\gamma}(z)
  :=\int_{\pd\Omega}
  |y-z|^{1-n-\gamma}
  \varphi\bigl( \tfrac{y-z}{|y-z|},y,z \bigr)
  \dd y
\end{equation}
is $C^{k,\beta}(\pd\Omega)$ for any $0<\beta<\alpha-\gamma$, and
\begin{equation}\label{eq:107}
  \norm{\mathcal{I}_{\gamma}}_{k,\beta} \leq
  C(n, \alpha-\gamma-\beta, k, \norm{\varphi}_{k,\alpha}, c_{\varphi},
  \norm{\rho}_{k+1,\alpha}, 1/{\rho_{\min}})
\end{equation}
for some positive constants $C$ depending only on these quantities listed above.
\end{lemma}

\begin{proof}{}
Utilizing the same arguments as in the proof of Lemma \ref{lem1205}, we have
\begin{equation*}
  \begin{split}
    \mathcal{I}_{\gamma}(\bar{z}) := \mathcal{I}_{\gamma}(z)
    &=\int_{\uS} |\tau|^{1-n-\gamma}
    \varphi\bigl( \tfrac{\tau}{|\tau|},\vec{r}(u),\vec{r}(\bar{z}) \bigr)
    \rho(u)^{n-2} |\delbar\rho(u)| \dd u \\
    &= \int_0^{\pi} \theta^{1-n-\gamma} \sin^{n-2}\theta \dd\theta
    \int_{v\bot \bar{z}} |\xi|^{1-n-\gamma}
    \varphi\bigl( \tfrac{\xi}{|\xi|},\vec{r}(u),\vec{r}(\bar{z}) \bigr)
    \rho(u)^{n-2} |\delbar\rho(u)| \dd v \\
    &= \int_0^{\pi} \theta^{1-n+\alpha-\gamma-\beta} \sin^{n-2}\theta \dd\theta
    \int_{v\bot \bar{z}} g(\bar{z},v,\theta) \dd v,
  \end{split}
\end{equation*}
where
\begin{equation}\label{eq:112}
  g(\bar{z},v,\theta)=
  |\xi|^{1-n-\gamma}
  \varphi\bigl( \tfrac{\xi}{|\xi|},\vec{r}(u),\vec{r}(\bar{z}) \bigr) \theta^{\beta-\alpha}
  \cdot
  \rho(u)^{n-2} |\delbar\rho(u)|.
\end{equation}
By estimate \eqref{eq:124} and the assumptions on $\varphi$, one sees that
$\varphi\bigl( \tfrac{\xi}{|\xi|},\vec{r}(u),\vec{r}(\bar{z}) \bigr)$
is $C^{k,\alpha}$ with respect to $(\bar{z},v,\theta)$.
Observe that the assumption \eqref{eq:117} leads to
\begin{equation*}
  \varphi\bigl( \tfrac{\xi}{|\xi|},\vec{r}(u),\vec{r}(\bar{z}) \bigr) \big|_{\theta=0}
  =0.
\end{equation*}
Then, by virtue of Lemma \ref{lem031}, for any
$\theta\in(0,\pi]$ and $\beta\in(0,\alpha-\gamma)$, we have that
$\varphi\bigl( \tfrac{\xi}{|\xi|},\vec{r}(u),\vec{r}(\bar{z}) \bigr)
\theta^{\beta-\alpha}$ is $C^{k,\beta}$
with respect to $(\bar{z},v)$, and that
\begin{equation}\label{eq:108}
  \bigl\Vert
  \varphi\bigl( \tfrac{\xi}{|\xi|},\vec{r}(u),\vec{r}(\bar{z}) \bigr)
  \theta^{\beta-\alpha}
  \bigr\Vert_{k,\beta}
  \leq C_1,
\end{equation}
where $C_1$ is a positive constant depending only on $n$, $k$,
$\norm{\varphi}_{k,\alpha}$, $c_{\varphi}$, $\norm{\rho}_{k+1,\alpha}$ and $1/\rho_{\min}$.
Thus, $g(\bar{z},v,\theta)$ is $C^{k,\beta}$ with respect to $(\bar{z},v)$ and
the corresponding norms are uniformly bounded for $\theta\in(0,\pi]$.
Since $\beta\in(0,\alpha-\gamma)$, $\alpha-\gamma-\beta>0$.
Therefore, $\mathcal{I}_{\gamma}(\bar{z})\in C^{k,\beta}$ and its norm satisfies
the estimate \eqref{eq:107}.
This completes the proof of this lemma.
\end{proof}

 Lemma \ref{lem117} enables us to reach our goal. We now derive the following lemma.

\begin{lemma}\label{lemNa<=1}
Let $0<a\leq1$, $k\geq0$, and $1-a<\alpha\leq1$.
Assume $\rho\in C^{k+1,\alpha}(\uS)$,
$\omega\in C^{k+1,\alpha}(\uS)\cap C^{1,1}(\uS)$, and
$\phi\in C^{k+1,\alpha}(\Omega)$.
Then $N_a$ is $C^{k+1,\beta}(\pd\Omega)$ for any $0<\beta<\alpha+a-1$, and there holds
\begin{equation}\label{eq:95}
  \norm{N_a}_{k+1,\beta} \leq
  C(n, a, \alpha+a-1-\beta, k, \norm{\omega}_{k+1,\alpha}, \norm{\omega}_{1,1},
  \norm{\phi}_{k+1,\alpha}, \norm{\rho}_{k+1,\alpha}, 1/{\rho_{\min}})
\end{equation}
for some positive constants $C$ depending only on these quantities listed above.
\end{lemma}

\begin{proof}{}
By Lemma \ref{lem1231}, $N_a\in C^1(\pd\Omega)$, and there is
\begin{equation*}
  \nabla_iN_a(z)
  =\int_{\pd\Omega} |y-z|^{1-n-(1-a)}
  \varphi\bigl( \tfrac{y-z}{|y-z|},y,z \bigr)
  \dd y,
\end{equation*}
where
\begin{multline*}
  \varphi(\varsigma,y,z)=
  (n-a) \langle z_i,\varsigma \rangle \langle \varsigma,\nu_y \rangle \omega(\varsigma) \phi_a
  -\langle z_i,\nu_y \rangle \omega(\varsigma) \phi_a \\
  -\langle \varsigma,\nu_y \rangle \langle z_i,\nabla\omega(\varsigma) \rangle \phi_a
  +\langle y-z,\nu_y \rangle \omega(\varsigma) \nabla_{z_i} \phi_a.
\end{multline*}
Recalling \eqref{Ig}, then we see that $\nabla_i N_a(z)=\mathcal{I}_{\gamma}(z)$
with $\gamma=1-a<\alpha$.

By the assumptions of this lemma, $\varphi\in C^{k,\alpha}$ and satisfies
\eqref{eq:164} with $c_{\varphi}$ depending only on $n$, $a$, $\norm{\omega}_{1,1}$,
$\norm{\phi}_1$, $\norm{\rho}_1$ and $1/\rho_{\min}$.
Recalling Lemma \ref{lem1213}, we have
\begin{equation*}
  \lim_{\pd\Omega\ni y\to z} \Bigl\langle \frac{y-z}{|y-z|},\nu_y \Bigr\rangle =0,
\end{equation*}
implying
\begin{equation*}
  \lim_{\pd\Omega\ni y\to z} \varphi\bigl( \tfrac{y-z}{|y-z|},y,z \bigr) =0.
\end{equation*}
Now, we have already verified all assumptions of Lemma \ref{lem117},
whose conclusion says that
$\nabla_i N_a$ is $C^{k,\beta}(\pd\Omega)$ for any $0<\beta<\alpha+a-1$, and
\begin{equation*}
  \norm{\nabla_iN_a}_{k,\beta} \leq
  C(n, \alpha+a-1-\beta, k, \norm{\varphi}_{k,\alpha}, c_{\varphi},
  \norm{\rho}_{k+1,\alpha}, 1/{\rho_{\min}}).
\end{equation*}
Hence, $N_a\in C^{k+1,\beta}(\pd\Omega)$, and then together with Lemma \ref{lem1128},
we have the desired estimate \eqref{eq:95} of this lemma.
\end{proof}
\begin{proof}[Proof of Theorem \ref{thm1}]
Theorem \ref{thm1} is a direct consequence of Lemmas \ref{lemNa>1} and \ref{lemNa<=1}.
\end{proof}
\begin{remark}\label{rmk0107}
In Lemmas \ref{lem1205} and \ref{lem117}, the assumptions \eqref{eq:110} and
\eqref{eq:164} are only used to show that $\varphi$ in \eqref{eq:125} and
\eqref{eq:112} is $C^{\alpha}$ with respect to $(\bar{z},v,\theta)$ when $k=0$.
If they are removed, we can see that $\varphi$ is $C^{\tilde{\alpha}}$ for some
$\tilde{\alpha}\leq \alpha$.
The other parts of the whole arguments are not affected.
Therefore, the conclusions of Theorem \ref{thm1} are still true, except that we
may obtain a smaller H\"older exponent.
However, in the case when $a\geq2$, the $C^{\alpha}$ regularity of $\nabla N_a$ can be
derived by Lemma \ref{lem601}, rather than Lemma \ref{lem1205}, indicating that the
H\"older exponent remains at $\alpha$.
\end{remark}

\begin{proof}[{Proof of Theorem \ref{thm2}}]
The conclusions of (i) and (ii) are direct consequences of Theorem \ref{thm1}.
For the conclusion of (iii), since \(\partial\Omega\in C^{1,\alpha}\) with
\(\alpha>-q\), we have the same expression as in Lemma \ref{lem012}:
\begin{equation*}
		\widetilde V_q(\Omega,z)
		=\frac1n\int_{\partial\Omega}
		|y-z|^{q-n}\langle y-z,\nu_y\rangle \dd y.
\end{equation*}
Applying Lemma \ref{lem117} with $\gamma=-q$, we obtain (iii).
\end{proof}

Building upon Theorem \ref{thm2}, we can prove Theorem \ref{thm3}.

\begin{proof}[Proof of Theorem \ref{thm3}]   Let $\mu$ be an even finite Borel measure on $\sn$ satisfying
\[
\frac{2q}{(n+q-1)\omega_{n}}\int_{\nu^{-1}_{\Omega}(\eta)}\langle z, \nu_{\Omega}(z)\rangle\widetilde{V}_{q-1}(\Omega,z)\dd \mathcal{H}^{n-1}(z)=\int_{\eta}\dd \mu
\]
for every Borel set $\eta\subset \sn$ and $1\leq q \leq n+1$. Assume that $\dd\mu=\frac{2q}{(n+q-1)\omega_{n}}f(x)\dd x$ for some positive function $f\in C^{k,\alpha}(\sn)$ with $0<\alpha<1$. Let $\phi$ be a convex, Lipschitz function defined on an open subset $O$ of $\rt$, whose graph $\{(\vartheta,\phi(\vartheta)):\vartheta\in O\}$ is a part of $\partial \Omega$. Then $\phi$ satisfies the {M}onge-{A}mp\`ere equation
\begin{equation}\label{M2}
\det(\phi_{ij}(\vartheta))=(1+|D \phi(\vartheta)|^{2})^{(n+1)/2}\langle(\vartheta,\phi(\vartheta)), x\rangle\frac{\widetilde{V}_{q-1}(\Omega,(\vartheta,\phi(\vartheta)))}{f(x)},
\end{equation}
in the sense of Aleksandrov, where
\[
x=\frac{(D\phi(\vartheta),-1)}{\sqrt{1+|D\phi(\vartheta)|^{2}}}.
\]
In view of \eqref{M2}, applying  \cite[Theorem 1.2]{LXYZ} with the assumption
that $f$ is bounded above and below by a positive constant, by Caffarelli's
result \cite{Ca901}, we conclude that $\partial \Omega$ is $C^{1,\alpha_{1}}$
and obtain $C^{1,\alpha_{1}}$ estimate of Aleksandrov solution for $0<\alpha_{1}<1$.  Now, assume
that $f\in C^{\alpha}(\sn)$ for $0<\alpha<1$ and $f>0$, since we have revealed that $\Omega$ is a
$C^{1,\alpha_{1}}$ domain, by Theorem \ref{thm2}, one can see that
$\widetilde{V}_{q-1}(\cdot)$ is bounded above and below by a positive constant, which belongs to $C^{1,\alpha_{1}}$ for $q>2$ and $C^{\alpha'_{1}}$ for $1< q \leq2$, then the r.h.s of \eqref{M2} belongs to
$C^{\alpha_{3}}$ for $\alpha_{3}>0$, so by Caffarelli \cite{Ca90}, $\phi \in C^{2,\alpha_{3}}$, which
illustrates that $\widetilde{V}_{q-1}(\cdot)$ belongs to $C^{2,\alpha_{3}}$ for $q>2$ and belongs to $C^{1,\alpha'_{3}}$ for $1<q\leq 2$.
The remainder of the regularity estimates of Theorem \ref{thm3} now follows from
\cite{Ca89,Ca90,Ca901}.

\end{proof}

\section{Nonlocal flow and related functionals} \label{sec4}

As stated in the introduction,  no existence results are currently available for generalized
solutions to the chord log-Minkowski problem \eqref{CLMP} when $q>n+1$.
In this case, based on Theorem \ref{thm2}, we instead study the convergence of the nonlocal geometric flow
\eqref{flow} to prove the existence of smooth solutions to \eqref{CLMP}.
This method actually works for every $q>3$.
Therefore, we always assume $q>3$ in the following sections of this paper, unless
otherwise specified.

Let $\Omega_0$ be a smooth, origin-symmetric, uniformly convex body in $\R^{n}$,
whose boundary $\pd\Omega_0$ is given by
$X_{0}: \mathbb{S}^{n-1} \rightarrow \R^{n}$,
and the chord integral $I_q(\Omega_0)$ satisfies $I_{q}(\Omega_0)
  =\frac{2q\int_{\uS}f\dd x}{(q+n-1)\omega_{n}}$.
We consider a family of closed convex hypersurfaces $\set{\pd\Omega_t}$ parameterized by
$\pd\Omega_t=X(\uS,t)$, where $X: \mathbb{S}^{n-1}\times[0,T) \rightarrow \R^{n}$ is a
smooth map satisfying the flow equation reflected in \eqref{flow}.
Since we require $\pd\Omega_t$ to be uniformly convex, by the definition of support function, one can see
that this flow equation is equivalent to the following evolution equation:
\begin{equation}\label{flow1}
  \begin{split}
  \left\{
\begin{array}{lr}
    \frac{\pd h}{\pd t} (x,t)
    = -\frac{f(x) \kappa}{\widetilde{V}_{q-1}(\Omega_t,\delbar h)} +h(x,t), \\
    h(x,0)= h_{0}(x),
      \end{array}\right.
  \end{split}
\end{equation}
where $h(\cdot,t)$ is the support function of $\Omega_t$.
From now on, we will mainly study the flow equation \eqref{flow1}, which involves the
nonlocal dual quermassintegral $\widetilde{V}_{q-1}(\Omega_t,\delbar h)$.
Due to $f(-x)=f(x)$, the origin-symmetry of $\Omega_0$, and
\begin{equation*}
\widetilde{V}_{q-1}(\Omega,-z) = \widetilde{V}_{q-1}(\Omega,z)
\end{equation*}
for any origin-symmetric $\Omega$ and $z\in\pd\Omega$,
$\Omega_t$ is origin-symmetric along the flow.
Denote by $\rho(\cdot,t)$ the radial function of $\Omega_t$.

Note that the short-time existence of a solution to \eqref{flow1} with a global term can be proved by a fixed point argument similar to that in \cite{Ge06,Mak13}. We now first note a simple but crucial property about the chord integral
$I_q(\Omega_t)$ along the flow  \eqref{flow1}.

\begin{lemma} \label{lem004}
Along the flow \eqref{flow1}, the chord integral $I_q(\Omega_t)$ can be computed by
\begin{equation}
  \label{eq:26}
  I_q(\Omega_t) =
  \frac{2q \int_{\uS}f \dd x}{(q+n-1)\omega_{n}}
  +
  \left(
    I_q(\Omega_0)
    -\frac{2q \int_{\uS}f \dd x}{(q+n-1)\omega_{n}}
  \right)
  e^{(q+n-1)t}.
\end{equation}
Therefore, if we choose $\Omega_0$ such that
\begin{equation*}
  I_q(\Omega_0)
  =\frac{2q \int_{\uS}f \dd x}{(q+n-1)\omega_{n}},
\end{equation*}
then $I_q(\Omega_t)$ remains constant, namely
\begin{equation*}
  I_q(\Omega_t)\equiv I_q(\Omega_0),
  \quad
  \forall t\in[0,T).
\end{equation*}
\end{lemma}

\begin{proof}{}
By virtue of \cite[Lemma 5.5]{LXYZ}, we have
\begin{equation}\label{eq:27}
  \begin{split}
    \frac{\dd}{\dd t} I_q(\Omega_t)
    &=
    \frac{2q}{\omega_n}
    \int_{\uS} \pd_th(x,t) \widetilde{V}_{q-1}(\Omega_t,\delbar h) \kappa^{-1} \dd x \\
    &=
    \frac{2q}{\omega_n}
    \int_{\uS} h(x,t) \widetilde{V}_{q-1}(\Omega_t,\delbar h) \kappa^{-1} \dd x
    -
    \frac{2q}{\omega_n} \int_{\uS}f \dd x \\
    &=
    (q+n-1) I_q(\Omega_t)
    -
    \frac{2q \int_{\uS}f \dd x}{\omega_n},
  \end{split}
\end{equation}
where the second equality holds due to the flow \eqref{flow1}, and the third equality holds due
to formula \eqref{c-i-formula}.
By solving the linear ordinary differential equation \eqref{eq:27}, one will gain
\eqref{eq:26}, thus completing the proof of this lemma.
\end{proof}

Another important functional related to the flow \eqref{flow1} is the following
\begin{equation}\label{Jt}
  J(t)=\int_{\uS} f(x) \log h(x,t) \dd x,
\end{equation}
which will be equipped with the monotonicity along the flow \eqref{flow1}. It is shown in the following lemma.
\begin{lemma} \label{lem005}
Along the flow \eqref{flow1}, $J(t)$ is non-increasing.
\end{lemma}

\begin{proof}{}
By a direct calculation, and utilizing \eqref{flow1}, we get
\begin{equation} \label{eq:28}
  \begin{split}
    J'(t)
    &=
    \int_{\uS} \frac{f(x)}{h(x,t)} \pd_th(x,t) \dd x \\
    &=
    \int_{\uS} f  \dd x
    - \int_{\uS} \frac{f^2 \kappa}{h\widetilde{V}_{q-1}} \dd x.
  \end{split}
\end{equation}
By H\"older's inequality, there is
\begin{equation}\label{eq:29}
  \begin{split}
    \left(\int_{\uS}f  \dd x\right)^2
    &=
    \left(
      \int_{\uS}
      \sqrt{\frac{h\widetilde{V}_{q-1}}{\kappa}}
      \cdot
      f
      \sqrt{\frac{\kappa}{h\widetilde{V}_{q-1}}} \dd x
    \right)^2 \\
    &\leq
    \left(
      \int_{\uS} \frac{h\widetilde{V}_{q-1}}{\kappa} \dd x
    \right)
    \cdot
    \left(
      \int_{\uS} \frac{f^2\kappa}{h\widetilde{V}_{q-1}} \dd x
    \right).
  \end{split}
\end{equation}
Recalling \eqref{c-i-formula} and Lemma \ref{lem004}, one sees
\begin{equation*}
  \int_{\uS} \frac{h\widetilde{V}_{q-1}}{\kappa} \dd x
  =
  \frac{(q+n-1)\omega_{n}}{2q}
  I_q(\Omega_t)
  =
  \int_{\uS}f \dd x,
\end{equation*}
which, together with \eqref{eq:29}, yields
\begin{equation*}
  \int_{\uS}f  \dd x
  \leq
  \int_{\uS} \frac{f^2\kappa}{h\widetilde{V}_{q-1}} \dd x.
\end{equation*}
Now \eqref{eq:28} says that $J'(t)\leq0$.
The proof of this lemma is complete.
\end{proof}

With the help of the above results, now we can  prove that the support
functions $h(\cdot,t)$ have uniform positive upper and lower bounds.

\begin{lemma}\label{lem03}
Let $f$ be an even, positive, smooth function on $\sn$, and $\Omega_{t}$ be an origin symmetric, uniformly convex solution to the flow \eqref{flow}. Then there exists a positive constant $C$
independent of $t$ such that for every $t\in[0,T)$,
\begin{equation}\label{eq:30}
  1/C \leq h(\cdot,t) \leq C \text{ on }\uS.
\end{equation}
It means that
\begin{equation}\label{eq:31}
  1/C \leq \rho(\cdot,t) \leq C \ \text{ on }\uS.
\end{equation}
\end{lemma}

\begin{proof}
For each $t$, write
\begin{equation*}
  R_t = \max_{x\in\uS} h(x,t) = h(x_t,t)
\end{equation*}
for some $x_t\in\uS$.
Recalling $\Omega_t$ is origin-symmetric,
by the definition of the support function, we have
\begin{equation*}
  h(x,t)\geq R_t\cdot|\langle x,x_t \rangle|, \quad \forall x\in\uS.
\end{equation*}
By Lemma \ref{lem005}, $J(t)$ is non-increasing.
Thus,
\begin{equation}\label{eq:32}
  \begin{split}
    J(0)
    &\geq \int_{\uS} f(x) \log h(x,t) \dd x \\
    &\geq \int_{\uS} f(x) \left( \log R_t +\log|\langle x,x_t \rangle| \right) \dd x \\
    &= \log R_t \int_{\uS} f(x) \dd x
    +\int_{\uS} f(x) \log|\langle x,x_t \rangle| \dd x.
  \end{split}
\end{equation}
Note $|\langle x,x_t \rangle|\leq1$, there is
\begin{equation*}
  \begin{split}
    \int_{\uS} f(x) \log|\langle x,x_t \rangle| \dd x
    &\geq
    f_{\max} \int_{\uS} \log|\langle x,x_t \rangle| \dd x \\
    &=
    f_{\max} \int_{\uS} \log|x_n| \dd x \\
    &=
    2(n-1)\omega_{n-1} \int_0^{\frac{\pi}{2}} (\log \cos\theta) \sin^{n-2}\theta \dd \theta
    \cdot
    f_{\max} \\
    &=
    -C_n f_{\max},
  \end{split}
\end{equation*}
which together with \eqref{eq:32}, yields
\begin{equation*}
  J(0)
  \geq
  \log R_t \int_{\uS} f(x) \dd x
  -C_n f_{\max}.
\end{equation*}
Hence, $R_t\leq C$ for a
 positive constant depending only on $n$, $f$ and
$\Omega_0$, which is just the second inequality of \eqref{eq:30}.

 We are now in a position to verify the first inequality of \eqref{eq:30}.
Suppose to the contrary that there exists a sequence of times $t_k\in[0,T)$ such
that
\begin{equation}\label{eq:33}
  \min_{\uS} h(\cdot, t_k) \to0^+ \text{ as }k\to\infty.
\end{equation}
Since we have already demonstrated that $\set{\Omega_{t_k}}$ is a bounded sequence of
origin-symmetric convex bodies, by Blaschke's selection theorem,
$\set{\Omega_{t_k}}$ has a subsequence which converges to a nonempty, compact,
origin-symmetric convex subset, denoted by $\Omega$.
Without loss of generality, we assume
\begin{equation*}
  \Omega_{t_k}\to \Omega \text{ as }k\to\infty.
\end{equation*}
Therefore, the support function of $\Omega$ satisfies
\begin{equation*}
  \min_{\uS} h_{\Omega}
  = \lim_{k\to\infty} \min_{\uS} h(\cdot,t_k)
  =0,
\end{equation*}
which means that $\Omega$ is contained in a hyperplane in $\R^n$.
Hence, the chord integral $I_q(\Omega)=0$.
However, on the other hand, by the continuity of the chord integral and
Lemma \ref{lem004}, we have
\begin{equation*}
  I_q(\Omega)
  = \lim_{k\to\infty} I_q(\Omega_{t_k})
  = I_q(\Omega_0)
  >0.
\end{equation*}
This contradiction indicates that the assumption \eqref{eq:33} cannot hold.
Thus, the first inequality of \eqref{eq:30} is true.
This completes the proof of this lemma.
\end{proof}

By this lemma, we have uniform bounds for $\widetilde{V}_{q-1}$ as below.

\begin{lemma}\label{lem04}
Let $f$ be an even, positive, smooth function on $\sn$, and $\Omega_{t}$ be an origin symmetric, uniformly convex solution to the flow \eqref{flow}. Then there is a positive constant $C$ depending only on $n$, $q$, and the
constant in Lemma \ref{lem03},
such that for every $t\in[0,T)$,
\begin{equation*}
  1/C \leq \widetilde{V}_{q-1}(\Omega_t,\delbar h) \leq C \ \text{ on }\uS.
\end{equation*}
\end{lemma}

\begin{proof}{}
  Utilizing Lemma \ref{lem03} and the definition of $\widetilde{V}_{q-1}$,
  the conclusion is obviously true.
\end{proof}

By the convexity of hypersurface, the gradient estimates of $h(\cdot,t)$ and $\rho(\cdot,t)$
directly follow from their upper and lower bounds.

\begin{lemma}\label{lem05}
Let $f$ be an even, positive, smooth function on $\sn$, and $\Omega_{t}$ be an origin symmetric, uniformly convex solution to the flow \eqref{flow}. Then there is a positive constant $C$ depending only on the constant in Lemma \ref{lem03} such that for every $t\in[0,T)$,
\begin{gather*}
  |\nabla h(x,t)| \leq C,
  \quad \forall (x,t) \in \mathbb{S}^{n-1} \times [0, T), \\
  |\nabla \rho(u,t)| \leq C,
  \quad \forall (u,t) \in \mathbb{S}^{n-1} \times [0, T).
\end{gather*}

\end{lemma}

\begin{proof}
Connecting $u$ and $x$ through
\begin{equation*}
  \rho(u,t)u=\nabla h(x,t)+h(x,t)x.
\end{equation*}
Thus we have the following well-known facts:
\begin{equation*}
  \rho=\sqrt{|\nabla h|^2+h^2},
  \qquad
  h=\frac{\rho^2}{\sqrt{|\nabla\rho|^2+\rho^2}}.
\end{equation*}
Now one can see that Lemma \ref{lem03} implies the conclusions of this lemma.
\end{proof}

\section{Derivatives and variations of the generalized dual quermassintegral} \label{sec5}

To study the flow \eqref{flow1}, we further need to establish uniform upper and
lower bounds for the principal curvatures,
which  requires differentiating \eqref{flow1} with respect to $x$
and $t$.
Since the flow involves the nonlocal dual quermassintegral
$\widetilde{V}_{q-1}(\Omega_t,\delbar h)$,
we first need to know how to calculate its derivatives with respect
to $x$ and $t$, this is the main purpose of this section.

\subsection{Explicit derivatives of $\widetilde{V}_q(\Omega,\cdot)$ when $q>2$}

As established in Section \ref{sec3}, for each $q>-1$, the functional $\widetilde{V}_q(\Omega,\cdot)$ is smooth whenever $\pd\Omega$ is smooth. However, it seems hard to provide explicit expressions for derivatives of all orders.
Recalling Lemma \ref{lemNak}, it is possible to obtain the explicit $k$-th order
derivatives whenever $k<q$.

In fact, when $q>1$, the first order derivatives have already been explicitly
calculated in \eqref{eq:13}.
Letting $\omega\equiv1$ and $\phi\equiv1$, and recalling the value of $\varsigma$  taken as $\frac{y-z}{|y-z|}$, we obtain the following lemmas.

\begin{lemma}\label{lem006}
Assume $q>1$ and $\pd\Omega\in C^1$. Then for any $z\in\pd\Omega$, there is the following identity:
\begin{equation} \label{eq:34}
  \nabla_i \widetilde{V}_{q}(\Omega,z)
  =  \frac{q(n-q)}{n}\int_{\Omega}
  \langle
  z_i,\tfrac{y-z}{|y-z|}
  \rangle
  |y-z|^{q-1-n}
  \dd y.
\end{equation}
\end{lemma}

Using the same notations in Lemma \ref{lemNak}, and recalling \eqref{eq:10} and
\eqref{eq:11}, we have
\begin{equation*}
  \begin{split}
    \nabla_i(\varsigma^A|y-z|^{q-1-n})
    &=
    |y-z|^{q-2-n}
    \left(
      (q-1-n)\varsigma^A\nabla_i|y-z|
      +
      |y-z| \nabla_i\varsigma^A
    \right) \\
    &=
    |y-z|^{q-2-n}
    \left(
      (n+1-q)z^{A_1}_i\varsigma^{A_1}\varsigma^A
      -z^{A_1}_i \pd_{\tilde{\alpha}} \varsigma^{A_1} \cdot \pd_{\tilde{\alpha}} \varsigma^A
    \right) \\
    &=
    |y-z|^{q-2-n}
    z^{A_1}_i
    \left(
      (n+2-q)\varsigma^{A_1}\varsigma^A
      -\delta_{A_1A}
    \right),
  \end{split}
\end{equation*}
which yields
\begin{equation*}
  z^A_j\nabla_i(\varsigma^A|y-z|^{q-1-n})
  =
  |y-z|^{q-2-n}
  \left[
    (n+2-q)
    \langle z_i,\varsigma \rangle
    \langle z_j,\varsigma \rangle
    - \langle z_i,z_j \rangle
  \right].
\end{equation*}
Now writing \eqref{eq:34} as
\begin{equation*}
  \nabla_j \widetilde{V}_{q}
  =\frac{q(n-q)}{n}
  z^{A}_j
  \int_{\Omega}
  \varsigma^{A}
  |y-z|^{q-1-n}
  \dd y,
\end{equation*}
then one can compute
\begin{equation*}
  \begin{split}
    \nabla_{ij}^2 \widetilde{V}_{q}
    &=\frac{q(n-q)}{n}
    z^{A}_{ji}
    \int_{\Omega}
    \varsigma^{A}
    |y-z|^{q-1-n}
    \dd y
    +\frac{q(n-q)}{n}
    z^{A}_j
    \int_{\Omega}
    \nabla_i(
    \varsigma^{A}
    |y-z|^{q-1-n})
    \dd y \\
    &=\frac{q(n-q)}{n}
    \int_{\Omega}
    \langle z_{ji},\varsigma \rangle
    |y-z|^{q-1-n}
    \dd y \\
    &\hskip1.2em
    + \frac{q(n-q)}{n}\int_{\Omega}
    \left[
      (n+2-q)
      \langle z_i,\varsigma \rangle
      \langle z_j,\varsigma \rangle
      - \langle z_i,z_j \rangle
    \right]
    |y-z|^{q-2-n}
    \dd y.
  \end{split}
\end{equation*}
Thus, the following lemma can be derived.
\begin{lemma}\label{lem007}
Assume $q>2$ and $\pd\Omega\in C^2$. Then for any $z\in\pd\Omega$, there is the following identity:
\begin{equation}
\begin{split}
 \label{eq:35}
  \nabla_{ij}^2 \widetilde{V}_{q}(\Omega,z)
  &=\frac{q(n-q)}{n}
  \int_{\Omega}
  \langle {z_{ji},\tfrac{y-z}{|y-z|}} \rangle
  |y-z|^{q-1-n}
  \dd y \\
  &\quad+ \frac{q(n-q)}{n}\int_{\Omega}
  \left(
    (n+2-q)
    \langle {z_i,\tfrac{y-z}{|y-z|}} \rangle
    \langle {z_j,\tfrac{y-z}{|y-z|}} \rangle
    - \langle z_i,z_j \rangle
  \right)
  |y-z|^{q-2-n}
  \dd y.
\end{split}
\end{equation}

\end{lemma}

Now we can compute the composite derivatives of $\widetilde{V}_q(\Omega,\delbar h(x))$.
Recall that $\set{e_1, \cdots, e_{n-1}}$ is a local orthonormal frame on $\uS$,
and recall the notation
\begin{equation} \label{bij}
  b_{ij} = h_{ij}+h\delta_{ij}.
\end{equation}
Selecting the local frame
$\set{z_1, \cdots, z_{n-1}}$
on $\pd\Omega$ as
\begin{equation*}
  z_i
  =(\delbar h)_*e_i
  =e_jb_{ij},
\end{equation*}
we have
\begin{equation} \label{eq:39}
  \begin{split}
    \nabla_{e_i}
    \widetilde{V}_q(\Omega,\delbar h(x))
    &=
    \nabla_{z_i}
    \widetilde{V}_q(\Omega,z) \big|_{z=\delbar h}, \\
    \nabla_{e_ie_j}^2
    \widetilde{V}_q(\Omega,\delbar h(x))
    &=
    \nabla_{z_iz_j}^2
    \widetilde{V}_q(\Omega,z) \big|_{z=\delbar h},
  \end{split}
\end{equation}
and one sees
\begin{equation}\label{eq:40}
  z_{ji}=e_kb_{jki}-xb_{ji}.
\end{equation}
Combining \eqref{eq:34}, \eqref{eq:35}, \eqref{eq:39} and \eqref{eq:40}, we
finally derive the explicit derivatives of $\widetilde{V}_q(\Omega,\delbar h(x))$.

\begin{lemma}\label{lem008}
Assume $q>2$, $\pd\Omega\in C^2$, and $h\in C^3(\uS)$. Then for any $x\in\uS$, there are the following identities:
\begin{equation} \label{eq:41}
  \nabla_i \widetilde{V}_{q}(\Omega,\delbar h(x))
  =\frac{q(n-q)}{n}
  b_{ik}
  \int_{\Omega}
  \langle
  e_k,\tfrac{y-\delbar h}{|y-\delbar h|}
  \rangle
  |y-\delbar h|^{q-1-n}
  \dd y,
\end{equation}
and
\begin{multline} \label{eq:36}
  \nabla_{ij}^2 \widetilde{V}_{q}(\Omega,\delbar h(x)) \\
  =\frac{q(n-q)}{n}
  b_{jki}
  \int_{\Omega}
  \langle {e_k ,\tfrac{y-\delbar h}{|y-\delbar h|}} \rangle
  |y-\delbar h|^{q-1-n}
  \dd y
  -\frac{q(n-q)}{n}b_{ji}
  \int_{\Omega}
  \langle {x ,\tfrac{y-\delbar h}{|y-\delbar h|}} \rangle
  |y-\delbar h|^{q-1-n}
  \dd y \\
  +\frac{q(n-q)}{n}b_{ik}b_{jl}
  \int_{\Omega}
  \left(
    (n+2-q)
    \langle {e_k,\tfrac{y-\delbar h}{|y-\delbar h|}} \rangle
    \langle {e_l,\tfrac{y-\delbar h}{|y-\delbar h|}} \rangle
    - \delta_{kl}
  \right)
  |y-\delbar h|^{q-2-n}
  \dd y.
\end{multline}
\end{lemma}

\subsection{Variational formulas for $\widetilde{V}_q(\Omega_t,z_t)$ when $q>1$}

\begin{lemma}\label{lem009}
Assume $q>1$, and
$\rho(\cdot,t)$ is the radial function of $\Omega_t$. For
$t\in(t_0,t_1)$,
  %and $\pd\Omega\in C^1$.
and for any $z_t\in\pd\Omega_t$, then we get
\begin{multline}\label{eq:47}
  \frac{\dd}{\dd t}
  \widetilde{V}_q(\Omega_t,z_t)
  =
  q \int_{\Omega_t} \frac{\rho'\Bigl( \frac{y}{|y|},t \Bigr)}{\rho\Bigl( \frac{y}{|y|},t \Bigr)}
  \frac{\dd y}{|y-z_t|^{n-q}} \\
  +
  \frac{q(q-n)}{n}
  \int_{\Omega_t}
  \Biggl\langle
  \frac{\rho'\Bigl( \frac{y}{|y|},t \Bigr)}{\rho\Bigl( \frac{y}{|y|},t \Bigr)} y -z_t',
  \tfrac{y-z_t}{|y-z_t|}
  \Biggr\rangle
  \frac{\dd y}{|y-z_t|^{n-(q-1)}},
\end{multline}
where $\rho'(\cdot,t)$ and $z'_{t}$ denote the derivatives of $\rho(\cdot,t)$ and $z_{t}$ with respect to $t$ respectively.
\end{lemma}

\begin{proof}{}
For each $t$, applying the variable substitution from the unit ball onto
$\Omega_t$:
\begin{equation}\label{eq:52}
  y= \rho\left( \frac{v}{|v|},t \right) v,
  \quad \forall |v|<1.
\end{equation}
Thus, we get
\begin{equation}\label{Vade}
  \widetilde{V}_q(\Omega_t,z_t)
  =\frac{q}{n}
  \int_{|v|<1}
  \frac{\rho(\bar{v},t)^n \dd v}
  {| \rho(\bar{v},t)v -z_t |^{n-q}},
\end{equation}
where $\bar{v}$ denotes $\tfrac{v}{|v|}$ for simplicity.
Differentiating formally, there is
\begin{equation}\label{eq:63}
  \begin{split}
    \frac{\dd}{\dd t}\widetilde{V}_q(\Omega_t,z_t)
    & =
   q \int_{|v|<1}
    \frac{\rho(\bar{v},t)^{n-1} \rho'(\bar{v},t) \dd v}
    {| \rho(\bar{v},t)v -z_t |^{n-q}} \\
    &\hskip1.2em +
    \frac{q(q-n)}{n}
    \int_{|v|<1}
    \bigl\langle
    \rho'(\bar{v},t) v -z_t',
    \frac{\rho(\bar{v},t)v -z_t}{| \rho(\bar{v},t)v -z_t |}
    \bigr\rangle
    \frac{\rho(\bar{v},t)^n \dd v}
    {| \rho(\bar{v},t)v -z_t |^{n-(q-1)}}.
  \end{split}
\end{equation}
By virtue of \eqref{Vade},
one can  see that \eqref{eq:63} is true.
Now by \eqref{eq:52} again, noting $\bar{v}=\tfrac{y}{|y|}$ and
$v=\tfrac{y}{\rho(\bar{v},t)}$, we conclude that \eqref{eq:63} is just \eqref{eq:47}.
The proof of this lemma is complete.
\end{proof}

When $z_t$ is given in terms of the support function, the variational formula \eqref{eq:47}
can be also derived in terms of the support function.

\begin{lemma}\label{lem002}
Assume $q>1$, and
$h(\cdot,t)$ is the support function of $\Omega_t$. For
$t\in(t_0,t_1)$,
  %and $\pd\Omega\in C^1$.
and for any $\overline{\nabla}h(x,t)\in\pd\Omega_t$, then we have %//TODO
\begin{equation} \label{eq:51}
\begin{split}
  \frac{\dd}{\dd t}
  \widetilde{V}_q(\Omega_t,\delbar h(x,t))
  &=
 q \int_{\Omega_t}
  \frac{h'(\tilde{x},t)}{h( \tilde{x},t)}
  \frac{\dd y}{|y-\delbar h(x,t)|^{n-q}} \\
  &\quad+
  \frac{q(q-n)}{n}
  \int_{\Omega_t}
  \Bigl\langle
  \frac{h'( \tilde{x},t)}{h(\tilde{x},t)} y -\delbar h( x,t)',
  \tfrac{y-\delbar h(x,t)}{|y-\delbar h(x,t)|}
  \Bigr\rangle
  \frac{\dd y}{|y-\delbar h(x,t)|^{n-(q-1)}},
  \end{split}
\end{equation}
where $\tilde{x}:= \nu_{\ssOmega_t}\bigl( \rho(y,t)y \bigr)$, $\delbar h( x,t)'$ and $h^{'}(\cdot, t)$ denote the derivatives of $\delbar h( x,t)$ and $h(\cdot,t)$ with respect to $t$ respectively.
\end{lemma}

\begin{proof}{}
By the assumption that $\pd\Omega_t$ is $C^1$ smooth and strictly convex, one
can link the support function $h(\cdot,t)$ and the radial function $\rho(\cdot,t)$ through the
 relationship:
\begin{equation}\label{eq:37}
\nabla h(x,t) +h(x,t)x =\rho(u,t)u,
\end{equation}
where $x$ can be viewed as a function of $u$ and $t$. From \cite[Lemma 2.1]{LSW20}, we find the following formula:
\begin{equation} \label{eq:45}
  \frac{h'(x,t)}{h(x,t)}
  = \frac{\rho'(u,t)}{\rho(u,t)},
\end{equation}
where $x$ and $u$ are related via \eqref{eq:37}.
 Hence, \eqref{eq:51} holds by substituting \eqref{eq:45} into \eqref{eq:47}.
\end{proof}

\section{Uniform bounds for principal curvatures}
\label{sec6}

Based on the uniform bounds of $h$ and $\nabla h$ obtained in Section
\ref{sec4}, and the differentiation formulas of
$\widetilde{V}_q(\Omega_t,\delbar h)$ provided in Section \ref{sec5},
we are now in a position to prove uniform bounds for the principal curvatures of
$\pd\Omega_t$ along the flow \eqref{flow1}.

For convenience, the remainder of this part adopts the convention that repeated indices are automatically summed over. We first estimate the upper bound of Gauss curvature $\kappa$ of $\pd\Omega_t$.
\begin{lemma}\label{lem010}
Let $f$ be an even, positive, smooth function on $\sn$, and $\Omega_{t}$ be an origin symmetric, uniformly convex solution to the flow \eqref{flow}. Then there exists a positive constant $C$ independent of $t$ such that the Gauss curvature of $\partial\Omega_{t}$ satisfies
\begin{equation*}
  \kappa(x,t) \leq C,
  \quad \forall (x,t) \in \mathbb{S}^{n-1} \times [0, T).
\end{equation*}
\end{lemma}

\begin{proof}
Consider the following auxiliary function:
\begin{equation}\label{eq:44}
  Q(x,t) =\frac{-\pd_th}{h-\varepsilon_0},
\end{equation}
where $\varepsilon_0=\frac{1}{2} \inf_{\uS\times[0,T)}h(x,t)$
is a positive constant due to Lemma \ref{lem03}.
We will use the evolution equation of $Q$ to obtain its upper bound.

From \eqref{eq:44}, there holds
\begin{equation}\label{eq:48}
  \pd_tQ
  =
  \frac{(\pd_th)^2}{(h-\varepsilon_0)^2}
  +\frac{-\pd_{tt}^2h}{h-\varepsilon_0}
  =
  Q^2
  +\frac{-\pd_{tt}^2h}{h-\varepsilon_0}.
\end{equation}
Differentiating the flow \eqref{flow1}, we find
\begin{equation*}
  \begin{split}
    \pd_{tt}^2h
    &=
    -\frac{f\pd_t\kappa}{\widetilde{V}_{q-1}}
    +\frac{f\kappa\pd_t\widetilde{V}_{q-1}}{\widetilde{V}_{q-1}^2}
    +\pd_th \\
    &=
    -\frac{f\kappa}{\widetilde{V}_{q-1}}
    \left(
      \frac{\pd_t\kappa}{\kappa}
      -\frac{\pd_t\widetilde{V}_{q-1}}{\widetilde{V}_{q-1}}
    \right)
    +\pd_th \\
    &=
    (\pd_th-h)
    \left(
      \frac{\pd_t\kappa}{\kappa}
      -\frac{\pd_t\widetilde{V}_{q-1}}{\widetilde{V}_{q-1}}
    \right)
    +\pd_th.
  \end{split}
\end{equation*}
Inserting it into \eqref{eq:48}, and recalling \eqref{eq:44}, we obtain
\begin{equation}\label{eq:50}
  \pd_tQ
  =
  Q^2 +Q
  + \left(
    Q+\frac{h}{h-\varepsilon_0}
  \right)
  \left(
    \frac{\pd_t\kappa}{\kappa}
    -\frac{\pd_t\widetilde{V}_{q-1}}{\widetilde{V}_{q-1}}
  \right).
\end{equation}

Note that in order to estimate the upper bound of $Q$, one only needs to concern
$\pd_tQ$ at any point $(x_t,t)$ satisfying
\begin{equation} \label{eq:53}
  Q(x_t,t)=\max_{x\in\uS}Q(x,t)
  \quad
  \text{ and }
  \quad
  Q(x_t,t)\geq C_1\geq 1,
\end{equation}
where $C_1$ is a positive constant, which is independent of $t$ and to be determined.
At each such point, there are
\begin{equation}\label{eq:46}
  0=Q_i
  =\frac{-\pd_t h_i}{h-\varepsilon_0}
  +\frac{(\pd_th)h_i}{(h-\varepsilon_0)^2}
  =\frac{-\pd_t h_i-Qh_i}{h-\varepsilon_0},
\end{equation}
and
\begin{equation} \label{eq:49}
  0\geq Q_{ij}
  = \frac{-\pd_th_{ij}-Qh_{ij}}{h-\varepsilon_0},
\end{equation}
where the fact $Q_i=0$ has been used when computing \eqref{eq:49}, and $Q_{ij}\leq0$ should be
understood in the sense of negative semi-definite matrix.
These two equations will help us simplify the estimation of $\pd_t\kappa$ and
$\pd_t\widetilde{V}_{q-1}$ in the evolution equation \eqref{eq:50}.

We first estimate $\pd_t\widetilde{V}_{q-1}$.
In fact, its expression has already been given in Lemma \ref{lem002}.
By \eqref{eq:44}, for any $(x,t)$, if $\pd_t h\leq0$, there is 
\begin{equation*} 
  \frac{|\pd_th|}{h}(x,t)
  = \frac{(h-\varepsilon_0)Q}{h}(x,t)
  \leq Q(x_t,t);
\end{equation*}
if $\pd_t h>0$, by \eqref{flow1}, there is $\pd_t h<h$.
Therefore, for any $(x,t)$, we have
\begin{equation} \label{eq:54}
  \frac{|\pd_th|}{h}(x,t) \leq Q(x_t,t).
\end{equation}
%for any unit outer normal $x$, there is
%\begin{equation} \label{eq:54}
 % \frac{|\pd_th|}{h}(x,t)
 % = \frac{(h-\varepsilon_0)|Q|}{h}(x,t)
 % \leq |Q|(x,t)
  %\leq Q(x_t,t).
%\end{equation}
Recalling \eqref{eq:46}, we have
\begin{equation*}
  \begin{split}
    (\pd_t\delbar h)(x_t,t)
    &=(\pd_t\nabla h+\pd_t h \,x)(x_t,t) \\
    &=\bigl(
    -Q\nabla h -Q(h-\varepsilon_0)x
    \bigr)(x_t,t) \\
    &=
    -Q(x_t,t)\left(
      \delbar h(x_t,t) -\varepsilon_0 x_t
    \right).
  \end{split}
\end{equation*}
Then, for any $y\in\Omega_t$, we have
\begin{equation} \label{eq:56}
  \begin{split}
    \abs{\frac{\pd_th}{h}(x,t)y-(\pd_t\delbar h)(x_t,t)}
    &\leq
    \frac{|\pd_th(x,t)|}{h}|y|
    + \abs{(\pd_t\delbar h)(x_t,t)} \\
    &\leq
    C Q(x_t,t) + Q(x_t,t) (C+\varepsilon_0) \\
    &\leq
    3C Q(x_t,t),
  \end{split}
\end{equation}
where $C$ is the same constant in Lemma \ref{lem03}.
Based on \eqref{eq:54} and \eqref{eq:56},  from \eqref{eq:51},
 we find
\begin{equation*}
  |(\pd_t \widetilde{V}_{q-1})(x_t,t)|
  \leq
  (q-1)\, Q(x_t,t)
  \left(
    \int_{\Omega_t}
    \frac{\dd y}{|y-\delbar h|^{n+1-q}}
    +
    \frac{3C|q-1-n|}{n}
    \int_{\Omega_t}
    \frac{\dd y}{|y-\delbar h|^{n+2-q}}
  \right).
\end{equation*}
Combining this with Lemma \ref{lem601} and Lemma \ref{lem03}, we obtain
\begin{equation} \label{eq:58}
  \abs{\frac{\pd_t\widetilde{V}_{q-1}}{\widetilde{V}_{q-1}}(x_t,t)}
  \leq C_2 Q(x_t,t),
\end{equation}
where $C_2$ is a positive constant depending only on $n$, $q$, and the constants
$C$  in Lemma \ref{lem03}.

We now estimate $\pd_t\kappa$.
Recall the notation $b_{ij}$ given in \eqref{bij}, and let $b^{ij}$ be its
inverse matrix.
Since $\kappa=1/\det(b_{ij})$, at $(x_t,t)$, we have
\begin{equation}\label{eq:59}
  \begin{split}
    \frac{\pd_t\kappa}{\kappa}
    = -b^{ji}\pd_tb_{ij}
    &= -b^{ji}(\pd_th_{ij}+\pd_th\delta_{ij}) \\
    &\leq Q b^{ji}h_{ij} -\pd_th \TR(b^{ji}) \\
    &= Q[(n-1)-h \TR(b^{ji})]
    + Q(h-\varepsilon_0) \TR(b^{ji}) \\
    &= Q [(n-1) -\varepsilon_0 \TR(b^{ji})],
  \end{split}
\end{equation}
where the inequality follows from \eqref{eq:49}, and the next equality follows from
\eqref{bij} and \eqref{eq:44}.
Observing that
\begin{equation*}
  \frac{\TR(b^{ji})}{n-1}
  \geq \det(b^{ji})^{\frac{1}{n-1}}
  =\kappa^{\frac{1}{n-1}}.
\end{equation*}
By \eqref{eq:59}, we obtain
\begin{equation} \label{eq:60}
  \frac{\pd_t\kappa}{\kappa}
  \leq
  \left(
    (n-1)
    -\varepsilon_0 (n-1) \kappa^{\frac{1}{n-1}}
  \right)Q.
\end{equation}
Recalling the definition of $Q$ in \eqref{eq:44}, and the flow \eqref{flow1},
there holds
\begin{equation*}
  (h-\varepsilon_0)Q +h
  =
  \frac{f}{\widetilde{V}_{q-1}} \kappa,
\end{equation*}
which, together with Lemmas \ref{lem03} and \ref{lem04}, shows
\begin{equation}\label{eq:55}
  C_3 Q(x,t)
  \leq
  \kappa(x,t)
  \leq
  C_4 [Q(x,t)+1],
\end{equation}
where $C_3$ and $C_4$ are positive constants depending only on
the upper and lower bounds of $f$,
and the constant $C$ in Lemma \ref{lem03}.
Inserting \eqref{eq:55} into \eqref{eq:60}, we find
\begin{equation} \label{eq:57}
  \frac{\pd_t\kappa}{\kappa}
  \leq
  \left(
  (n-1)
  -\varepsilon_0 (n-1) C_3^{\frac{1}{n-1}} Q^{\frac{1}{n-1}}
  \right)Q.
\end{equation}

Now, by the estimates \eqref{eq:57} and \eqref{eq:58}, the evolution equation
\eqref{eq:50} at point $(x_t,t)$ reads
\begin{equation}\label{eq:62}
  \pd_tQ
  \leq
  Q^2 +Q
  + \left(
    Q+\frac{h}{h-\varepsilon_0}
  \right) \left(
    (C_2+n-1)
    -\varepsilon_0 (n-1) C_3^{\frac{1}{n-1}} Q^{\frac{1}{n-1}}
  \right)Q.
\end{equation}
Selecting $C_1$ in \eqref{eq:53} such that
\begin{equation*}
  \varepsilon_0 (n-1) C_3^{\frac{1}{n-1}} C_1^{\frac{1}{n-1}}
  =2(C_2+n-1),
\end{equation*}
there is
\begin{equation*}
  \begin{split}
    (C_2+n-1)
    -\varepsilon_0 (n-1) C_3^{\frac{1}{n-1}} Q^{\frac{1}{n-1}}
    &\leq
    (C_2+n-1)
    -\varepsilon_0 (n-1) C_3^{\frac{1}{n-1}} C_1^{\frac{1}{n-1}} \\
    &= -(C_2+n-1).
  \end{split}
\end{equation*}
Therefore, \eqref{eq:62} becomes
\begin{equation*}
  \begin{split}
    \pd_tQ
    &\leq
    Q^2 +Q -(C_2+n-1) \left(
      Q+\frac{h}{h-\varepsilon_0}
    \right) Q \\
    &<
    Q^2 +Q -(C_2+n-1) \left(Q+1 \right) Q \\
    &=
    -(C_2+n-2) \left(Q+1 \right) Q.
  \end{split}
\end{equation*}
Thus, we have established that
$(\pd_tQ)(x_t,t)<0$ whenever $Q(x_t,t)\geq C_1$.
Without loss of generality, assume $Q(x_0,0)\leq C_1$.
By \cite[Section 3]{Ham86},
$Q(x_t,t)\leq C_1$ for any $t\in[0,T)$. 
Namely
\begin{equation*}
  Q(x,t) \leq C_1,
  \quad
  \forall (x,t) \in \mathbb{S}^{n-1} \times [0, T).
\end{equation*}
Recalling \eqref{eq:55}, we obtain the conclusion of this lemma.
\end{proof}

Combining with this lemma, one can derive uniform lower and upper bounds for the principal curvatures
$\kappa_{i}(x,t)$ of $\pd\Omega_t$ with $i=1,\cdots, n-1$ as follows.

\begin{lemma}\label{lem06}
 Let $f$ be an even, positive, smooth function on $\sn$, and $\Omega_{t}$ be an origin symmetric, smooth, uniformly convex solution to the flow \eqref{flow}. Then there is a positive constant $C$ independent of $t$ such that the principal curvatures $\kappa_{1}, \kappa_{2}, \ldots, \kappa_{n-1}$ of $\partial\Omega_{t}$ satisfy
  \begin{equation}\label{eq:64}
    1/C \leq \kappa_{i}(x,t) \leq C,
    \quad \forall (x,t)\in\uS\times[0,T).
  \end{equation}
\end{lemma}

\begin{proof}
 To establish Lemma \ref{lem06}, based on Lemma \ref{lem010}, we only need to prove the  upper bound of the principal radii of curvature.  Consider the auxiliary function,
\begin{equation*}
  \Lambda(x, t)=\log \TR b-\tilde{A}\log h+\tilde{B} |\nabla h|^2,
  \quad
  \forall (x,t)\in\uS\times[0,T),
\end{equation*}
where $\TR b$ is the trace of the matrix $\{b_{ij}\}$ with $b_{ij}=h_{ij}+h\delta_{ij}$,  $\tilde{A}$ and $\tilde{B}$ are positive
constants to be chosen later. For any fixed $T'\in (0,T)$, assume $\Lambda(x,t)$ attains its maximum at $(x_0,t_0)\in \sn \times [0,T']$. By choosing a suitable orthonormal frame, we may assume $\{b_{ij}\}(x_0, t_0)$ is
diagonal. Then, at the point $(x_0,t_0)$, we have
\begin{equation}\label{eq:68}
  0=\Lambda_i =
  \frac{1}{\TR b}\sum_k b_{kk,i}
  -\frac{\tilde{A}h_i}{h}+2\tilde{B}\sum_k h_k h_{ki},
\end{equation}
and
\begin{multline} \label{eq:69}
  0\geq \Lambda_{ij} =
  \frac{1}{\TR b}\sum_k b_{kk,ij}
  -\frac{1}{(\TR b)^2} \sum_k b_{kk,i} \sum_l b_{ll,j} \\
  -\frac{\tilde{A}h_{ij}}{h}+\frac{\tilde{A}h_i h_j}{h^2}
  +2\tilde{B}\sum_k (h_{ki} h_{kj}+h_kh_{kij}),
\end{multline}
where $\Lambda_{ij}\leq0$ means that it is a negative semi-definite matrix.
Here, $b_{kk, i}$ and $b_{kk, ij}$ represent the first-order and second-order covariant derivatives of $b_{kk}$, respectively.
Without loss of generality, we may assume $t_0>0$.
Then, at $(x_0,t_0)$, we also have
\begin{equation} \label{eq:70}
  0\leq\pd_t\Lambda=
  \frac{1}{\TR b}\sum_k (\pd_t h_{kk} +\pd_th)
  -\frac{\tilde{A}\pd_th}{h} +2\tilde{B}\sum_k h_k\pd_th_k.
\end{equation}
We now make use of \eqref{eq:68}--\eqref{eq:70} and the evolution equation
\eqref{flow1} to derive the uniform upper bound estimate as below.

By \eqref{flow1}, there is
\begin{equation*}
  \log(h-\pd_th)
  =
  \log\kappa
  +\log f
  -\log\widetilde{V}_{q-1}.
\end{equation*}
Denoting $\beta(x,t)= \log f -\log\widetilde{V}_{q-1}$
and differentiating  both sides yields
\begin{equation}\label{eq:72}
  \frac{h_k-\pd_th_k}{h-\pd_th}
  = -b^{ji}b_{ij,k}+\beta_k,
\end{equation}
and
\begin{equation}\label{eq:73}
  \frac{h_{kk}-\pd_th_{kk}}{h-\pd_th}
  = \left(
    \frac{h_k-\pd_th_k}{h-\pd_th}
  \right)^2
  -b^{ji}b_{ij,kk}
  +b^{jr}b^{si}b_{rs,k}b_{ij,k}
  +\beta_{kk},
\end{equation}
where the equality $b^{ji}_{,k}=-b^{jr}b^{si}b_{rs,k}$ has been used.
Recalling \eqref{eq:70}, we carry out the following calculations:
\begin{equation*}
\begin{split}
 & 2\tilde{B}\sum_k \frac{h_k(\pd_th_k-h_k)}{h-\pd_th}
  + \frac{2\tilde{B}|\nabla h|^2}{h-\pd_th}
  +\frac{\tilde{A}}{h}
  -\frac{\tilde{A}}{h-\pd_th}
   \\
& =2\tilde{B}\sum_k \frac{h_k\pd_th_k}{h-\pd_th}
  -\frac{\tilde{A}\pd_th}{h(h-\pd_th)}\\
 & \geq
  \frac{1}{\TR b}\sum_k \frac{-\pd_t h_{kk} -\pd_th}{h-\pd_th} \\
  &=
  \frac{1}{\TR b}\sum_k \frac{h_{kk}-\pd_t h_{kk}}{h-\pd_th}
  + \frac{n-1}{\TR b}
  - \frac{1}{h-\pd_th}.
  \end{split}
\end{equation*}
\eqref{eq:72} and \eqref{eq:73} tell
\begin{equation}\label{eq:74}
  \begin{split}
  2\tilde{B}\sum_k h_k&(b^{ji}b_{ij,k}-\beta_k)
  + \frac{2\tilde{B}|\nabla h|^2}{h-\pd_th}
  +\frac{\tilde{A}}{h}
  -\frac{\tilde{A}-1}{h-\pd_th} \\
  &\geq
  \frac{1}{\TR b} \sum_k \left(
    \left(
      b^{ji}b_{ij,k}-\beta_k
    \right)^2
    -b^{ji}b_{ij,kk}
    +b^{jr}b^{si}b_{rs,k}b_{ij,k}
    +\beta_{kk}
  \right)
  + \frac{n-1}{\TR b} \\
  &\geq
  \frac{1}{\TR b} \sum_k \left(
    -b^{ji}b_{ij,kk}
    +b^{jr}b^{si}b_{rs,k}b_{ij,k}
    +\beta_{kk}
  \right)
  + \frac{n-1}{\TR b}.
\end{split}
\end{equation}
In addition, by the Ricci identity, there is
\begin{equation*}
  b_{kk, ii} =b_{ii, kk}+b_{kk}-b_{ii}.
\end{equation*}
Since
$\{b_{ij}\}(x_0, t_0)$ is diagonal, and by virtue of \eqref{eq:69} and the Ricci identity,
there is $b^{ij}\Lambda_{ij}\leq0$, namely
\begin{equation} \label{eq:75}
  \begin{split}
    \frac{\tilde{A}b^{ii}h_{ii}}{h}-\frac{\tilde{A}b^{ii}h_i^2}{h^2}
    &-2\tilde{B}\sum_k b^{ii}\left(h_{ki}^2+h_kh_{kii}\right) \\
    &\geq
    \frac{1}{\TR b}\sum_{k,i} b^{ii}b_{kk,ii}
    -b^{ii}\left(\frac{\sum_k b_{kk,i}}{\TR b}\right)^2 \\
    &=
    \frac{1}{\TR b} \sum_{k,i} b^{ii}b_{ii, kk}
    -b^{ii}\left(\frac{\sum_k b_{kk,i}}{\TR b}\right)^2
    +\TR b^{-1}
    -\frac{(n-1)^2}{\TR b}.
  \end{split}
\end{equation}
By combining \eqref{eq:74} and \eqref{eq:75}, it yields
\begin{multline}\label{eq:76}
  2\tilde{B}\sum_k h_k(b^{ii}b_{ii,k}-\beta_k)
  + \frac{2\tilde{B}|\nabla h|^2}{h-\pd_th}
  +\frac{\tilde{A}}{h}
  -\frac{\tilde{A}-1}{h-\pd_th} \\
  +\frac{\tilde{A}b^{ii}h_{ii}}{h}-\frac{\tilde{A}b^{ii}h_i^2}{h^2}
  -2\tilde{B}\sum_k b^{ii}\left(h_{ki}^2+h_kh_{kii}\right)
  +\frac{(n-1)(n-2)}{\TR b} \\
  \geq
  \frac{1}{\TR b} \sum_k \left(
    b^{ii}b^{jj}b_{ij,k}^2
    +\beta_{kk}
  \right)
  -b^{ii}\left(\frac{\sum_k b_{kk,i}}{\TR b}\right)^2
  +\TR b^{-1}.
\end{multline}
Note that
\begin{gather*}
  b^{ii}h_{ii}
  =b^{ii}(b_{ii}-h)
  =n-1-h\TR b^{-1}, \\
  \sum_{k}b^{ii}h_{ki}^2
  =b^{ii}h_{ii}^2
  =b^{ii}(b_{ii}^2-2hb_{ii}+h^2)
  =\TR b -2(n-1)h+h^2\TR b^{-1},  \\
  \sum_{k}b^{ii}h_kh_{kii}
  =\sum_{k}b^{ii}h_k(b_{ki,i}-h_i\delta_{ki})
  =\sum_kh_kb^{ii}b_{ii,k} -b^{ii}h_i^2,
\end{gather*}
where the fact that $b_{ij,k}$ is symmetric in all indices is utilized. Then
\eqref{eq:76} turns into
\begin{multline}\label{eq:71}
  -2\tilde{B}\langle \nabla h,\nabla \beta \rangle
  + \frac{2\tilde{B}|\nabla h|^2}{h-\pd_th}
  +\frac{n\tilde{A}}{h}
  -\frac{\tilde{A}-1}{h-\pd_th}
  -\tilde{A}\TR b^{-1} \\
  \hskip1.6em
  +\frac{2\tilde{B}h^2-\tilde{A}}{h^2}b^{ii}h_i^2
  -2\tilde{B}\left(\TR b -2(n-1)h+h^2\TR b^{-1}\right)
  +\frac{(n-1)(n-2)}{\TR b} \\
  \geq
  \frac{1}{\TR b} \sum_k \left(
    b^{ii}b^{jj}b_{ij,k}^2
    +\beta_{kk}
  \right)
  -b^{ii}\left(\frac{\sum_k b_{kk,i}}{\TR b}\right)^2
  +\TR b^{-1}.
\end{multline}
For each $i$, there is
\begin{equation*}
  \begin{split}
    \TR b \sum_{j,k} b^{jj}b_{ij,k}^2
    &\geq
    \biggl(\sum_k b_{kk} \biggr)
    \biggl(\sum_k b^{kk}b_{ik,k}^2 \biggr) \\
    &\geq
    \biggl(\sum_k \sqrt{ b_{kk} } \sqrt{ b^{kk}  b_{ik,k}^2 } \biggr)^2 \\
    &=
    \biggl(\sum_k |b_{kk,i}| \biggr)^2 \\
    &\geq
    \biggl(\sum_k b_{kk,i} \biggr)^2.
  \end{split}
\end{equation*}
Using this inequality, thus \eqref{eq:71} says
\begin{multline*}
  \frac{2\tilde{B}|\nabla h|^2+1-\tilde{A}}{h-\pd_th}
  +\frac{n\tilde{A}}{h}
  +4(n-1)\tilde{B}h
  +\frac{2\tilde{B}h^2-\tilde{A}}{h^2}b^{ii}h_i^2
  +\frac{(n-1)(n-2)}{\TR b} \\
  \geq
  \frac{\Delta \beta}{\TR b}
  +2\tilde{B}\langle \nabla h,\nabla \beta \rangle
  +2\tilde{B}  \TR b
  +(1+\tilde{A}+2\tilde{B}h^2)\TR b^{-1}.
\end{multline*}
If we choose
\begin{equation}\label{AB}
  \tilde{A}=2\tilde{B}C^2+1,
\end{equation}
where $C$ is the constant in Lemma \ref{lem03}, then
we have
\begin{equation}\label{eq:83}
  \frac{\Delta \beta}{\TR b}
  +2\tilde{B}\langle \nabla h,\nabla \beta \rangle
  +2\tilde{B}\TR b
  +2\tilde{B}\TR b^{-1}
  \leq
  nC(\tilde{A}+4\tilde{B})
  +\frac{(n-1)(n-2)}{\TR b}.
\end{equation}

To deal with the terms involving $\widetilde{V}_{q-1}$, we need Lemma \ref{lem008}.
For simplicity, let $c^k(x,t)$, $c(x,t)$ and $c^{kl}(x,t)$ be the
integrals appearing in \eqref{eq:41} and \eqref{eq:36}.
Then Lemma \ref{lem008} tells
\begin{equation}\label{eq:81}
  \begin{gathered}
    \nabla_i \widetilde{V}_{q-1} = b_{ik} c^k, \\
    \nabla_{ij}^2 \widetilde{V}_{q-1}
    = b_{ijk} c^k +b_{ij} c +b_{ik}b_{jl} c^{kl}.
  \end{gathered}
\end{equation}
Here we note that the coefficients of $b_{ik}$ and $b_{ijk}$ are same, which
is crucial for establishing the following estimates.
With the aid of Lemma \ref{lem03}, we know that there exists a positive
constant $C$, independent of $t$, such that
\begin{equation}\label{eq:78}
  \sup (|c^k| +|c|+|c^{kl}|) \leq C.
\end{equation}

Recalling $\beta(x,t)= \log f -\log\widetilde{V}_{q-1}$, we have
\begin{equation}\label{eq:79}
  \beta_k
  =\frac{f_k}{f}-\frac{\nabla_k\widetilde{V}_{q-1}}{\widetilde{V}_{q-1}},
\end{equation}
and
\begin{equation}\label{eq:80}
  \beta_{kk}
  =
  \frac{f_{kk}}{f}
  - \frac{f_k^2}{f^{2}}
  -\frac{\nabla^2_{kk}\widetilde{V}_{q-1}}{\widetilde{V}_{q-1}}
  +\frac{(\nabla_{k}\widetilde{V}_{q-1})^2}{\widetilde{V}_{q-1}^2}.
\end{equation}
Moreover,
\begin{equation*}
  \begin{split}
    \Delta \beta
    &\geq
    -C_f
    -\frac{1}{\widetilde{V}_{q-1}} \left(\sum_k\nabla^2_{kk}\widetilde{V}_{q-1} \right) \\
    &=
    -C_f
    -\frac{1}{\widetilde{V}_{q-1}} \left(
      \sum_k b_{kkl} c^l +c \TR b +\sum_k b_{kk}^2 c^{kk}
    \right) \\
    &\geq
    -C_f
    -\frac{1}{\widetilde{V}_{q-1}} \sum_k b_{kkl} c^l
    -C_1\TR b
    -C_1 (\TR b)^2,
  \end{split}
\end{equation*}
and
\begin{equation*}
  \begin{split}
    2\tilde{B}\langle \nabla h,\nabla \beta \rangle
    &=
    \frac{2\tilde{B}\langle \nabla h,\nabla f \rangle}{f}
    -\frac{2\tilde{B} h_l b_{ll} c^l}{\widetilde{V}_{q-1}} \\
    &\geq
    -C_2\tilde{B}
    -\frac{2\tilde{B} h_l h_{ll} c^l}{\widetilde{V}_{q-1}}. \\
  \end{split}
\end{equation*}
Thus, together with \eqref{eq:68}, we find
\begin{equation}\label{eq:84}
  \begin{split}
    \frac{\Delta \beta}{\TR b}
    +2\tilde{B}\langle \nabla h,\nabla \beta \rangle
    &\geq
    -\frac{C_f}{\TR b} -C_1 -C_1\TR b -C_2\tilde{B}
    -\frac{c^l}{\widetilde{V}_{q-1}} \left(
      \frac{\sum_k b_{kkl}}{\TR b}
      +2\tilde{B} h_l h_{ll}
    \right) \\
    &=
    -\frac{C_f}{\TR b} -C_1 -C_1\TR b -C_2\tilde{B}
    -\frac{\tilde{A}h_lc^l}{h\widetilde{V}_{q-1}} \\
    &\geq
    -\frac{C_f}{\TR b} -C_1 -C_1\TR b -C_2\tilde{B} -C_3\tilde{A}.
  \end{split}
\end{equation}
Employing this inequality, \eqref{eq:83} becomes 
\begin{equation}\label{trb}
  (2\tilde{B}-C_1)\TR b
  +2\tilde{B}\TR b^{-1}
  \leq
  C_4(1+\tilde{A}+\tilde{B})
  +\frac{(n-1)(n-2)+C_f}{\TR b},
\end{equation}
where $C_4$ depends only on $n$, $C$, $C_1$, $C_2$ and $C_3$. Choose a suitable $\tilde{B}$ with satisfying $\tilde{B}>C_{1}$. So \eqref{trb} implies that $\TR b$ is uniformly bounded, it follows that the maximum of principal radii of curvature is uniformly bounded. This completes the proof.
\end{proof}

\section{Smooth solutions to the Chord log-Minkowski problem}
\label{sec7}

This section is devoted to the proof of Theorem \ref{thm4}. The key step is to analyze the long-time behavior and establish the convergence of solutions to the nonlocal Gauss curvature flow equation $\eqref{flow1}$. 

Since the nonlocal potential in $\eqref{flow1}$ requires additional care in passing from the curvature bounds to uniform space-time regularity, we here explain the bootstrap argument in detail.  We first obtain a uniform H\"older estimate for $h_t$ by  controlling its nonlocal term. Elliptic regularity and parabolic Schauder estimates then yield uniform bounds for all spatial derivatives of $h(\cdot,t)$. Finally, we establish estimates for the time derivatives of the potential and complete the bootstrap. The specific steps are as follows.

For simplicity, put $a=q-1>2$, $d=n-1$, and write
\begin{equation*}\label{r7:notation}
 \begin{gathered}
 X(x,t)=\nabla h(x,t)+h(x,t)x,\quad
 b=(b_{ij})=(h_{ij}+h\delta_{ij}),\\
V(x,t)=\widetilde V_a(\Omega_t,X(x,t)),\quad
 F=\frac{f}{ V\det b}=h-h_t.
 \end{gathered}
\end{equation*}
The uniform a priori estimates for the support function and the principal curvatures established in 
 Sections \ref{sec4} and \ref{sec6} imply that $\eqref{flow1}$ is uniformly
 parabolic. So for constants $r,R,\lambda,\Lambda,c,C>0$, independent of
$t<T$, there are
\begin{equation}\label{r7:bounds}
 B_r\subset\Omega_t\subset B_R,\qquad
 \lambda I\le b\le\Lambda I,\qquad c\le  V,F\le C,
 \qquad \|h\|_{C^2}+\|h_t\|_{C^0}\le C.
\end{equation}

First, for $y\ne0$ in $\Omega_t$, set
$ \xi_t(y)=\nu_{\Omega_t}\!\left(\rho(y/|y|,t)y/|y|\right).
$
 Define, with $z=X(x,t)$,
\begin{align*}
 \Re(x,t)&=\frac{a(n-a)}{n}\int_{\Omega_t}
       (y-z)|y-z|^{a-n-2}\,\dd y,\\
 \mathcal N_t[v](x)&=a\int_{\Omega_t}\frac{v(\xi_t(y),t)}{h(\xi_t(y),t)}|y-z|^{a-n}\,\dd y\notag\\
 &\quad+\frac{a(a-n)}n\int_{\Omega_t}
      \frac{v(\xi_t(y),t)}{h(\xi_t(y),t)}\langle y,y-z\rangle|y-z|^{a-n-2}\,\dd y.
\end{align*}
Lemma \ref{lem002}, with its index replaced by $a$, gives, for $w=h_t$,
\begin{equation*}\label{r7:variation}
  V_t=\mathcal N_t[w]+\Re\cdot(\nabla w+wx).
\end{equation*}

For $s>0$ and $z\in\Omega_t$,
\[
 \int_{\Omega_t}|y-z|^{s-n}\,\dd y
 \le\frac{n\omega_n}{s}(2R)^s.
\]
Since $\big|\frac{v(\xi_t(y),t)}{h(\xi_t(y),t)}\big|\le r^{-1}\|v\|_\infty$ and $|y|\le R$, applying this
estimate with $s=a$ and $s=a-1$ proves
\begin{equation}\label{r7:operatorbound}
 \|\Re\|_\infty\le C,\qquad
 \|\mathcal N_t[v]\|_\infty\le C\|v\|_\infty.
\end{equation}
Moreover, the divergence theorem gives
\[
 \mathcal N_t[v](x)=\frac an\int_{\partial\Omega_t}
      v(\nu_{\Omega_t}(y),t)|y-X(x,t)|^{a-n}\,\dd \mathcal H^{n-1}(y).
\]

Now differentiating $h_t=h-F$ and using
$\partial_t\log\det b=b^{ij}(w_{ij}+w\delta_{ij})$, we obtain
\begin{equation}\label{r7:w}
 w_t=a^{ij}w_{ij}+\beta^iw_i+cw+\mathcal L_t[w],
\end{equation}
where
\begin{equation*}\label{r7:coeff}
 \begin{split}
 a^{ij}=Fb^{ij},  \quad 
 \beta^i&=(F/ V)\Re\cdot e_i,\\
 c=1+F\operatorname{tr}(b^{-1})+(F/ V)\Re\cdot x, \quad 
 \mathcal L_t[v]&=(F/ V)\mathcal N_t[v].
 \end{split}
\end{equation*}
Using \eqref{r7:bounds} and \eqref{r7:operatorbound}, we have
\begin{equation*}\label{r7:ellipticity}
 c_0|\eta|^2\le a^{ij}\eta_i\eta_j\le C_0|\eta|^2,
 \quad |\beta|+|c|\le C_0,\quad
 \|\mathcal L_t[v]\|_\infty\le C_0\|v\|_\infty.
\end{equation*}
 Since $w$ is bounded by \eqref{r7:bounds}, the term
$cw+\mathcal L_t[w]$ is bounded. Applying 
Krylov-Safonov interior H\"older estimate (see \cite{KS80,K87}) to \eqref{r7:w},
for a coordinate cylinder
$Q_s=\tilde{B}_s\times(t_0-s^2,t_0]$, where $\tilde{B}_s\subset\mathbb R^{n-1}$ denotes a Euclidean ball of radius $s$, we obtain
\[
 [w]_{C^{\alpha_0,\alpha_0/2}(Q_s)}
 \le C s^{-\alpha_0}\left(
       \|w\|_{L^\infty(Q_{2s})}
       +s^2\|cw+\mathcal L_t[w]\|_{L^\infty(Q_{2s})}\right).
\]
The constant depends on the ellipticity bounds and an upper bound for
 $s\|\beta\|_\infty$. For every $0<\tau<T$, there exist $\alpha_0\in(0,1)$ and $C_\tau$
such that
\begin{equation}\label{r7:wHolder}
 \|w\|_{C^{\alpha_0,\alpha_0/2}
       (\mathbb S^{n-1}\times[\tau,T))}\le C_\tau.
\end{equation}

Second, 
from $\|h_t\|_\infty\le C$ and the uniform $C^2$ bound,
spatial interpolation gives
\begin{equation}\label{r7:lowtime}
 \|h(\cdot,t)-h(\cdot,s)\|_\infty\le C|t-s|,\quad
 \|\nabla h(\cdot,t)-\nabla h(\cdot,s)\|_\infty
       \le C|t-s|^{1/2},\qquad (|t-s|\le1).
\end{equation}
By \eqref{eq:45}, the radial functions are Lipschitz with respect to time in $C^0$ norm. From 
\eqref{Vade}, since $a>1$, we find
\[
 |\widetilde V_a(\Omega_t,\rho(u,t)u)
    -\widetilde V_a(\Omega_s,\rho(u,s)u)|\le C|t-s|.
\]
Using \eqref{r7:lowtime} proves
\begin{equation}\label{r7:Vtime}
 \| V(\cdot,t)-V(\cdot,s)\|_\infty\le C|t-s|^{1/2}.
\end{equation}
Moreover, from \eqref{eq:41}, we find $V_i=\Re \cdot X_i$ with $X_i=b_{ij}e_{j}$. Thus, $V$ is uniformly Lipschitz in space and $1/2$ H\"older continuous in time.

At each fixed time, the flow equation can be written as
\begin{equation}\label{r7:elliptic}
 \det(\nabla ^2h+hI)^{1/d}
       =\left(\frac{f}{(h-w) V}\right)^{1/d}.
\end{equation}
Recall $h-w=F\ge c$.
By \eqref{r7:wHolder}, the uniform spatial Lipschitz bound
for $V$, and \eqref{r7:bounds}, the right-hand side has a
uniformly bounded spatial H\"older norm.
 We therefore use the Evans-Krylov
estimate with H\"older data (see, e.g., ~\cite{GT01,K87}).

Write
$\phi(P)=(\det P)^{1/d}$ and
$E=\{P\in\operatorname{Sym}(d): (\lambda/2)I\le P\le 2\Lambda I\}$, where $\operatorname{Sym}(d)$ denotes the space of real symmetric
$d\times d$ matrices.
Define the concave extension
\[
 \overline\phi(M)=\inf_{P\in E}
       \set{\phi(P)+ \TR\left((D\phi(P))^T (M-P)\right)},
       \qquad M\in\operatorname{Sym}(d),
\]
which equals $\phi$ on $E$, is globally uniformly elliptic, and is globally
Lipschitz.  Here $D\phi(P)$ denotes the derivative of $\phi$
with respect to the matrix variable at $P$. Indeed, the matrix
\[
 D\phi(P)=d^{-1}(\det P)^{1/d}P^{-1}
\]
has eigenvalues bounded above and below by positive constants throughout
$E$.
In a fixed local coordinate chart, the matrix of principal radii of curvature in an orthonormal frame can be written as
$T_x(D^2h)+S_t(x)$, where $T_x$ is a smooth, positive-definiteness-preserving congruence transformation determined by the spherical metric, and $S_t$ contains only $h$, $Dh$
and the fixed connection coefficients. The uniform $C^2$ bound for $h$
gives a uniform Lipschitz bound for $S_t$ in space. Consequently
\[
 \Phi_t(M,x)=\overline\phi(T_x(M)+S_t(x))
\]
is concave and uniformly elliptic in $M$, and satisfies
$|\Phi_t(M,x)-\Phi_t(M,y)|\le C(1+|M|)|x-y|$.
It coincides with the left-hand side of \eqref{r7:elliptic} when evaluated at the Hessian of the solution. Thus the H\"older interior estimate, followed by a finite
covering, gives some $\alpha\in(0,1)$ such that
\begin{equation}\label{r7:C2a}
 \sup_{\tau\le t<T}\|h(\cdot,t)\|_{C^{2,\alpha}}
       \le C_\tau.
\end{equation}
 In particular, interpolating
\eqref{r7:C2a} with \eqref{r7:lowtime} gives
\[
 \|\nabla^2h(\cdot,t)-\nabla^2h(\cdot,s)\|_\infty
       \le C_\tau|t-s|^{\alpha/(2+\alpha)}.
\]
Similarly, interpolation gives
\[
 \|\nabla h(\cdot,t)-\nabla h(\cdot,s)\|_\infty
 \le C_\tau |t-s|^{(1+\alpha)/(2+\alpha)}.
\]
Choosing $0<\beta<\min\{\alpha_0,\alpha/(2+\alpha)\}$, the above
estimates and \eqref{r7:wHolder} give a uniform
$C^{2+\beta,1+\beta/2}$ bound for $h$ on interior time strips.

Next, for every integer $m\ge2$, uniform convexity, \eqref{eq:37} and Theorem \ref{thm2} imply that
\begin{equation}\label{r7:spatialpotential}
 \|h(\cdot,t)\|_{C^{m,\gamma}}\le M
 \quad\Longrightarrow\quad
 \|\rho(\cdot,t)\|_{C^{m,\gamma}}+\| V(\cdot,t)\|_{C^{m-1,\gamma}}\le C_M,
 \qquad 0<\gamma<1.
\end{equation}
Set the map
$x\mapsto U(x,t)=X(x,t)/|X(x,t)|$.
If $U(\cdot,t)$ has a uniform $C^{k,\gamma}$ bound and
$\|U(\cdot,t)-U(\cdot,s)\|_\infty\le C|t-s|^\theta$, then for $0\le j\le k$, interpolation gives
\begin{equation}\label{r7:interpolation}
 \|\nabla^jU(\cdot,t)-\nabla^jU(\cdot,s)\|_\infty
 \le C'|t-s|^{\theta(k+\gamma-j)/(k+\gamma)}.
\end{equation}
This follows by applying the spatial interpolation inequality to
$U(\cdot,t)-U(\cdot,s)$. In particular, every derivative up to
order $k$ has a positive time H\"older exponent.

Assume now that $h$ is uniformly $C^{m,\gamma}$ in space on a time
strip.  \eqref{r7:lowtime},
\eqref{r7:Vtime}, \eqref{r7:spatialpotential} and \eqref{r7:interpolation} give parabolic H\"older
bounds, with some positive exponent $\gamma_m$, for every spatial
derivative of $h$ of order at most $m$ and of $ V$ of order at most
$m-1$. Treat the already estimated function $ V(x,t)$ as a prescribed
coefficient when differentiating 
$h_t=h-f/( V\det b)$ in a fixed local coordinate chart. For any spatial derivative
$v=\partial^Ih$ of order $|I|=m-1$, differentiation yields
\begin{equation}\label{r7:spatiallinear}
 v_t=a^{ij}\partial_i\partial_jv+R_I.
\end{equation}
The spatial derivatives of $h$ of order $m+1$ appear only in
the principal term $a^{ij}\partial_i\partial_jv$.
The remainder $R_I$ involves only spatial derivatives of $h$
up to order $m$ and of $V$ up to order $m-1$, together with
fixed smooth coefficients coming from $f$ and the metric,
including the connection coefficients.  Choose
\[
 0<\gamma_m<\frac{\gamma}{m-1+\gamma}.
\]
By the preceding interpolation estimates, all spatial derivatives
of $h$ up to order $m$ and of $V$ up to order $(m-1)$ have uniform
$C^{\gamma_m,\gamma_m/2}$ bounds.  Thus $a^{ij}$ and
$R_I$ have uniform $C^{\gamma_m,\gamma_m/2}$ bounds.

The interior parabolic Schauder estimate applied to
\eqref{r7:spatiallinear} gives $v\in C^{2+\gamma_m,1+\gamma_m/2}$
on a slightly smaller strip, with a uniform bound. This improves the
spatial estimate for $h$ from $C^{m,\gamma}$ to
$C^{m+1,\gamma_m}$. Starting with \eqref{r7:C2a}, iteration proves
\begin{equation}\label{r7:allspace}
 \sup_{\tau\le t<T}\|h(\cdot,t)\|_{C^k}\le C_{k,\tau}
       \quad\text{for every integer }k\ge0.
\end{equation}
By the flow $\eqref{flow1}$ and \eqref{eq:45}, \eqref{r7:allspace} also gives uniform bounds for $h_t$ and $\rho_t$ in every
spatial $C^k$ norm. 

In addition, fix $0<\tau<T$ and put $J=[\tau,T)$. For a local smooth family of
rotations $R(u)$ satisfying $R(u)e_n=u$, polar coordinates give
\begin{equation*}\label{r7:fixedsingularity}
 \widehat V(u,t):=\widetilde V_a(\Omega_t,\rho(u,t)u)
 =\frac an\int_{\mathbb S^{n-1}}\int_0^1
 \rho(R(u)\eta,t)^n|Z|^{a-n}r^{n-1}\dd r\dd\eta,
\end{equation*}
where
\[
 Z=r\rho(R(u)\eta,t)R(u)\eta-\rho(u,t)u.
\]
The zero of $Z$ is fixed at $(r,\eta)=(1,e_n)$. Writing
$\delta=|r\eta-e_n|$, the uniform bi-Lipschitz bounds for the radial
map imply
\begin{equation*}\label{r7:distance}
 c\delta\le |Z|\le C\delta.
\end{equation*}
For integers $k,\ell\ge0$, set
\[
 M_{k+1,\ell}=\max_{0\le j\le\ell}\sup_{t\in J}
 \|\partial_t^j\rho(\cdot,t)\|_{C^{k+1}}.
\]
Note that $M_{k+1,1}<\infty$ for every $k$. If $M_{k+1,\ell}<\infty$, we claim that
\begin{equation}\label{r7:quantitativepotential}
 \max_{0\le j\le\ell}\sup_{t\in J}
 \|\partial_t^j\widehat V(\cdot,t)\|_{C^k}
 \le C(k,\ell,M_{k+1,\ell}),
\end{equation}
where the constant also depends on the fixed parameters and uniform
geometric bounds.

Indeed, $Z(u,1,e_n,t)=0$ identically, so the mean-value estimate in
the integration variables gives
\begin{equation*}\label{r7:cancellation}
 |\partial_u^I\partial_t^jZ|\le C\delta,
 \qquad |I|\le k,\quad j\le\ell.
\end{equation*}
Each $N$-th derivative of $|Z|^{a-n}$ is bounded by
$C_N|Z|^{a-n-N}$, while its $N$ chain-rule factors contribute
$O(\delta^N)$. The assumed bounds on $\rho$ imply
\[
 \left|
 \partial_u^I\partial_t^j
 \bigl[\rho(R(u)\eta,t)^n\bigr]
 \right|\le C,
 \qquad |I|\le k,\quad j\le\ell.
\]
Combining this with the preceding estimates for the derivatives
of $|Z|^{a-n}$, Leibniz's rule gives
\[
 \left|
 \partial_u^I\partial_t^j
 \left[
 \rho(R(u)\eta,t)^n |Z|^{a-n}r^{n-1}
 \right]
 \right|
 \le C\delta^{a-n}r^{n-1},
 \qquad |I|\le k,\quad j\le\ell.
\]
Here $C$ is independent of the integration variables $r,\eta$
and of $t\in J$. This is integrable because $a>0$.
Differentiation under the integral proves
\eqref{r7:quantitativepotential}.

We now induct on the number of time derivatives. Suppose
$\partial_t^jh$ is bounded in every spatial $C^m$ norm for
$0\le j\le\ell$. The same holds for $\partial_t^j\rho$. Taking $m=k+1$.
\eqref{r7:quantitativepotential} and the identity
$V(x,t)=\widehat V(X(x,t)/|X(x,t)|,t)$ then bound
$\partial_t^jV$ for $j\le\ell$ in every spatial $C^k$ norm.
Differentiating $h_t=h-f/(V\det b)$ exactly $\ell$ times gives the
corresponding bounds for $\partial_t^{\ell+1}h$.
Hence
\begin{equation}\label{r7:allmixed}
 \sup_{\tau\le t<T}
 \|\partial_t^\ell h(\cdot,t)\|_{C^k}
 \le C_{k,\ell,\tau},
 \qquad k,\ell\ge0.
\end{equation}

Based on the above results, we are now in a position to prove Theorem \ref{thm4}.

\begin{proof}[Proof of Theorem \ref{thm4}]
If $T<\infty$,  \eqref{r7:allmixed}  yields that
$\|h(\cdot,t)-h(\cdot,s)\|_{C^k}\le C_{k,1,\tau}|t-s|$ for every $k$ and
$s,t\in[\tau,T)$. Hence there is a smooth limiting support function
as $t\uparrow T$.
The uniform bounds in \eqref{r7:bounds} persist at $T$.  The local existence result therefore extends the solution beyond $T$,
contradicting the maximality of $[0,T)$.
Hence $T=\infty$. 

 Set
\[
 A(x,t)=\frac{h(x,t) V(x,t)}{\kappa(x,t)}
        =h(x,t) V(x,t)\det b(x,t),\qquad
 M=\int_{\mathbb S^{n-1}}f\dd x.
\]
Lemma \ref{lem004} and \eqref{c-i-formula} imply $\int_{\mathbb S^{n-1}} A\dd x=M$. Consequently, \eqref{eq:28} gives
the exact identity
\begin{equation}\label{r7:dissipation}
 -J'(t)=\int_{\mathbb S^{n-1}}\frac{f^2}{A}\dd x-M
       =\int_{\mathbb S^{n-1}}\frac{(f-A)^2}{A}\dd x.
\end{equation}
Since $h$ is bounded above and below, $J(t)$ is
bounded. Combining with Lemma \ref{lem005}, we conclude that $-J'\ge0$ is integrable over $[0,\infty)$, and there
exist $t_j\to\infty$ such that $-J'(t_j)\to0$.
 \eqref{r7:allspace} and a diagonal compactness
argument give a subsequence of these same times for which
$h(\cdot,t_j)\to h_\infty$ in every $C^k$ norm. This gives
$  V(\cdot,t_j)\to  V_\infty$ uniformly. Passing to the limit in
\eqref{r7:dissipation}, with $A$ uniformly positive, yields
\[
 h_\infty\det(\nabla ^2h_\infty+h_\infty I)
       \widetilde V_{q-1}(\Omega_{h_\infty},X_\infty)=f.
\]
The limiting body is smooth, origin-symmetric and uniformly convex.
Hence the proof of Theorem \ref{thm4} is complete.

\end{proof}

\section*{Acknowledgment} The authors would like to thank the referee for detailed reading and comments that are both  thorough and insightful.

{\bf Conflict of interest:} The authors declare that they have no conflict of interest.

{\bf Data availability:} Data sharing was not applicable to this article and no datasets were generated or analyzed during the current study.

\end{document}